%% file: root.tex
\RequirePackage{fix-cm}
\documentclass[numbook,envcountsame]{svjour3}

\smartqed  
\usepackage{graphicx}
\usepackage{amsmath,enumerate}
\usepackage{amssymb}
\usepackage{graphicx}
\usepackage{moreverb}
\usepackage[T1]{fontenc}
\usepackage{graphicx}
\usepackage{psfrag}
\usepackage{color}
\newtheorem{assumption}[theorem]{Assumption}
\newcommand{\BbbR}{{\mathbb R}}
\newcommand{\BbbN}{{\mathbb N}}

\newcommand{\eex}{\hbox{}\hfill\rule{0.8ex}{0.8ex}}
\newcommand{\eremk}{\eex}
\newcommand{\supp}{\operatorname*{supp}}
\newcommand{\cyl}{{\mathcal C}}
\newcommand{\A}{\mathbf A}
\newcommand\cutoff{\chi_{S_h}}
\newcommand\edges{{\mathcal E}}
\newcommand\singular{{M_s}}
\newcommand{\appB}{B}
\newcommand{\appC}{C}

 \journalname{}
\newif\iftechreport
\techreporttrue
\iftechreport
\else
\fi

\begin{document}
\iftechreport
\title{On optimal $L^2$- and surface flux convergence in FEM (extended version)}
\else 
\title{On optimal $L^2$- and surface flux convergence in FEM}
\fi 


\author{T. Horger \and J.M. Melenk\and
        B. Wohlmuth
}


\institute{J.M. Melenk \at
              Institut f\"ur Analysis und Scientific Computing, 
              Technische Universit\"at Wien\\
              \email{melenk@tuwien.ac.at}           
           \and T. Horger, 
           B. Wohlmuth \at
              Technische Universit\"at M\"unchen \\
              \email{horger, wohlmuth@ma.tum.de}           
}

\date{Received: date / Accepted: date}

\maketitle

\begin{abstract}
We show that optimal $L^2$-convergence in the finite element method
on quasi-uniform meshes can be achieved if, for some $s_0 > 1/2$, the 
boundary value problem has  
the mapping property $H^{-1+s} \rightarrow H^{1+s}$ for $s \in [0,s_0]$.  
The lack of full elliptic regularity in the dual problem 
has to be compensated by additional regularity of the exact solution. Furthermore, we 
analyze for a Dirichlet problem the approximation of
the normal derivative on the boundary without convexity assumption on the domain. We show that 
(up to logarithmic factors) the optimal rate is obtained. 

\keywords{$L^2$ {\sl a priori} bounds \and duality argument \and reentrant corners}
\end{abstract}

\section{Introduction}
\label{sec:intro}
The finite element method (FEM) is a widely used numerical technique for approximating
solutions of boundary value problems. It is based on approximating the solution 
by piecewise polynomials of degree $k$. 
In the classical case of second order elliptic equations with an
$H^1$-coercive bilinear form, the method is of optimal convergence order in the
$H^1$-norm. 
An important tool for the convergence analysis in other norms such as
the $L^2$-norm are duality arguments (``Nitsche trick''). 
The textbook procedure for optimal order convergence in $L^2$ is to exploit 
full elliptic regularity for the dual problem. Conversely, this procedure
suggests a loss of the optimal convergence rate in $L^2$ if 
$H^2$-regularity fails to hold. This occurs, 
for example, in polygonal domains with reentrant corners. 
 
Nevertheless, 
it is possible to recover the optimal convergence rate in $L^2$, if the 
exact solution has additional regularity to compensate for the lack of 
full regularity of the dual problem. 
More precisely: In this 
note, we consider a setting where an elliptic shift theorem holds for 
both the dual and the bidual problem in the 
range $[-1,-1+s_0]$ for some $s_0 \in (1/2,1]$ 
(see Assumption~\ref{assumption:shift-theorem})
and show that if the solution is in
the Sobolev space $H^{k+1+(1-s_0)}$, then the extra regularity $1-s_0$ 
can be exploited to recover the optimal convergence rate in $L^2$
(up to a logarithmic factor in the lowest order case $k=1$).

In the second part of this note, we consider the convergence in $L^2$
of the normal derivative on the boundary. We show that the optimal
rate $O(h^{k})$ (up to a logarithmic factor in the lowest order case)
can be achieved, if the solution is sufficiently smooth. 
The proof 
is based on a local error analysis of the FEM as discussed, e.g.,
in \cite{wahlbin91,wahlbin95}. Here, we extract error bounds for the 
flux on the boundary from an optimal FEM estimate on a strip of width $O(h)$ 
near the boundary. Although we present the convergence of the flux 
for an $H^1$-conforming discretization, the techniques are applicable 
to mixed methods, \cite{melenk-rezaijafari-wohlmuth14}, FEM-BEM coupling,
\cite{melenk-praetorius-wohlmuth14}, and mortar 
and DG methods, \cite{melenk-wohlmuth12,waluga-wohlmuth14}. In fact, the results of the present
work lead to a sharpening of  \cite{melenk-wohlmuth12}, 
where convexity of the domain was assumed to avoid the analysis of a suitable
additional dual problem.
The techniques employed here are in part similar to those 
developed in \cite{melenk-wohlmuth12}. Nevertheless, they are also 
significantly 
different since we have opted to forego the direct use of anisotropic norms
and instead rely on weighted Sobolev norms and the embedding result 
of Lemma~\ref{lemma:weighted-embedding}. 

The analysis of the optimal convergence of fluxes has attracted some attention
recently. Besides our own contributions 
\cite{melenk-praetorius-wohlmuth14,melenk-rezaijafari-wohlmuth14,melenk-wohlmuth12},
we mention the works \cite{apel-pfefferer-roesch12,apel-pfefferer-roesch14,larson-massing14}
where similar results have been obtained by different methods. 

\subsection{Notation} 
For bounded Lipschitz domains $\Omega \subset \BbbR^d$ with boundary $\Gamma:= \partial\Omega$,
we employ standard notation for 
Sobolev spaces $H^s(\Omega)$, \cite{adams75a,tartar07}. 
We will formulate certain regularity results in terms of Besov space: 
for $s > 0$, $s \not\in\BbbN$, and $q \in [1,\infty]$ the Besov space
$B^s_{2,q}(\Omega)$ is defined by interpolation (the ``real'' method, also known
as $K$-method as described, e.g., in  \cite{tartar07,triebel95}) as 
$$
B^s_{2,q}(\Omega) = (H^{\lfloor s \rfloor}(\Omega),H^{\lceil s\rceil}(\Omega))_{\theta,q}, 
\qquad \theta = s - \lfloor s \rfloor.  
$$ 
Integer order Besov spaces $B^n_{2,q}(\Omega)$ with $n \in \BbbN$ are also defined by interpolation: 
$$
B^n_{2,q}(\Omega) = (H^{n-1}(\Omega),H^{n+1}(\Omega))_{1/2,q}, 
\qquad n \in \BbbN. 
$$ 
To give some indication of the relevance of the second parameter $q$ in the definition
of the Besov spaces, we recall the following (continuous) embeddings:  
$$
H^{s+\varepsilon}(\Omega) \subset B^s_{2,1}(\Omega) \subset H^s(\Omega) \subset B^{s}_{2,\infty}(\Omega)
\subset H^{s-\varepsilon}(\Omega) \qquad \forall \varepsilon > 0. 
$$
Of importance will be the distance function $\delta_\Gamma$ and the regularized 
distance function $\widetilde \delta_\Gamma$ given by 
\begin{equation}
\label{eq:delta_Gamma}
\delta_\Gamma(x) := \operatorname*{dist}(x,\Gamma), 
\qquad 
\widetilde \delta_\Gamma(x) := h + \operatorname*{dist}(x,\Gamma). 
\end{equation}
Later on, the parameter $h > 0$ will be the mesh size of the quasi-uniform triangulation. Also of 
importance will be neighborhoods $S_D$ of the boundary $\partial\Omega$ given by 
\begin{equation}
\label{eq:strip}
S_{D}:= \{x \in \Omega\,|\, \delta_\Gamma(x) < D\}, 
\end{equation}
with particular emphasis on the case $D = O(h)$. 
\subsection{Model problem} 
\label{sec:model-problem}
We let $\Omega \subset \BbbR^d$, $d \in \{2,3\}$, be a bounded Lipschitz domain 
with a polygonal/polyhedral boundary and let (\ref{eq:model-problem}) be our model problem:
\begin{equation}
\label{eq:model-problem}
-\nabla \cdot (\A(x) \nabla u) = f \quad \mbox{ in $\Omega$}, 
\qquad u = 0 \quad \mbox{ on $\partial\Omega$}. 
\end{equation}
We assume that $\A$ and $f$ are sufficiently smooth. Moreover $\A$ is pointwise symmetric positive definite,
and $\A(x) \ge \alpha_0 \operatorname*{I}$ for some $\alpha_0 > 0$ and all $x \in \Omega$. 
As usual, (\ref{eq:model-problem}) is understood in a weak sense, i.e., for a 
right-hand side $f \in \left(H_0^{1}(\Omega)\right)^\prime$ the boundary value problem 
(\ref{eq:model-problem}) reads: 
Find $u \in H^1_0(\Omega)$ such that 
\begin{equation}
\label{eq:model-problem-weak}
a(u,v) := \int_\Omega \A \nabla u \cdot \nabla v = \langle f, v\rangle \qquad 
\forall v \in H^1_0(\Omega). 
\end{equation}
We denote by $T:(H^1_0(\Omega)) ^\prime \rightarrow H^1_0(\Omega)$ the solution operator.
We emphasize that the choice of boundary conditions (here: homogeneous Dirichlet boundary conditions) 
is not essential for our purposes. Essential, however, is the following assumption: 
\begin{assumption}
\label{assumption:shift-theorem}
There exists $s_0 \in (1/2,1]$ such that the solution operator 
$f \mapsto T f$ for (\ref{eq:model-problem-weak}) satisfies
$$
\|T f\|_{H^{1+s_0}(\Omega)} 
\leq C \|f\|_{(H_0^{1-s_0}(\Omega))^\prime} \leq C \| f \|_{L^2(\Omega)}. 
$$
\end{assumption}
Here and in the following $0 < c$, $C < \infty $ denote generic constants 
that do not depend on the mesh-size but possibly depend on $s_0$. We also use $
\lesssim $ to abbreviate $\leq C$.
\begin{remark}
{\rm 
The present problem is symmetric. As a consequence  certain dual problems that will be 
needed below coincide with the primal problem. This will simplify the presentation 
but is not essential. Inspection of the procedure below shows that we need 
Assumption~\ref{assumption:shift-theorem} 
for the dual problem and the bidual problem with weighted right-hand side.
}\eremk
\end{remark}
Let ${\mathcal T}$ be an affine simplicial quasi-uniform triangulation
of $\Omega$ with mesh size $h$ and $V_h:= S^{k,1}_0({\mathcal T}) \subset H^1_0(\Omega)$ the continuous space 
of piecewise polynomials of degree $k$. This space has the following well-known 
properties: 
\begin{enumerate}[(i)]
\item 
Existence of a quasi-local approximation operator: 
The Scott-Zhang operator $I^k_h: H^1(\Omega) \rightarrow S^{k,1}({\mathcal T})$ 
of \cite{scott-zhang90} satisfies: 
\begin{itemize}
\item If $u \in H^1_0(\Omega)$ then $I^k_h u \in V_h = S^{k,1}_0({\mathcal T})$. 
\item $I^k_h$ is quasi-local and stable: 
$\|\nabla I^k_h u\|_{L^2(K)} \lesssim \|\nabla u\|_{L^2(\omega_K)}$, where $\omega_K$ is the 
patch of elements sharing a node with $K$. 
\item $I^k_h$ has approximation properties: 
\begin{equation}
\label{eq:scott-zhang-approximation}
\|\nabla^j (u - I^k_h u)\|_{L^2(K)} \lesssim h^{l+1-j} \|\nabla^{l+1} u\|_{L^2(\omega_K)}, 
\quad j \in \{0,1\}, \quad 0 \leq l \leq k. 
\end{equation}
\end{itemize}
\item For every $v \in B^{3/2}_{2,\infty}(\Omega) \cap H^1_0(\Omega)$ there holds 
$$
\inf_{z \in V_h} \|v - z\|_{H^1(\Omega)} \leq h^{1/2} \|v\|_{B^{3/2}_{2,\infty}(\Omega)}. 
$$
(This follows from property (i) and an interpolation argument using the $K$-method).
\item  
The space $V_h$ satisfies standard elementwise inverse estimates: 
for integer $0 \leq j \leq m \leq k$ 
\begin{equation}
\label{eq:inverse-estimate}
|v|_{H^m(K)} \leq C h^{-(m-j)} |v|_{H^j(K)}  \qquad \forall v \in V_h.
\end{equation}
\end{enumerate}
The Galerkin method for (\ref{eq:model-problem-weak}) is then: Find $u_h \in V_h$ such that 
\begin{equation}
\label{eq:model-problem-FEM}
a(u_h,v) = \langle f,v\rangle \qquad \forall v \in V_h.
\end{equation}
\begin{remark}
{\rm 

The restriction to simplicial triangulations is not essential. Our primary motivation 
for this restriction is that in this 
case the space $V_h$ is known to have the above approximation properties, 
the inverse estimates, and moreover it has 
the ``superapproximation property'' that underlies 
the local error analysis as presented in \cite[Sec.~{5.4}]{wahlbin95}. 
}\eremk
\end{remark}
\section{Regularity} 
\label{sec:regularity}
\subsection{Preliminaries}
\label{sec:regularity-preliminaries}
A key mechanism in our arguments that will allow us to exploit additional regularity 
of a function is the following embedding theorem. 
\begin{lemma}
\label{lemma:weighted-embedding}
The following estimates hold, if $\Omega \subset \BbbR^d$ is a bounded Lipschitz domain and $z$ sufficiently regular.
\begin{eqnarray}
\label{eq:lemma:weighted-embedding-1}
\|\widetilde \delta_\Gamma^{-1/2+\varepsilon} z\|_{L^2(\Omega)} &\leq& 
\|\delta_\Gamma^{-1/2+\varepsilon} z\|_{L^2(\Omega)} \leq 
C_\varepsilon \|z\|_{H^{1/2-\varepsilon}(\Omega)}, \quad \varepsilon \in (0,1/2],\\ 
\label{eq:lemma:weighted-embedding-2}
\|\widetilde \delta_\Gamma^{-1/2} z\|_{L^2(\Omega)} &\leq& 
C |\ln h|^{1/2} \|z\|_{B^{1/2}_{2,1}(\Omega)}, \\
\label{eq:lemma:weighted-embedding-4}
\|\widetilde \delta_\Gamma^{-1/2-\varepsilon} z\|_{L^2(\Omega)} &\leq& 
C_\varepsilon  h^{-\varepsilon}\|z\|_{B^{1/2}_{2,1}(\Omega)}, \qquad \varepsilon > 0,  \\
\label{eq:lemma:weighted-embedding-5}
\|z\|_{L^2(S_h)} &\leq& 
C h^{1/2}\|z\|_{B^{1/2}_{2,1}(\Omega)}, \qquad h > 0,  \\
\label{eq:lemma:weighted-embedding-6}
\|z\|_{L^2(\Gamma)} &\leq& 
C \|z\|_{B^{1/2}_{2,1}(\Omega)}. 
\end{eqnarray}
\end{lemma}
\begin{proof}
The estimate involving $\delta_\Gamma$ in (\ref{eq:lemma:weighted-embedding-1}) 
can be found, e.g., in \cite[Thm.~{1.4.4.3}]{grisvard85a} and (\ref{eq:lemma:weighted-embedding-5}) is shown in \cite[Lemma~{2.1}]{li-melenk-wohlmuth-zou10}. The estimates 
(\ref{eq:lemma:weighted-embedding-2}), (\ref{eq:lemma:weighted-embedding-4}), (\ref{eq:lemma:weighted-embedding-6}) 
follow from 
1D Sobolev embedding theorems for $L^\infty$ and locally flattening the boundary $\Gamma$ in the same
way as it is done in the proof of \cite[Lemma~{2.1}]{li-melenk-wohlmuth-zou10}. For example, 
for (\ref{eq:lemma:weighted-embedding-6}) 
we note that a local flattening of the boundary $\Gamma$
and the 1D embedding $\|v\|^2_{L^\infty(0,1)} \lesssim \|v\|_{L^2(0,1)} \|v\|_{H^1(0,1)}$ imply
$\|z\|_{L^2(\Gamma)}^2 \lesssim \|z\|_{L^2(\Omega)} \|z\|_{H^1(\Omega)}$. This implies
the estimate 
$\|z\|_{L^2(\Gamma)} \lesssim \|z\|_{B^{1/2}_{2,1}(\Omega)}$ by 
\cite[Lemma~{25.3}]{tartar07}. 
We recall that for half-spaces, the upper bound  (\ref{eq:lemma:weighted-embedding-6}) can be directly found in \cite[Thm.~{2.9.3}]{triebel95}, see also the comment in \cite[Sec. 32, Eq. (32.8)]{tartar07}. 
\qed 
\end{proof}
One of several applications of Lemma~\ref{lemma:weighted-embedding} is that it allows us 
to transform negative norms into weighted $L^2$-estimates: 
\begin{lemma}
\label{lemma:negative-norm-vs-weighted-L2}
For $\varepsilon \in (0,1/2]$ and sufficiently regular $z$ there holds 
\begin{eqnarray}
\label{eq:lemma:negative-norm-vs-weighted-L2-1}
\|\delta_{\Gamma}^{\beta} z\|_{(H^{1/2-\varepsilon}(\Omega))^\prime} 
&\leq & C_\varepsilon \|\delta_\Gamma^{\beta +1/2-\varepsilon} z\|_{L^2(\Omega)}, \quad
-1 < \beta  \leq 0, \\
\label{eq:lemma:negative-norm-vs-weighted-L2-3}
\|\widetilde\delta_\Gamma^{-1} z\|_{(B^{1/2}_{2,1}(\Omega))^\prime} 
&\leq & C |\ln h|^{1/2} \|\widetilde \delta_\Gamma^{-1/2} z\|_{L^2(\Omega)}. 
\end{eqnarray}
\end{lemma}
\begin{proof}
Firstly, we show 
(\ref{eq:lemma:negative-norm-vs-weighted-L2-1}): 
\begin{eqnarray*}
\|\delta_{\Gamma}^{\beta} z\|_{(H^{1/2-\varepsilon}(\Omega))^\prime} 
&=& 
\sup_{v \in H^{1/2-\varepsilon}(\Omega)} \frac{\langle \delta_\Gamma^{\beta} z,v\rangle }{\|v\|_{H^{1/2-\varepsilon}(\Omega)}}  \\
&=& \sup_{v \in H^{1/2-\varepsilon}(\Omega)} 
\frac{\langle \delta_\Gamma^{\beta+1/2-\varepsilon} z,\delta_\Gamma^{-1/2+\varepsilon} v\rangle }{\|v\|_{H^{1/2-\varepsilon}(\Omega)}}
\leq C_\varepsilon \|\delta_\Gamma^{\beta +1/2-\varepsilon} z\|_{L^2(\Omega)}, 
\end{eqnarray*}
where, in the last step, we employed \eqref{eq:lemma:weighted-embedding-1}
of Lemma~\ref{lemma:weighted-embedding}. Secondly, (\ref{eq:lemma:negative-norm-vs-weighted-L2-3}) 
follows by the same type of arguments, where the application of \eqref{eq:lemma:weighted-embedding-1} is 
replaced with that of \eqref{eq:lemma:weighted-embedding-2}.
\qed
\end{proof}
\subsection{Regularity}
We recall the following variant of interior regularity of elliptic problems: 
\begin{lemma}
\label{lemma:weighted-shift-theorem}
Let $\Omega$ be a bounded Lipschitz domain and $z \in H^{1+\beta}(\Omega)$, $\beta \in (0,1]$, 
solve 
$$
-\nabla \cdot (\A \nabla z) = f \quad \mbox{ in $\Omega$.}
$$
Then, for a constant $C>0$ depending only on $\|\A\|_{C^{0,1}(\overline{\Omega})}$, $\alpha_0$, $\beta$, 
and $\Omega$
$$
\|\delta_\Gamma^{1-\beta} \nabla^2 z\|_{L^2(\Omega)} 
\leq C \left( \|\delta_\Gamma^{1-\beta} f\|_{L^2(\Omega)} + \|z\|_{H^{1+\beta}(\Omega)}
\right). 
$$
\end{lemma}
\begin{proof} 
The upper bound follows from local interior regularity  for elliptic problems 
(see \cite[Lemma~{5.7.2}]{morrey66} or \cite[Thm.~{8.8}]{gilbarg-trudinger77a})
and a  Besicovitch covering argument, see, e.g.,~\cite[Section 1.5.2]{evans98} and~\cite[Chapter 5]{melenk02}.
We refer also to \cite[Lemma~{A.3}]{khoromskij-melenk03} where a closely related result 
is worked out in detail. 
\qed
\end{proof}
\subsubsection{Refined regularity for polygons and polyhedra}
\label{sec:regularity-polygon}
It is worth pointing out that neither the structure of the boundary $\Gamma$ nor
the kind of boundary conditions play a role in Lemma~\ref{lemma:weighted-shift-theorem}. 
One possible interpretation of Lemma~\ref{lemma:weighted-shift-theorem} is that $z$ could 
lose the $H^2$-regularity anywhere near $\Gamma$. For certain boundary conditions
such as homogeneous Dirichlet conditions and piecewise smooth geometries $\Gamma$ the solution
fails to be in $H^2$ only near the points of nonsmoothness of the geometry. With methods similar
to those of Lemma~\ref{lemma:weighted-shift-theorem} one can show the following, stronger result: 
\begin{lemma}
\label{lemma:weighted-shift-theorem-polygon-simple}
Let $\Omega$ be a (curvilinear) polygon in 2D or a (curvilinear) polyhedron in $3D$. 
Denote by $\edges$ the set of all vertices of $\Omega$ in 2D and the set of all edges of $\Omega$ in 3D. 
Let $\delta_\edges$ be the distance from $\edges$. Let $z \in H^{1+\beta}(\Omega)$, $\beta \in (0,1]$, 
solve (\ref{eq:model-problem}).  Then, for  a constant $C$ depending only
on $\alpha_0$, $\|\A\|_{C^{0,1}(\overline{\Omega})}$, $\beta$, and $\Omega$, 
$$
\|\delta_\edges^{1-\beta} \nabla^2 z\|_{L^2(\Omega)} 
\leq C_\beta \left(
\|\delta_\edges^{1-\beta} f\|_{L^2(\Omega)} + \|z\|_{H^{1+\beta}(\Omega)}
\right). 
$$
\end{lemma}
\begin{proof}
Follows from local considerations as in Lemma~\ref{lemma:weighted-shift-theorem}. The novel aspect 
is the behavior near the boundary away from the vertices (in 2D) and the edges (in 3D). This is achieved
with an adapted covering theorem of the type described in Theorems~\ref{thm:covering-2D}, \ref{thm:covering-3D}. 
The key feature of these coverings is that they allow us to reduce the considerations to 
balls $B = B_r(x)$ and stretched balls $\widehat B = B_{(1+\varepsilon)r}(x)$ (with fixed $\varepsilon >0$) 
with $r \sim \operatorname*{dist}(x,\edges)$ and the following properties: 
either $x \in \Omega$ with $\widehat B_r(x) \subset \Omega$ 
or $x \in \Gamma$ and $\widehat B \cap \Omega$ is a half-ball. Local elliptic regularity assertions
can then be employed for each ball $B$. 
\qed 
\end{proof}
Lemma~\ref{lemma:weighted-shift-theorem-polygon-simple} assumes that a loss of $H^2$-regularity 
occurs at any point of non-smoothness of $\Gamma$. However, the set of ``singular'' vertices or edges 
can be further reduced. For example, in 2D for 
$\A = \operatorname*{Id}$, it is well-known that only the vertices of $\Omega$ with interior angle
greater than $\pi$ lead to a loss of full $H^2$-regularity. It will therefore be useful to introduce 
the closed set $\singular$ of boundary points associated with a loss of $H^2$-regularity. Before introducing 
this set, we point out that this set is a subset of the vertices and edges: 
\begin{definition}[$H^2$-regular part and singular part of the boundary]
\label{def:H2-regular-edges}
{\it
Let $\Omega$ be a polygon (in 2D) or a polyhedron (in 3D) 
with vertices ${\mathcal A}$ and edges ${\mathcal E}$.  
\begin{enumerate}
\item 
A vertex $A \in {\mathcal A}$ of $\Omega$ is said to 
be $H^2$-regular, if there is a ball $B_\varepsilon(A)$ of radius $\varepsilon > 0$ such that the solution $u$ of (\ref{eq:model-problem})
satisfies $u|_{B_\varepsilon(A)\cap\Omega} \in H^2(\Omega)$ whenever $f \in L^2(\Omega)$ together with the {\sl a priori} estimate
$\|u\|_{H^2(B_\varepsilon(A)\cap\Omega)} \leq C \|f\|_{L^2(\Omega)}$. 
\item 
In 3D, an edge $e \in {\mathcal E}$ of $\Omega$ with endpoints $A_1$, $A_2$ is said to be $H^2$-regular
if the following condition is satisfied: There is $c > 0$ such that for the neighborhood
$S= \cup_{x \in e} B_{c \operatorname*{dist}(x,\{A_1,A_2\})}(x)$ of the edge $e$ we have 
the regularity assertion $u|_{S\cap\Omega} \in H^2$ for the solution $u$ of (\ref{eq:model-problem}) 
whenever $f \in L^2(\Omega)$ together with the {\sl a priori} estimate 
$\|u\|_{H^2(S\cap\Omega)} \leq C \|f\|_{L^2(\Omega)}$. 
\end{enumerate} 
Denote by ${\mathcal A}_r \subset {\mathcal A}$ the set of $H^2$-regular vertices and by 
${\mathcal E}_r \subset {\mathcal E}$ the set of 
$H^2$-regular edges.  
Correspondingly, let ${\mathcal A}_s:= {\mathcal A} \setminus {\mathcal A}_r$
and ${\mathcal E}_s:= {\mathcal E}\setminus {\mathcal E}_r$ be the set of vertices and edges, respectively, 
associated with a loss of $H^2$-regularity. Define the {\em singular set} $\singular$ as 
\begin{equation}
\label{eq:singular-set}
\singular := {{\mathcal A}_s \bigcup {\mathcal E}_s}\subset \Gamma. 
\end{equation}
}
\end{definition}
With the notion of the singular set in hand, we can formulate the following regularity result: 
\begin{lemma}
\label{lemma:weighted-shift-theorem-polygon}
Let $\Omega$ be a polygon or a polyhedron. Let $\singular$ be the singular set as defined 
in Definition~\ref{def:H2-regular-edges}. 
Then the following is true for any solution $z \in H^1_0(\Omega)$ of (\ref{eq:model-problem}): 
If $z \in H^{1+\beta}(\Omega)$ for some $\beta \in (0,1]$, then 
with $\delta_\singular:= \operatorname*{dist}(\cdot,\singular)$, there holds 
for some $C > 0$ depending only on $\alpha_0$, $\|\A\|_{C^{0,1}(\overline{\Omega})}$, $\beta$, and $\Omega$,
$$
\|\delta_\singular^{1-\beta} \nabla^2 z\|_{L^2(\Omega)} 
\leq C_\beta \left(
\|\delta_\singular^{1-\beta} f\|_{L^2(\Omega)} + \|z\|_{H^{1+\beta}(\Omega)}
\right). 
$$
\end{lemma}
\begin{proof}
The proof is based on local considerations as in Lemma~\ref{lemma:weighted-shift-theorem-polygon-simple}. 
We recall that not all vertices and edges (in 3D) are included in the singular set $\singular$. This is accounted for 
by a further refinement of the covering employed. We restrict ourselves to
 the 3D situation. Using finite coverings provided 
by Theorem~\ref{thm:covering-3D}, one may restrict the attention to balls $B_r = B_r(x)$ 
and stretched balls $\widehat B = B_{(1+\varepsilon)r}(x)$ (with fixed $\varepsilon >0$) 
with $r \sim \operatorname*{dist}(x,\singular)$ where one of the following additional properties
is satisfied: a) $x \in \Omega$ with $\widehat B \subset \Omega$; 
b)  $x \in {\mathcal A}_r$ and $\widehat B \cap \Omega$ is a solid angle; 
c) $x \in \cup {\mathcal E}_r$ and $\widehat B \cap \Omega$ is a dihedral angle; 
d)  $x$ lies in the interior of a face and $\widehat B \cap \Omega$ is a half-ball. 
We emphasize that we do not need to consider balls $B_r(x)$ with $x \in {\mathcal A}_s$ or
$x \in {\mathcal E}_s$ since the covering provided by Theorem~\ref{thm:covering-3D} is such that
for every such $x$ there is a neighborhood ${\mathcal U}_x$ of $x$ that is covered by (countably many)
balls whose radii tend to $0$ as their centers approach $x$.

\qed 
\end{proof}
\subsubsection{Shift theorems for locally supported right-hand sides}
We have the following continuity results for the solution operator $T$ for 
our model problem (\ref{eq:model-problem}): 
\begin{lemma}
\label{lemma:B32-regularity}
Let Assumption~\ref{assumption:shift-theorem} be valid. Then
 $T: (H^1_0(\Omega))^\prime \rightarrow H^1_0(\Omega)$ satisfies
\begin{eqnarray}
\label{eq:lemma:B32-regularity-1}
\|T f\|_{B^{3/2}_{2,\infty}(\Omega)}   
&\leq & C \|f\|_{(B^{1/2}_{2,1}(\Omega))^\prime}, \\
\label{eq:lemma:B32-regularity-1a}
\|T f\|_{H^{3/2+\varepsilon}(\Omega)}  
&\leq & C_\varepsilon  \|\delta_\Gamma^{1/2-\varepsilon} f\|_{L^2(\Omega)}, 
\qquad 0 < \varepsilon \leq s_0-1/2. 
\end{eqnarray}
In particular, if $f \in L^2(\Omega)$ with $\supp f \subset \overline{S_{h}}$,
then
\begin{eqnarray}
\label{eq:lemma:B32-regularity-2}
\|T f\|_{B^{3/2}_{2,\infty}(\Omega)}   
& \leq & C h^{1/2} \|f\|_{L^2(\Omega)},  \\
\label{eq:lemma:B32-regularity-3}
\|T f\|_{H^{3/2+\varepsilon}(\Omega)}  
&\leq & C_\varepsilon h^{1/2-\varepsilon} \|f\|_{L^2(\Omega)}, 
\qquad 0 < \varepsilon \leq s_0-1/2. 
\end{eqnarray}
\end{lemma}
\begin{proof}
We follow the arguments of \cite[Lemma~{5.2}]{melenk-wohlmuth12}.
The starting point for the proof of (\ref{eq:lemma:B32-regularity-1}) is that interpolation and
Assumption~\ref{assumption:shift-theorem} yield
with $\theta \in (0,1)$
$$
T: ((H^{1-s_0}_0(\Omega))^\prime,(H^1_0(\Omega))^\prime)_{\theta,\infty} 
\rightarrow 
(H^{1+s_0}(\Omega),H^1(\Omega))_{\theta,\infty}  = 
B^{1+s_0(1-\theta)}_{2,\infty}(\Omega).
$$
Next, we recognize as in \cite[Lemma~{5.2}]{melenk-wohlmuth12}
(cf.~\cite[Thm.~{1.11.2}]{triebel95} or \cite[Lemma~{41.3}]{tartar07})
\begin{align*}
((H^{1-s_0}_0(\Omega))^\prime,(H^1_0(\Omega))^\prime)_{\theta,\infty}  &= 
((H^{1-s_0}_0(\Omega),H^1_0(\Omega))_{\theta,1})^\prime \\
&\supset 
((H^{1-s_0}(\Omega),H^1(\Omega))_{\theta,1})^\prime = 
(B^{1-s_0(1-\theta)}_{2,1}(\Omega))^\prime. 
\end{align*}
Setting $\theta = 1-1/(2s_0) \in (0,1/2]$, we get
$(B^{1-s_0(1-\theta)}_{2,1}(\Omega))^\prime = (B^{1/2}_{2,1}(\Omega))^\prime$
and $B^{1+s_0(1-\theta)}_{2,\infty}(\Omega) =B^{3/2}_{2,\infty}(\Omega)$.
The assertion (\ref{eq:lemma:B32-regularity-1a}) follows from 
 the Assumption~\ref{assumption:shift-theorem} 
and (\ref{eq:lemma:negative-norm-vs-weighted-L2-1}) with $\beta=0$. 
For the bound (\ref{eq:lemma:B32-regularity-2}), we argue 
as in the proof of Lemma~\ref{lemma:negative-norm-vs-weighted-L2} 
and use   \eqref{eq:lemma:weighted-embedding-5}, see also
\cite[Lemma~{5.2}]{melenk-wohlmuth12}. 
Finally, the proof of (\ref{eq:lemma:B32-regularity-3}) follows from (\ref{eq:lemma:B32-regularity-1a})
and the assumed support properties of $f$. 
\qed
\end{proof}
We will also require mapping properties of the solution operator $T$ in weighted spaces: 
\begin{lemma}
\label{lemma:regularity-weighted-rhs}
Let Assumption~\ref{assumption:shift-theorem} be valid.
Then for $v \in L^2(\Omega)$
\begin{eqnarray}
\label{eq:lemma:regularity-weighted-rhs-1}
\|T (\widetilde \delta_\Gamma^{-1} v)\|_{B^{3/2}_{2,\infty}(\Omega) } 
&\leq & C |\ln h|^{1/2} \|\widetilde\delta_\Gamma^{-1/2} v\|_{L^2(\Omega)}, \\ 
\label{eq:lemma:regularity-weighted-rhs-1a}
\|T (\widetilde \delta_\Gamma^{-1} v)\|_{H^{3/2+\varepsilon}(\Omega) } 
&\leq &C_\varepsilon h^{-\varepsilon} \|\widetilde\delta_\Gamma^{-1/2} v\|_{L^2(\Omega)},
\quad \varepsilon \in (0,s_0-1/2], \\
\label{eq:lemma:regularity-weighted-rhs-2}
\|T(\delta_\Gamma^{-1+2\varepsilon} v)\|_{H^{3/2+\varepsilon}(\Omega) } 
& \leq & C_\varepsilon \|\delta_\Gamma^{-1/2+\varepsilon}
v\|_{L^2(\Omega)},
\quad \varepsilon \in (0,s_0-1/2].  
\end{eqnarray}
\end{lemma}
\begin{proof}
The results follow by combining 
Lemmas~\ref{lemma:negative-norm-vs-weighted-L2} and \ref{lemma:B32-regularity}. 
\qed
\end{proof}
For the analysis of the FEM error on the neighborhood $S_h$, we need a 
refined version of interior regularity for elliptic problems. 
The following result is very similar to \cite[Lemma~{5.4}]{melenk-wohlmuth12}
and closely related to Lemma~\ref{lemma:weighted-shift-theorem}:
\begin{lemma}
\label{lemma:5.4}
Let $z$ solve the equation 
$$
-\nabla \cdot (\A \nabla z)  = v \quad \mbox{ in $\Omega$}.
$$
Then there exist $C_A$ (depending only on $\alpha_0$, $\|\A\|_{C^{0,1}(\overline{\Omega})}$ and $\Omega$)  
and $c_1 > 0$ (depending only on $\Omega$) such that 
for $z \in B^{3/2}_{2,\infty}(\Omega)$, we have
\begin{equation}
\label{eq:lemma:5.4-1000}
\|\delta_\Gamma^{1/2} \nabla^2 z\|_{L^2(\Omega\setminus S_{h})} 
\leq C_A \left[ 
\sqrt{|\ln h|} \|z\|_{B^{3/2}_{2,\infty}(\Omega)} +  \|\sqrt{\delta_\Gamma} v\|_{L^2(\Omega\setminus S_{c_1 h})}
\right].
\end{equation}
If the right-hand side $v$ satisfies additionally 
$\supp v \subset \overline{S_h}$ and furthermore $z = Tv$, then 
there are constants $C_A$ (depending only on $\alpha_0$, $\|\A\|_{C^{0,1}(\overline{\Omega})}$, and $\Omega$) 
and $c>1$, $\tilde c>c^\prime > 1$ (depending only on $\Omega$) such that 
for all sufficiently small $h>0$:
\begin{enumerate}[(i)]
\item
\label{item:lemma:5.4-1}
If $z \in B^{3/2}_{2,\infty}(\Omega)$ then
$\|\delta_\Gamma^{1/2} \nabla^2 z\|_{L^2(\Omega\setminus S_{\tilde ch})} 
\leq C_A \sqrt{|\ln h|} \|z\|_{B^{3/2}_{2,\infty}(\Omega)}. 
$
\item
\label{item:lemma:5.4-2}
For every $\alpha > 0$ there holds
$$
\!\!\!\!\! 
\|\delta_\Gamma^{\alpha} \nabla^3 z\|_{L^2(\Omega\setminus S_{\tilde ch})} 
\!\!\leq\!\! C_A \!\!\left[ 
\|\delta_\Gamma^{\alpha-1}\nabla^2 z\|_{L^2(\Omega\setminus S_{c^\prime h})} 
+ \|\A \|_{C^{1,1}(\Omega)} 
\|\delta_\Gamma^{\alpha}\nabla z\|_{L^2(\Omega\setminus S_{c^\prime h})} 
\right]\!.\qquad 
$$
\item
\label{item:lemma:5.4-3}
If $z \in H^{3/2+\varepsilon}(\Omega)$ for some $\varepsilon \in (0,1/2)$, then
for some $C_{A,\varepsilon} >0$ (depending on $\alpha_0$, $\|\A\|_{C^{0,1}(\overline{\Omega})}$, 
$\Omega$,  and $\varepsilon$) there holds 
$$
\|\nabla^2 z\|_{L^2(\Omega\setminus S_{\tilde c h})}
\leq C_{A,\varepsilon} h^{-1/2+\varepsilon} \|z\|_{H^{3/2+\varepsilon}(\Omega)}. 
$$
\item
\label{item:lemma:5.4-4}
If Assumption~\ref{assumption:shift-theorem} is valid, then for some $C > 0$ depending only on 
the solution operator $T$ and $C_{A,\varepsilon}$ of
(\ref{item:lemma:5.4-3}) we have 
$\|\nabla^2 z\|_{L^2(\Omega\setminus S_{\tilde ch})} 
\leq C \|v\|_{L^2(\Omega)}.
$
\end{enumerate}
\end{lemma}
\begin{proof}
{\em of (\ref{eq:lemma:5.4-1000}), 
(\ref{item:lemma:5.4-1}), (\ref{item:lemma:5.4-2}):}
\cite[Lemma~{5.4}]{melenk-wohlmuth12} is formulated for $-\Delta$.
However, the essential property of the differential operator $\Delta$
that is required is just interior regularity. Hence, the result also
stands for the present, more general elliptic operator 
(with the appropriate modifications due to the fact that the coefficient
$\A$ is allowed to be non-constant).
In the interest
of generality, we have also tracked in (\ref{eq:lemma:5.4-1000}) the dependence
on the right-hand side $v$, which was not done in \cite[Lemma~{5.4}]{melenk-wohlmuth12}. 
\iftechreport
A full proof can be found in Appendix~\ref{sec:lemma:5.4}.
\else 
A full proof can be found in \cite[Appendix~\appC]{melenk-wohlmuth14a}.
\fi

{\em Proof of (\ref{item:lemma:5.4-3}):} This follows again by local
considerations similar to those employed in the proof of
\cite[Lemma~{5.4}]{melenk-wohlmuth12} and the obvious bound
$\delta_\Gamma \ge h$ on $\Omega\setminus S_{\tilde ch}$.
\iftechreport
A full proof can be found in Appendix~\ref{sec:lemma:5.4}.
\else 
A full proof can be found in \cite[Appendix~\appC]{melenk-wohlmuth14a}.
\fi

{\em Proof of (\ref{item:lemma:5.4-4}):} In view of (\ref{item:lemma:5.4-3}),
we have to estimate $\|z\|_{H^{3/2+\varepsilon}(\Omega)}$. By the
support properties of $v$, the bound (\ref{eq:lemma:B32-regularity-3})
yields $\|z\|_{H^{3/2+\varepsilon}(\Omega)} \leq C h^{1/2-\varepsilon} \|v\|_{L^2(\Omega)}$.
Inserting this in (\ref{item:lemma:5.4-3}) gives the result.
\qed
\end{proof}

\section{FEM $L^2$-error analysis} 
\label{sec:FEM-L2}
Let $u_h$ be the FEM approximation and denote by $e = u-u_h$ the FEM error. 
The standard workhorse is the Galerkin orthogonality 
\begin{equation}
a(e,v) = a(u - u_h,v) = 0 \qquad \forall v \in V_h.
\end{equation}
We start with a weighted $L^2$-error: 
\begin{lemma}
\label{lemma:weighted-L2-estimate}
Let Assumption~\ref{assumption:shift-theorem} be valid. Assume that a function 
$z \in H^1_0(\Omega)$ satisfies the Galerkin orthogonality 
$$
a(z,v) = 0 \qquad \forall v \in V_h. 
$$
Then 
\begin{eqnarray}
\label{eq:lemma:weighted-L2-estimate-1}
\|\delta_\Gamma^{-1/2+\varepsilon} z\|_{L^2(\Omega)} &\leq & C_\varepsilon h^{1/2+\varepsilon} \|z\|_{H^1(\Omega)}, 
\qquad  
\varepsilon \in (0,s_0-1/2], \\
\label{eq:lemma:weighted-L2-estimate-2}
\|\widetilde \delta_\Gamma^{-1/2} z\|_{L^2(\Omega)} &\leq & C h^{1/2} |\ln h|^{1/2} \|z\|_{H^1(\Omega)}. 
\end{eqnarray}
\end{lemma}
\begin{proof}
The proof follows standard lines. Define $\psi = T (\delta_{\Gamma}^{-1+2\varepsilon} z)$, which solves 
$$
\langle v,\delta_\Gamma^{-1+2\varepsilon} z\rangle = a(v,\psi) \qquad \forall v \in H^1_0(\Omega). 
$$ 
Then we have by Galerkin orthogonality for arbitrary $I\psi \in V_h$
\begin{eqnarray*}
\|\delta_\Gamma^{-1/2+\varepsilon} z\|^2_{L^2(\Omega)} &=& a(z,\psi) = a(z,\psi - I \psi)
\leq C \|z\|_{H^1(\Omega)} \|\psi - I \psi\|_{H^1(\Omega)}. 
\end{eqnarray*}
{}From \eqref{eq:lemma:regularity-weighted-rhs-2} in
 Lemma~\ref{lemma:regularity-weighted-rhs}, we have $
\|\psi\|_{H^{3/2+\varepsilon}(\Omega)} \leq C_\varepsilon \|\delta_\Gamma^{-1/2 + \varepsilon} z\|_{L^2(\Omega)}$ 
so that with the approximation properties of $V_h$ we get 
$$
\inf_{I\psi \in V_h} \|\psi - I \psi\|_{H^1(\Omega)} \leq C_\varepsilon h^{1/2+\varepsilon} \|\delta_\Gamma^{-1/2+\varepsilon} z\|_{L^2(\Omega)}. 
$$
This shows (\ref{eq:lemma:weighted-L2-estimate-1}). For (\ref{eq:lemma:weighted-L2-estimate-2}), 
we proceed similarly using the regularity assertion (\ref{eq:lemma:regularity-weighted-rhs-1})
and the approximation property of $V_h$. 
\qed
\end{proof}
\begin{corollary}
\label{cor:weighted-L2-estimate}
Let Assumption~\ref{assumption:shift-theorem} be valid and
the  solution $u$ be in $H^{s}(\Omega)$, $s \ge 1$. Then the FEM error $e = u - u_h$ satisfies 
for $\varepsilon \in (0,s_0-1/2]$
$$
\|\delta_\Gamma^{-1/2+\varepsilon} e\|_{L^2(\Omega)} \leq C_\varepsilon h^{\mu-1/2+\varepsilon}
\|u\|_{H^{\mu}(\Omega)}, 
\qquad \mu:= \min\{s,k+1\}.
$$
\end{corollary}
The following Theorem~\ref{thm:optimal-L2} shows that the optimal rate of 
the $L^2$-convergence of the FEM can be achieved also for non-convex geometries
if the solution has some additional regularity: 
\begin{theorem}
\label{thm:optimal-L2}
Let Assumption~\ref{assumption:shift-theorem} be valid. 
Let the exact solution $u$ satisfy the extra regularity 
$u \in H^{k+2-s_0}(\Omega)$.  
Then the FEM error $u - u_h$ satisfies 
\begin{equation}
\label{eq:thm;optimal-L2-1}
\|u - u_h\|_{L^2(\Omega)} \lesssim h^{k+1} 
\|u\|_{H^{k+2-s_0}(\Omega)}. 
\end{equation}
More generally, if $u \in H^s(\Omega)$, $s \in[1,k+2-s_0]$, then 
\begin{equation}
\label{eq:thm;optimal-L2-2}
\|u - u_h\|_{L^2(\Omega)} \lesssim 
h^{s-1+s_0} \|u\|_{H^s(\Omega)},  \quad 1 \leq s \leq k+2-s_0.  
\end{equation}
\end{theorem}
\begin{proof}
{\em of (\ref{eq:thm;optimal-L2-1}):}
We proceed along a standard duality argument. To begin with, we note that 
the case $s_0 = 1$ is classical so that we may assume $s_0 < 1$ for the remainder
of the proof. 
Set $\varepsilon :=s_0-1/2 \in (0,1/2)$ by our assumption $1/2 < s_0 < 1$. 
Let $w = T e$ and let $w_h \in V_h$ be its Galerkin approximation. 
Quasi-optimality and the use of \eqref{eq:lemma:negative-norm-vs-weighted-L2-1}
give us the following energy error estimate: 
\begin{eqnarray}
\nonumber 
\|w - w_h\|_{H^1(\Omega)} &\lesssim & \inf_{v \in V_h} \|w - v\|_{H^1(\Omega)} 
\lesssim h^{1/2+\varepsilon} \|w\|_{H^{3/2+\varepsilon}(\Omega)}  \\
\label{eq:thm:optimal-L2-100}
&\lesssim & h^{1/2+\varepsilon} \|e\|_{(H^{1/2-\varepsilon}(\Omega))^\prime}  
\lesssim  h^{1/2+\varepsilon} \|e\|_{L^2(\Omega)}. 
\end{eqnarray}
The Galerkin orthogonalities satisfied by $e$ and $w - w_h$ and a
weighted Cauchy-Schwarz inequality yield for
the Scott-Zhang interpolant $I_h^k u$
\begin{eqnarray}
\label{eq:key-duality}
\|e\|^2_{L^2(\Omega)} &=& a(e, w) = a(e,w - w_h) = a(u - I_h^ku,w - w_h) 
\\
\label{eq:key-CS}
&\leq& C \|\widetilde \delta_\Gamma^{-1/2+\varepsilon} \nabla (u - I u)\|_{L^2(\Omega)}
\|\widetilde\delta_\Gamma^{1/2-\varepsilon} \nabla (w - w_h)\|_{L^2(\Omega)}. 
\end{eqnarray}
We get by a covering argument and \eqref{eq:lemma:negative-norm-vs-weighted-L2-1} of Lemma~\ref{lemma:weighted-embedding}
\begin{eqnarray}
\nonumber 
\|\widetilde\delta_\Gamma^{-1/2+\varepsilon} \nabla (u - I^k_h u)\|_{L^2(\Omega)} 
&\lesssim&  h^{k} \|\widetilde\delta_\Gamma^{-1/2+\varepsilon} \nabla^{k+1} u\|_{L^2(\Omega)}  \\
\label{eq:thm:optimal-L2-110}
&\lesssim&  h^{k} 
\|\nabla^{k+1} u\|_{H^{1/2-\varepsilon}(\Omega)}. 
\end{eqnarray}
It should also be noted at this point that in (\ref{eq:thm:optimal-L2-110}), 
the weight $\widetilde\delta_\Gamma^{-1/2+\varepsilon}$ can be considered as constant in each element $K$.
For the contribution $\|\widetilde\delta_\Gamma^{1/2-\varepsilon} \nabla (w - w_h)\|_{L^2(\Omega)}$ 
in (\ref{eq:key-CS}),
we have to analyze the Galerkin error $w - w_h$ in more detail, which will be done with 
the techniques from the local error analysis of the FEM.
We split
$\Omega$ into $S_{ch} \cup (\Omega \setminus S_{ch})$ where $c > 0$ will be selected
sufficiently large below. 
For fixed $c>0$, the $L^2$-norm on $S_{ch}$ can easily be bounded
with (\ref{eq:thm:optimal-L2-100}) by 
\begin{equation}
\label{eq:thm:optimal-L2-1000}
\|\widetilde \delta_\Gamma^{1/2-\varepsilon} \nabla (w - w_h)\|_{L^2(S_{ch})}
\lesssim h^{1/2-\varepsilon} \|\nabla (w - w_h)\|_{L^2(\Omega)} 
\lesssim h \|e\|_{L^2(\Omega)}. 
\end{equation}
The term $\|\widetilde\delta_\Gamma^{1/2-\varepsilon} \nabla (w - w_h)\|_{L^2(\Omega\setminus S_{ch})}$
requires more care. Obviously,
$\widetilde\delta_\Gamma^{1/2-\varepsilon}\lesssim 
\delta_\Gamma^{1/2-\varepsilon}$ on $\Omega \setminus S_{ch}$. 
We have to employ the tools from the local error analysis 
 in FEM. The Galerkin orthogonality satisfied by $w - w_h$ 
allows us to use  techniques as described in \cite[Sec.~{5.3}]{wahlbin95}, which  yields 
the following estimate for arbitrary balls $B_r \subset B_{r^\prime}$ with the same 
center (implicitly, $r^\prime > r + O(h)$ is assumed in (\ref{eq:wahlbin-1})) 
\begin{equation}
\label{eq:wahlbin-1}
\|\nabla( w - w_h)\|_{L^2(B_r)} \lesssim \|\nabla (w - I^k_h w)\|_{L^2(B_{r^\prime})} + 
\frac{1}{r^\prime - r}\|w - w_h\|_{L^2(B_{r^\prime})}.
\end{equation}  
By a covering argument (which requires $r^\prime - r \sim c \delta_\Gamma(x)$, where $x$ is the center
of the ball $B_r$, and $c$ is sufficiently small) 
these local estimates can be combined into a global estimate of the following form, 
where for sufficiently small $c_1 >0$ ($c_1$ depends 
only on $\Omega$ and the shape regularity of the triangulation but is independent of $h$):
\begin{align}
\label{eq:wahlbin-10}
& \|\delta_\Gamma^{1/2-\varepsilon} \nabla (w - w_h)\|_{L^2(\Omega\setminus S_{c h})}
\lesssim \\
\nonumber 
& \qquad \|\delta_\Gamma^{1/2-\varepsilon} \nabla (w - I^k_h w)\|_{L^2(\Omega\setminus S_{c c_1 h})} 
+ \|\delta_\Gamma^{-1/2-\varepsilon} (w - w_h)\|_{L^2(\Omega\setminus S_{c c_1 h})}.
\end{align}
This estimate implicitly assumed $c_1 c h >  2h$, i.e., at least two layers of elements separate 
$\Gamma$ from $\Omega\setminus S_{c_1 c h}$. We now fix $c > 2/c_1$. 
The first term in (\ref{eq:wahlbin-10}) can easily be bounded by 
standard approximation properties of $I^k_h$, Lemma~\ref{lemma:weighted-shift-theorem}, 
and Assumption \ref{assumption:shift-theorem}: 
\begin{align*}
& \| \delta_\Gamma^{1/2 - \varepsilon} \nabla (w - I^k_h w)\|_{L^2(\Omega\setminus S_{c c_1 h})} 
\lesssim  h \|\delta_\Gamma^{1/2 - \varepsilon} \nabla^2 w\|_{L^2(\Omega)} \\
& \qquad {\lesssim} h \left[ \|\delta_\Gamma^{1/2 - \varepsilon} e\|_{L^2(\Omega)}  + 
\|w\|_{H^{3/2+\varepsilon}(\Omega)} \right] 
\lesssim h \|e\|_{L^2(\Omega)}.
\end{align*}
In the last step, we have to deal with the term
$\| \delta_\Gamma^{-1/2-\varepsilon} (w - w_h)\|_{L^2(\Omega \setminus S_{c c_1 h})}$ 
of (\ref{eq:wahlbin-10}). 
Lemma~\ref{lemma:weighted-L2-estimate} and (\ref{eq:thm:optimal-L2-100}) 
imply
\begin{eqnarray}
\nonumber \hspace*{-1cm}
\| \delta_\Gamma^{-1/2-\varepsilon} (w - w_h)\|_{L^2(\Omega \setminus S_{c c_1 h})}
&\lesssim&  h^{-2\varepsilon} \|\delta_\Gamma^{-1/2+\varepsilon}( w - w_h)\|_{L^2(\Omega)} 
\\
\label{eq:thm:optimal-L2-700}
&\lesssim & h^{-2\varepsilon}h^{1/2+\varepsilon} \|w - w_h\|_{H^1(\Omega)}   
\lesssim  h \|e\|_{L^2(\Omega)}. 
\end{eqnarray}
Here we have used the quasi-optimality of the Galerkin approximation with respect to the $H^1$-norm.

{\em Proof of (\ref{eq:thm;optimal-L2-2}):} The above arguments show that the regularity of 
$u$ enters in the bound (\ref{eq:thm:optimal-L2-110}). For $u \in H^1(\Omega)$, the stability 
properties of the Scott-Zhang operator $I^k_h$ show 
\begin{equation}
\label{eq:thm:optimal-L2-1001}
\|\widetilde \delta_\Gamma^{-1/2+\varepsilon} \nabla (u - I^k_h u)\|_{L^2(\Omega)} 
\lesssim h^{-1/2+\varepsilon} \|u\|_{H^1(\Omega)}. 
\end{equation}
Hence, a standard interpolation argument that combines (\ref{eq:thm:optimal-L2-110}) and 
(\ref{eq:thm:optimal-L2-1001}) yields 
$
\|\widetilde \delta_\Gamma^{-1/2+\varepsilon} \nabla (u - I^k_h u)\|_{L^2(\Omega)} 
\lesssim h^{-1/2+\varepsilon +s-1}\|u\|_{H^s(\Omega)}$ for 
$s \in [1,k+2-s_0]$. Combining this estimate with the above 
control of $\|\widetilde\delta_\Gamma^{1/2-\varepsilon} \nabla (w - w_h)\|_{L^2(\Omega)}$ yields
the desired bound in the range $s \in [1,k+2-s_0]$. 
\qed
\end{proof}
\section{FEM $L^2$-error analysis on piecewise smooth geometries}
\label{sec:L2-error-polygon}
The convergence analysis of Theorem~\ref{thm:optimal-L2} did not make 
{\em explicit} use of the fact that a piecewise smooth geometry is 
considered; the essential ingredient was 
Assumption~\ref{assumption:shift-theorem} (which, of course, is related
to the geometry of the problem). This is reflected in our use of 
 $\widetilde\delta_\Gamma$, which measures the distance
from the boundary $\Gamma$. One interpretation of this procedure is that 
one assumes of the dual solution $w$ (and, in fact, also of the solution 
of the ``bidual'' problem employed to estimate 
$\|\widetilde\delta_\Gamma^{-1/2+\varepsilon} (w - w_h)\|_{L^2(\Omega)}$
in Theorem~\ref{thm:optimal-L2}) that it may lose $H^2$-regularity
{\em anywhere} near $\Gamma$. However, for piecewise smooth geometries in conjunction
with certain homogeneous boundary conditions (here: homogeneous Dirichlet
conditions), this 
loss of $H^2$-regularity is restricted to a much smaller set, namely, 
a subset of vertices in 2D and a subset of the skeleton (i.e., the union of vertices
and edges) in 3D. This set is given by $\singular$ in Definition~\ref{def:H2-regular-edges}. 
For this set $\singular$, we introduce the distance function 
\begin{equation}
\label{eq:delta_M} 
\delta_\singular:= \operatorname*{dist}(\cdot,\singular), 
\qquad \widetilde \delta_\singular:= h + \delta_\singular. 
\end{equation}
\begin{theorem} 
\label{thm:abstract-optimal-L2} 
Let $\Omega$ be a polygon (in 2D) or a polyhedron (in 3D).
Let $\singular$ be the set of vertices (in 2D) or edges and vertices (in 3D) associated
with a loss of $H^2$-regularity for (\ref{eq:model-problem}) as given 
in Definition~\ref{def:H2-regular-edges}. 
Let Assumption~\ref{assumption:shift-theorem} be valid.  
Let $Iu \in V_h$ be arbitrary. Then we have 
$$
\|u - u_h\|_{L^2(\Omega)} \leq C h 
\|\widetilde\delta_\singular^{s_0-1} \nabla (u - Iu)\|_{L^2(\Omega)} .
$$
\end{theorem}
\begin{proof}
We may assume $s_0 < 1$ since the case $s_0 = 1$ corresponds to the standard 
duality argument with full elliptic regularity and
set $\varepsilon:= s_0 - 1/2 \in (0,1/2)$. 
The key observation is that, starting from the duality argument 
(\ref{eq:key-duality}), one can replace the weight function 
$\widetilde\delta_\Gamma^{-1/2+\varepsilon}$ in (\ref{eq:key-CS}) 
with {\em any} positive weight function. Taking as the weight
function $ \widetilde \delta_\singular^{-1/2+\varepsilon}$, we get 
\begin{equation} 
\label{eq:thm:abstract-optimal-L2-10}
\|e\|^2_{L^2(\Omega)} \lesssim \| \widetilde \delta_\singular^{-1/2+\varepsilon} \nabla (u - Iu)\|_{L^2(\Omega)} 
\| \widetilde \delta_\singular^{1/2-\varepsilon} \nabla (w-w_h)\|_{L^2(\Omega)}.
\end{equation} 
The estimate of $w-w_h$ in the weighted norm is done similarly as 
in the proof of Theorem~\ref{thm:optimal-L2}, taking into account the 
improved knowledge of the regularity of $w$. With 
$S_{\singular,ch}:=\{x \in \Omega\,|\, \delta_\singular(x) < c h\}$ we have 
the trivial bound
\begin{align}
\label{eq:thm:abstract-optimal-L2-12}
&
\|\widetilde\delta_{\singular}^{1/2-\varepsilon} \nabla (w - w_h)\|_{L^2(\Omega)}  \\
\nonumber 
&\quad  \lesssim 
\|\widetilde\delta_{\singular}^{1/2-\varepsilon} \nabla (w - w_h)\|_{L^2(S_{\singular,ch})} + 
\|\widetilde\delta_{\singular}^{1/2-\varepsilon} \nabla (w - w_h)\|_{L^2(\Omega\setminus S_{\singular,ch})},  
\end{align}
where the parameter $c$ will be selected sufficiently large below. 
The first term in (\ref{eq:thm:abstract-optimal-L2-12}) is estimated in exactly the 
same way as in (\ref{eq:thm:optimal-L2-1000}) and produces 
$$
\|\widetilde\delta_{\singular}^{1/2-\varepsilon} \nabla (w - w_h)\|_{L^2(S_{\singular,ch})}
\leq C h \|e\|_{L^2(\Omega)}. $$ 
The second term in (\ref{eq:thm:abstract-optimal-L2-10}) again requires the techniques
from the local error analysis of the FEM, this time with the appropriate modifications
to account for the boundary conditions. Inspection of the arguments in \cite[Sec.~{5.3}]{wahlbin95}
shows that the key estimate (\ref{eq:wahlbin-1}) extends up to the boundary in the following 
sense: 
\begin{equation}
\label{eq:wahlbin-1-up-to-bdy}
\|\nabla (w - w_h)\|_{L^2(B_r \cap \Omega)} \lesssim 
\|\nabla (w - I^k_h w)\|_{L^2(B_{r^\prime} \cap \Omega)} + 
\frac{1}{r^\prime - r} \|w - w_h\|_{L^2(B_{r^\prime} \cap \Omega)}; 
\end{equation}
besides the implicit assumption $r^\prime > r + O(h)$, the balls $B_r$ and $B_{r^\prime}$ 
are assumed to have the same center $x$ and satisfy one of the following conditions: 
\begin{enumerate} 
\item $B_{r^\prime} = B_{r^\prime}(x) \subset \Omega$; 
\item $x \in \partial \Omega$ and $B_{r^\prime}(x) \cap \Omega$ is a half-disk; 
\item $x$ is a vertex of $\Omega$;  
\item (only for $d =3$) $x$ lies on an edge $e$ and $B_{r^\prime}(x) \cap \Omega$ is 
a dihedral angle (i.e., the intersection of $\partial (B_{r^\prime}(x) \cap \Omega)$ with $\partial\Omega$
is contained in the two faces that share the edge $e$.
\end{enumerate}
The reason for the restriction of the location of the centers of the balls is that 
the procedure presented in \cite[Sec.~{5.3}]{wahlbin95} relies on Poincar\'e inequalities so that the 
number of possible shapes for the intersections $B_{r^\prime} \cap \Omega$ should be finite. 
A covering argument (see Theorem~\ref{thm:covering-2D} for the 2D case and 
Theorem~\ref{thm:covering-3D} for the 3D situation) then leads to the following bound  
with an appropriate $c_1 > 0$ (here, $c>0$ is implicitly assumed sufficiently large): 
\begin{align}
\label{eq:thm:abstract-optimal-L2-20}
& \|\widetilde \delta_\singular^{1/2-\varepsilon} \nabla (w - w_h)\|_{L^2(\Omega\setminus S_{\singular,ch})}
\lesssim \\
\nonumber 
& \quad 
\|\widetilde \delta_\singular^{1/2-\varepsilon} \nabla (w - I^k_h w)\|_{L^2(\Omega\setminus S_{\singular,c c_1 h})}
+ \|\widetilde \delta_\singular^{-1/2-\varepsilon} (w - w_h)\|_{L^2(\Omega\setminus S_{\singular,c c_1 h})}. 
\end{align}
The first term in (\ref{eq:thm:abstract-optimal-L2-20}) can be estimated with the improved regularity assertion of 
Lemma~\ref{lemma:weighted-shift-theorem-polygon} to produce (with appropriate $c_2 > 0$ 
and the implicit assumption on $c$ that $c c_1 c_2 > 2$)
\begin{align*}
& \|\widetilde \delta_\singular^{1/2-\varepsilon} \nabla (w - I^k_h w)\|_{L^2(\Omega\setminus S_{\singular,c c_1 h})}
\lesssim h 
\|\widetilde \delta_\singular^{1/2-\varepsilon} \nabla^2 w\|_{L^2(\Omega\setminus S_{\singular,c c_1 c_2 h})} \\
&\qquad 
{\lesssim}
h \left[ \|\widetilde\delta_\singular^{1/2-\varepsilon} e\|_{L^2(\Omega)} + \|w\|_{H^{3/2+\varepsilon}(\Omega)}\right]
\lesssim h \|e\|_{L^2(\Omega)}. 
\end{align*}
For the second term in (\ref{eq:thm:abstract-optimal-L2-20}) we note that $-1/2-\varepsilon < 0$ so that 
$\widetilde\delta_\singular^{-1/2-\varepsilon} \leq \widetilde\delta_\Gamma^{-1/2-\varepsilon}$. This leads to  
\begin{align*}
& \|\widetilde\delta_\singular^{-1/2-\varepsilon} (w - w_h)\|_{L^2(\Omega\setminus S_{\singular,c c_1 h})}
\lesssim 
\|\widetilde\delta_\singular^{-1/2-\varepsilon} (w - w_h)\|_{L^2(\Omega)} \\
& \quad \lesssim 
\|\widetilde\delta_\Gamma^{-1/2-\varepsilon} (w - w_h)\|_{L^2(\Omega)}
\lesssim 
h^{-2\varepsilon} \|\widetilde\delta_\Gamma^{-1/2+\varepsilon} (w - w_h)\|_{L^2(\Omega)};
\end{align*}
the term 
$h^{-2\varepsilon} \|\widetilde\delta_\Gamma^{-1/2+\varepsilon} (w - w_h)\|_{L^2(\Omega)}$ 
has already been estimated in (\ref{eq:thm:optimal-L2-700})
in the desired form.
\qed
\end{proof}
The regularity requirements on the solution $u$ can still be slightly weakened. 
As written, the exponent $s_0-1$ is related to the {\em global} regularity of the 
dual solution $w$. However, the developments above show that a {\em local}
lack of full regularity of the dual solution $w$ (and the bidual solution) 
needs to be offset by additional local regularity of the solution. To be more 
specific, we restrict our attention now to the 2D Laplacian, i.e., 
$\A = \operatorname*{Id}$. In this case, 
the situation can be expressed as follows with the aid of the singular 
exponents $\alpha_j:= \pi/\omega_j$, where $\omega_j \in (\pi,2\pi)$ 
is the interior angle at the reentrant vertices $A_j$, $j=1,\ldots,J$. 
\begin{corollary} 
\label{cor:2D-L2-abstract}
Let $\Omega\subset\BbbR^2$ be a polygon and let $\A = \operatorname*{Id}$. 
Let $\delta_j:= \operatorname*{dist}(\cdot,A_j)$, $j=1,\ldots,J$, for the 
$J$ reentrant corners. Set $\widetilde \delta_j:= h + \delta_j$. 
Let $\omega_j$ be the interior angle at $A_j$ and $\alpha_j = \pi/\omega_j $. 
Fix $\beta_j > 1-\alpha_j$ arbitrary. 
Then for any $Iu \in V_h$ 
$$
\|u - u_h\|_{L^2(\Omega)} \lesssim 
h \sum_{j=1}^J \|\widetilde \delta_j^{-\beta_j} \nabla (u - Iu)\|_{L^2(\Omega)}. 
$$
\end{corollary}
\begin{proof}
The proof follows by an inspection of how the regularity of 
the solution $w = Te$ of the dual problem enters the proof of 
Theorem~\ref{thm:abstract-optimal-L2}. 
By, e.g., \cite{grisvard85a} the solution $w = Te$ is in a weighted $H^2$-space with 
\begin{equation}
\label{eq:cor:2D-L2-abstract-10}
\|\prod_{j=1}^J \delta_j^{\beta_j} \nabla^2 w\|_{L^2(\Omega)} \lesssim \|e\|_{L^2(\Omega)}, 
\end{equation}
and Assumption~\ref{assumption:shift-theorem} holds with any $s_0 < \min_{j} \alpha_j$. 
The regularity assertion (\ref{eq:cor:2D-L2-abstract-10}) suggests to choose 
$\prod_{i=1}^J \widetilde \delta_i^{\beta_i}$ as the weight in the proof of 
Theorem~\ref{thm:optimal-L2}. Inspection of the procedure in the proof of
Theorem~\ref{thm:abstract-optimal-L2} then leads to the result.
\qed 
\end{proof}
We extract from this result another corollary that we will prove useful 
in the numerical results. We formulate it in terms of (standard, unweighted) 
Sobolev regularity in order to emphasize the difference in regularity requirements of 
the solution near the reentrant corners and away from them: 
\begin{corollary}
\label{cor:2D-L2-concrete}
Assume the hypotheses of Corollary~\ref{cor:2D-L2-abstract}. 
Let $s > 1$ and $s_i > 1$, $i=1,\ldots,J$. Let ${\mathcal U} := \Omega \setminus \cup \overline{\mathcal U}_i $, 
for some neighborhoods ${\mathcal U}_i$ of the reentrant vertices $A_i$. 
Let $u \in H^{s_i}({\mathcal U}_i)$, $i=1,\ldots,J$ and $u \in H^s({\mathcal U})$. Then 
for arbitrary $\varepsilon > 0$ 
$$
\|u - u_h\|_{L^2(\Omega)} 
\leq C_\varepsilon h^\tau, \quad \tau := 
\min (1+k, s, \min_{j=1,\ldots,J}(-1 + \alpha_j + s_j-\varepsilon) ). 
$$
\end{corollary}
\begin{proof}
The approximant $Iu$ in Corollary~\ref{cor:2D-L2-abstract} may be taken as 
any standard nodal interpolant or the Scott-Zhang 
projection. Then standard estimates and Corollary~\ref{cor:2D-L2-abstract} 
produce with the choice $\beta_j := 1-\alpha_j +\varepsilon$ for arbitrary small but  fixed $\varepsilon > 0$: 
\begin{align*}
\|u - u_h\|_{L^2(\Omega)} 
& \lesssim 
h \min_{j=1,\ldots,J}  \{h^{\min\{k,s-1\}}, h^{-\beta_j + s_j-1}\} \\
& \lesssim \min_{j=1,\ldots,J} \{h^{\min\{k+1,s\}}, h^{\alpha_j + s_j-1-\varepsilon}\}. 
\end{align*} 
\qed
\end{proof}
\section{Optimal $L^2(S_h)$-convergence}
\label{sec:error-on-strip}
Additional regularity of the solution also allows us to prove that the 
error on the strip $S_h$ of width $O(h)$ near $\Gamma$ is of higher order: 
\begin{theorem}
\label{thm:error-on-strip}
Let Assumption~\ref{assumption:shift-theorem} be valid. Then the FEM error
$u - u_h$ satisfies 
\begin{align*}
\|u - u_h\|_{L^2(S_h)} &\lesssim h^{k+3/2} (1 + \delta_{k,1} |\ln h|) \|u\|_{B^{k+3/2}_{2,1}(\Omega)}, \\
\|u - u_h\|_{L^2(S_h)} &\lesssim h^{s+3/2} (1 + \delta_{k,1} |\ln h|) \|u\|_{B^{s+3/2}_{2,\infty}(\Omega)}, 
\quad s \in (0,k), \\
\|u - u_h\|_{L^2(S_h)} &\lesssim h^{3/2} (1 + \delta_{k,1} |\ln h|) \|u\|_{B^{3/2}_{2,1}(\Omega)}, 
\end{align*}
where $\delta_{k,1}$ is the Kronecker symbol. The implies constant depends on the shape regularity 
of the triangulation, $\Omega$, and the coefficient $\A$. Specifically, it depends on 
$\alpha_0$ and $\|\A\|_{C^{0,1}(\overline{\Omega})}$ and, in the case $k > 1$, additionally on 
$\|\A\|_{C^{1,1}(\overline{\Omega})}$. 
\end{theorem}
\begin{remark}
{\rm 
\begin{enumerate}
\item The regularity requirement $B^{k+3/2}_{2,1}(\Omega)$ can be weakened: 
it suffices that $u$ be in $B^{k+3/2}_{2,1}(S_D)$ in a fixed neighborhood $S_D$
of $\Gamma$ and in $H^{k+1}(\Omega)$. See \cite{melenk-praetorius-wohlmuth14}
for the details of a closely related problem.
\item 

Since $B^{s+3/2}_{2,\infty}(\Omega) \supset H^{s+3/2}(\Omega)$, the assertions 
for $s \in (0,k)$ can be weakened by replacing $\|u\|_{B^{s+3/2}_{2,\infty}(\Omega)}$ with 
$\|u\|_{H^{s+3/2}(\Omega)}$ on the right-hand side. Only for the limiting cases $s = 0$ and $s = k$, 
we require the stronger requirement 
$u \in B^{s+3/2}_{2,1}(\Omega) \subset H^{s+3/2}(\Omega)$. 
\eremk
\end{enumerate}
}
\end{remark}
\begin{proof}
The structure of the proof is very similar to that of Theorem~\ref{thm:optimal-L2}.
The main difference arises from the fact that the right-hand side of the dual
problem is supported by the thin neighborhood $S_h$, and this support
property has to be exploited. 

Let $e = u - u_h$. 
Let $\cutoff$ be the characteristic function of $S_h$. 
Let $w = T(\cutoff e)$ and $w_h \in V_h$ its Galerkin approximation. Again, 
Galerkin orthogonality for $u - u_h$ and $w - w_h$ implies 
\begin{eqnarray}
\nonumber 
\|e\|^2_{L^2(S_h)} &=& \langle e,\chi e\rangle = a(e,w) = a(e,w - w_h) = 
a(u - I^k_h u,w -w_h) \\
\label{eq:thm:error-on-strip-5}
&\leq& C \|\widetilde\delta_\Gamma^{-1/2-\varepsilon} \nabla (u - I^k_h u)\|_{L^2(\Omega)}
\|\widetilde\delta_\Gamma^{1/2+\varepsilon} \nabla (w - w_h)\|_{L^2(\Omega)},
\end{eqnarray}
where $\varepsilon \ge 0$ is arbitrary (in fact, $\varepsilon \in \BbbR$ would be admissible).
We flag at this point already that we will select $\varepsilon = 0$ for $k = 1$ and 
$\varepsilon > 0$ arbitrary (but sufficiently small) for $k > 1$. 
Each of the two factors in (\ref{eq:thm:error-on-strip-5}) is estimated separately. 

{\em 1.~step:}
For the first factor in (\ref{eq:thm:error-on-strip-5}) we use 
approximation properties of the Scott-Zhang operator $I^k_h$ together with 
Lemma~\ref{lemma:weighted-embedding} to get for $j \in \{0,\ldots,k\}$
\begin{eqnarray}
\nonumber 
\|\widetilde\delta_\Gamma^{-1/2-\varepsilon} \nabla (u - I^k_h u)\|_{L^2(\Omega)}
&\lesssim & h^{j} \|\widetilde\delta_\Gamma^{-1/2-\varepsilon} \nabla^{j+1} u\|_{L^2(\Omega)}, \\   
\label{eq:thm:error-on-strip-10}
&\lesssim & h^{j} 
\begin{cases} |\ln h|^{1/2} \|\nabla^{j+1} u\|_{B^{1/2}_{2,1}(\Omega)} & \mbox{ if $\varepsilon = 0$}\\
h^{-\varepsilon} \|\nabla^{j+1} u\|_{B^{1/2}_{2,1}(\Omega)} & \mbox{ if $\varepsilon > 0$}. 
\end{cases}
\qquad 
\end{eqnarray}
With the Kronecker symbol $\delta_{0,\varepsilon}$, we have shown
for $j \in \{0,1,\ldots,k\}$
\begin{align}
\label{eq:thm:error-on-strip-10a}
\|\widetilde\delta_\Gamma^{-1/2-\varepsilon} \nabla (u - I^k_h u)\|_{L^2(\Omega)}
\lesssim h^j h^{-\varepsilon} (1 + \delta_{0,\varepsilon} |\ln h|^{1/2}) \|u\|_{B^{j+3/2}_{2,1}(\Omega)}. 
\end{align}
Since the Scott-Zhang operator $I^k_h$ is defined on $H^1(\Omega)$ irrespective of boundary conditions, 
we may use an interpolation argument to lift the restriction to integer values $j$. 
Specifically, the reiteration theorem (cf., e.g., 
\cite[Thm.~{23.6}]{tartar07}) asserts that the Besov space 
$B^{s+3/2}_{2,\infty}(\Omega)$, which we have defined by interpolation between (integer order)
Sobolev spaces, coincides with the interpolation space between Besov spaces, viz., 
$$
B^{s+3/2}_{2,\infty}(\Omega) = (B^{3/2}_{2,1}(\Omega),B^{k+3/2}_{2,1}(\Omega))_{s/k,\infty}
\qquad \mbox{ (equivalent norms).}
$$
Hence, we may 
decompose for arbitrary $t > 0$ a function 
$u \in B^{s+3/2}_{2,\infty}(\Omega)$, $s \in (0,k)$, 
as $u = u - u_1 + u_1$ with 
$u_1 \in B^{k+3/2}_{2,1}(\Omega)$ and $u_0:=u - u_1 \in B^{3/2}_{2,1}(\Omega)$ together with 
$$
\|u_0\|_{B^{3/2}_{2,1}(\Omega) } \leq C t^{s/k} \|u\|_{B^{s+3/2}_{2,\infty}(\Omega)}, 
\qquad 
\|u_1\|_{B^{k+3/2}_{2,1}(\Omega) } \leq C t^{s/k-1} \|u\|_{B^{s+3/2}_{2,\infty}(\Omega)}. 
$$
Writing $u - I^k_h u = \left( u_0 - I^k_h u_0\right) + 
\left( u_1 - I^k_h u_1\right)$ we can use 
(\ref{eq:thm:error-on-strip-10a}) with $j = k$ for the second term in brackets 
and $j=0$ for the first term in brackets to get with the choice 
$t = h^k$
\begin{align}
\label{eq:thm:error-on-strip-10b}
\|\widetilde\delta_\Gamma^{-1/2-\varepsilon} \nabla (u - I^k_h u)\|_{L^2(\Omega)}
\lesssim 
h^s h^{-\varepsilon} (1 + \delta_{0,\varepsilon} |\ln h|^{1/2}) 
\|u\|_{B^{s+3/2}_{2,\infty}(\Omega)}. 
\end{align}
Combining the estimates 
(\ref{eq:thm:error-on-strip-10a}) with $j=0$ and $j=k$ and 
(\ref{eq:thm:error-on-strip-10b}) for $s \in (0,k)$  we arrive at 
\begin{align}
\label{eq:thm:error-on-strip-153}
&\|\widetilde\delta_\Gamma^{-1/2-\varepsilon} \nabla (u - I^k_h u)\|_{L^2(\Omega)} \\
\nonumber 
&\qquad \lesssim 
(1 + \delta_{0,\varepsilon} |\ln h|^{1/2}) 
h^{-\varepsilon} 
\begin{cases}
h^s \|u\|_{B^{s+3/2}_{2,1}(\Omega)}, & s = 0,\\ 
h^s \|u\|_{B^{s+3/2}_{2,\infty}(\Omega)}, & s \in  (0,k), \\
h^s \|u\|_{B^{k+3/2}_{2,1}(\Omega)}, & s = k.
\end{cases}
\end{align}

{\em 2.~step:}
The second factor in (\ref{eq:thm:error-on-strip-5}) requires more work. We start with 
a regularity assertion for $w$ that exploits the support properties of $\cutoff e$ and 
follows from (\ref{eq:lemma:B32-regularity-2}) and (\ref{eq:lemma:B32-regularity-3}):
\begin{eqnarray}
\label{eq:thm:error-on-strip-regularity-of-w-a}
\|w\|_{B^{3/2}_{2,\infty}(\Omega)} &\lesssim & h^{1/2} \|\cutoff e\|_{L^2(\Omega)}, \\
\label{eq:thm:error-on-strip-regularity-of-w-b}
\|w\|_{H^{3/2+\varepsilon}(\Omega)} &\lesssim & h^{1/2-\varepsilon} \|\cutoff e\|_{L^2(\Omega)}, 
\qquad \varepsilon \in (0,s_0-1/2].
\end{eqnarray}
We obtain an energy error estimate for $w-w_h$ in the standard way by using 
quasi-optimality, the approximation properties of $V_h$, and the regularity 
assertion (\ref{eq:thm:error-on-strip-regularity-of-w-a}):
\begin{eqnarray}
\label{eq:thm:error-on-strip-15}
\|w -w_h\|_{H^1(\Omega)} &\lesssim& \inf_{v \in V_h} \|w - v\|_{H^1(\Omega)} 
\lesssim h^{1/2} \|w\|_{B^{3/2}_{2,\infty}(\Omega) }
\lesssim h \|\cutoff e\|_{L^2(\Omega)}. 
\quad 
\end{eqnarray}
Lemma~\ref{lemma:weighted-L2-estimate} is applicable with $z = w-w_h$; hence, 
obtain with (\ref{eq:thm:error-on-strip-15})
\begin{eqnarray}
\label{eq:thm:error-on-strip-17}
\|\delta_\Gamma^{-1/2+\varepsilon} (w - w_h)\|_{L^2(\Omega)} &\lesssim & h^{3/2+\varepsilon}\|\cutoff e\|_{L^2(\Omega)}, 
\qquad \varepsilon \in (0,s_0-1/2], \\
\label{eq:thm:error-on-strip-18}
\|\widetilde \delta_\Gamma^{-1/2} (w - w_h)\|_{L^2(\Omega)} &\lesssim & h^{3/2} |\ln h|^{1/2} \|\cutoff e\|_{L^2(\Omega)}. 
\end{eqnarray}
The bound (\ref{eq:thm:error-on-strip-5}) informs us that control of $w - w_h$ in a weighted 
$H^1$-norm is required. In this direction, we first write for a constant $c > 0$ that will be determined 
later sufficiently large 
\begin{align}
\nonumber 
& \|\widetilde\delta_\Gamma^{1/2+\varepsilon} \nabla (w - w_h)\|_{L^2(\Omega)} \\
\nonumber 
&\qquad \leq 
\|\widetilde\delta_\Gamma^{1/2+\varepsilon} \nabla (w - w_h)\|_{L^2(S_{ch})} 
+ \|\widetilde\delta_\Gamma^{1/2+\varepsilon} \nabla (w - w_h)\|_{L^2(\Omega\setminus S_{ch})} \\
\nonumber 
&\qquad \leq 
C h^{1/2+\varepsilon} \|\nabla( w - w_h) \|_{L^2(\Omega)} + 
\|\widetilde\delta_\Gamma^{1/2+\varepsilon} \nabla (w - w_h)\|_{L^2(\Omega\setminus S_{ch})}\\
\label{eq:thm:error-on-strip-100}
&\qquad \stackrel{(\ref{eq:thm:error-on-strip-15})}{\leq}
C h^{3/2+\varepsilon} \|\cutoff e \|_{L^2(\Omega)} + 
\|\widetilde\delta_\Gamma^{1/2+\varepsilon} \nabla (w - w_h)\|_{L^2(\Omega\setminus S_{ch})}. 
\end{align}
We emphasize that $\varepsilon = 0$ is allowed in (\ref{eq:thm:error-on-strip-100}). 
It remains to control 
$\|\widetilde\delta_\Gamma^{1/2+\varepsilon} \nabla (w - w_h)\|_{L^2(\Omega\setminus S_{ch})}$. 
This is done again with the same arguments from the local error analysis as in the 
proof of Theorem~\ref{thm:optimal-L2}. The estimate (\ref{eq:wahlbin-10}) holds verbatim, i.e., 
\begin{align}
\label{eq:wahlbin-10-thm:error-on-strip}
& \|\delta_\Gamma^{1/2+\varepsilon} \nabla (w - w_h)\|_{L^2(\Omega\setminus S_{c h})} \\
\nonumber 
& \quad \lesssim \|\delta_\Gamma^{1/2+\varepsilon} \nabla (w - I^k_h w)\|_{L^2(\Omega\setminus S_{c c_1 h})} 
+ \|\delta_\Gamma^{-1/2+\varepsilon} (w - w_h)\|_{L^2(\Omega\setminus S_{c c_1 h})}. 
\end{align}
We emphasize that $\varepsilon  = 0$ is admissible in (\ref{eq:wahlbin-10}). 
As in the proof of Theorem~\ref{thm:optimal-L2}, the constant $c$ will be selected in dependence 
of various inverse estimates that are applied. 
Combining (\ref{eq:thm:error-on-strip-17}), (\ref{eq:thm:error-on-strip-18}), (\ref{eq:thm:error-on-strip-100}),
(\ref{eq:wahlbin-10-thm:error-on-strip})
we see that we have shown 
\begin{align}
\label{eq:thm:error-on-strip-150}
& 
\|\widetilde \delta_\Gamma^{1/2+\varepsilon} \nabla (w - w_h)\|_{L^2(\Omega)} \\
\nonumber 
& \quad 
\lesssim 
\begin{cases}
h^{3/2+\varepsilon} \|\cutoff e\|_{L^2(\Omega)} + \|\widetilde\delta_\Gamma^{1/2+\varepsilon} \nabla (w - I^k_h w)\|_{L^2(\Omega\setminus S_{c_1 c h})} & \mbox{ if $\varepsilon > 0$,} \\
h^{3/2} |\ln h|^{1/2} \|\cutoff e\|_{L^2(\Omega)} + \|\widetilde\delta_\Gamma^{1/2} \nabla (w - I^k_h w)\|_{L^2(\Omega\setminus S_{c_1 c h})} & \mbox{ if $\varepsilon = 0$.}
\end{cases}
\end{align}
{\em 3.~step:}
We estimate the approximation error 
$\|\widetilde\delta_\Gamma^{1/2+\varepsilon} \nabla (w  - I^k_h w)\|_{L^2(\Omega\setminus S_{c_1 c h})}$. 
At this point the cases $k = 1$ and $k > 1$ diverge: since $w$ solves a {\em homogeneous} elliptic equation
on $\Omega\setminus S_{c_1 c h}$ (if $c_1 c > 1$), interior regularity is available so that higher order approximation 
can be brought to bear if $k > 1$ in contrast to the case $k  =1$. We start with the simpler case $k = 1$.

{\em The case $k = 1$:}
{}From standard approximation results for $I^k_h$, the inverse estimate 
of Lemma~\ref{lemma:5.4}, (\ref{item:lemma:5.4-1}), and (\ref{eq:thm:error-on-strip-regularity-of-w-a})
we get for a constant $c_2  \in (0,1)$ (implicitly, we assume that $c$ is so large that 
$c_2 c_1 c h > 2h$)
\begin{align}
\nonumber 
& \|\widetilde\delta_\Gamma^{1/2} \nabla (w - I^k_h w)\|_{L^2(\Omega\setminus S_{c_1 c h})}
\lesssim  h^1 \|\widetilde\delta_\Gamma^{1/2} \nabla^2 w\|_{L^2(\Omega\setminus S_{c_2 c_1 c h}) } \\
\label{eq:thm:error-on-strip-200}
&\qquad \lesssim  h |\ln h|^{1/2} \|w\|_{B^{3/2}_{2,\infty}(\Omega)} 
\lesssim h^{3/2} |\ln h|^{1/2} \|\cutoff e\|_{L^2(\Omega)} . 
\end{align} 
Inserting 
(\ref{eq:thm:error-on-strip-153})
(with $\varepsilon = 0$) 
with the combination of (\ref{eq:thm:error-on-strip-200}) and 
(\ref{eq:thm:error-on-strip-150})  (again with $\varepsilon = 0$) in (\ref{eq:thm:error-on-strip-5}) yields 
the desired final estimate for the case $k = 1$ if we fix $c = 2/(c_1 c_2)$.

{\em The case $k > 1$:} 
We fix an $\varepsilon  \in (0,s_0-1/2]$ arbitrary. 
{}From standard approximation results for $I^k_h$, the inverse estimates
of Lemma~\ref{lemma:5.4}, and the regularity assertion (\ref{eq:thm:error-on-strip-regularity-of-w-b})
we get (again for suitable constants $c_2$, $c_3 \in (0,1)$ and the implicit assumption that
$c$ is such that $c_3 c_2 c_1 c$ is sufficiently large)

\begin{align}
\nonumber 
& \|\widetilde\delta_\Gamma^{1/2+\varepsilon} \nabla (w - I^k_h w)\|_{L^2(\Omega\setminus S_{c_1 c h})}  
\lesssim h^2 \|\widetilde\delta_\Gamma^{1/2+\varepsilon} \nabla^3 w\|_{L^2(\Omega\setminus S_{c_2 c_1 c h}) }  \\
\nonumber 
& \stackrel{\text{Lem.}~\ref{lemma:5.4}, (\ref{item:lemma:5.4-2})}{\lesssim}  
h^2 \left[ \|\widetilde\delta_\Gamma^{-1/2+\varepsilon} \nabla^2 w\|_{L^2(\Omega\setminus S_{c_3 c_2 c_1 c h}) }  
+ \|\widetilde \delta_\Gamma^{1/2+\varepsilon} \nabla w\|_{L^2(\Omega\setminus S_{c_3 c_2 c_1 c h})}
\right] \\
\nonumber 
&\stackrel{\phantom{\text{Lem.}~\ref{lemma:5.4}, (\ref{item:lemma:5.4-2})}}{\lesssim}  h^{2-1/2+\varepsilon} 
\left[ \|\nabla^2 w\|_{L^2(\Omega\setminus S_{c_3 c_2 c_1 c h}) }  
     + \|\nabla w\|_{L^2(\Omega\setminus S_{c_3 c_2 c_1 c h})}\right] \\
& \stackrel{\text{Lem.}~\ref{lemma:5.4}, (\ref{item:lemma:5.4-4})}{\lesssim}  
h^{2-1/2+\varepsilon} \|\cutoff e\|_{L^2(\Omega)}. 
\end{align}
Combining this 
with (\ref{eq:thm:error-on-strip-153})
produces in (\ref{eq:thm:error-on-strip-5}) the desired final estimate for the case $k > 1$. 
%
%
\qed 
\end{proof}
{}From Theorem~\ref{thm:error-on-strip} we can extract optimal convergence 
estimates for the flux error $\|\partial_n  (u  - u_h)\|_{L^2(\Gamma)}$: 
\begin{corollary}
\label{cor:optimal-flux-minimal-regularity}
Let Assumption~\ref{assumption:shift-theorem} be valid. Then with the Kronecker symbol $\delta_{k,1}$
\begin{equation*}
\|\partial_n u - \partial_n u_h\|_{L^2(\Gamma)} \lesssim (1 + \delta_{k,1} |\ln h|) 
\begin{cases} 
h^k  \|u\|_{B^{k+3/2}_{2,1}(\Omega)}, & \\
h^s  \|u\|_{B^{s+3/2}_{2,\infty}(\Omega)}, & s \in (0,k),  \\
     \|u\|_{B^{3/2}_{2,1}(\Omega)}. & 
\end{cases} 
\end{equation*}
\end{corollary}
\begin{proof}
Structurally, the proof follows \cite[Cor.~{6.1}]{melenk-wohlmuth12} in that estimating 
the error on $\Gamma$ is transferred to an estimate on the strip $S_h$. The triangle 
inequality gives 
\begin{equation}
\label{eq:cor:optimal-flux-minimal-regularity-10}
\|\partial_n (u - u_h)\|_{L^2(\Gamma)} \leq 
\|\partial_n (u - I^k_h u)\|_{L^2(\Gamma)} + 
\|\partial_n (I^k_h u - u_h)\|_{L^2(\Gamma)}. 
\end{equation}
The two terms in (\ref{eq:cor:optimal-flux-minimal-regularity-10}) are estimated separately.

{\em 1.~step:} We claim that 
\begin{equation}
\label{eq:cor:optimal-flux-minimal-regularity-20}
\|\partial_n (u - I^k_h u)\|_{L^2(\Gamma)} 
\lesssim 
\begin{cases}
h^k \|u\|_{B^{k+3/2}_{2,1}(\Omega)} \\
h^s \|u\|_{B^{s+3/2}_{2,\infty}(\Omega)}, \quad s \in (0,k), \\
 \|u\|_{B^{3/2}_{2,1}(\Omega)}.
\end{cases}
\end{equation}
We will only show the limiting cases $u \in B^{k+3/2}_{2,1}(\Omega)$ and $u \in B^{3/2}_{2,1}(\Omega)$; 
the intermediate cases follow by an interpolation argument similar to the one used in the proof of 
Theorem~\ref{thm:error-on-strip}. For the case of maximal regularity, we 
use an elementwise multiplicative trace inequality for the elements abutting 
$\Gamma$ to get 
\begin{align*}
& \|\partial_n (u - I^k_h u)\|_{L^2(\Gamma)}  \lesssim 
h^{k/2}\sqrt{ \|\nabla^{k+1} u\|_{L^2(S_{2h})}} 
h^{(k-1)/2} \sqrt{\|\nabla^{k+1} u\|_{L^2(S_{2h})}}  \\
&\lesssim h^{k-1/2} \|\nabla^{k+1} u\|_{L^2(S_{2h})}
\stackrel{(\ref{eq:lemma:weighted-embedding-5})}{\lesssim} h^k \|\nabla^{k+1} u\|_{B^{1/2}_{2,1}(\Omega)}
\lesssim h^k \|u\|_{B^{k+3/2}_{2,1}(\Omega)}. 
\end{align*}
For the case of minimal regularity, $u \in B^{3/2}_{2,1}(\Omega)$ we first note 
that we obtain from (\ref{eq:lemma:weighted-embedding-6}) that 
$\|v\|_{L^2(\Gamma)} \lesssim \|v\|_{B^{1/2}_{2,1}(\Omega)}$. 
Using this and inverse estimates, we get  
\begin{align*}
\|\partial_n (u - I^k_h u)\|_{L^2(\Gamma)} 
&\leq 
\|\partial_n u\|_{L^2(\Gamma)} + \|\partial_n I^k_h u\|_{L^2(\Gamma)}  \\
&\lesssim 
\|\nabla u\|_{B^{1/2}_{2,1}(\Omega)} + h^{-1/2} \|\nabla I^k_h u\|_{L^2(S_h)} \\
&\stackrel{I^k_h \text{stable}}{\lesssim}
\|\nabla u\|_{B^{1/2}_{2,1}(\Omega)} + h^{-1/2} \|\nabla u\|_{L^2(S_{2h})} 
\stackrel{(\ref{eq:lemma:weighted-embedding-5})}{\lesssim}
\|\nabla u\|_{B^{1/2}_{2,1}(\Omega)}. 
\end{align*}

{\em 2.~step:} The term 
$\|\partial_n (I^k_h u - u_h)\|_{L^2(\Gamma)}$ in 
(\ref{eq:cor:optimal-flux-minimal-regularity-10}) is controlled with inverse estimates and 
Theorem~\ref{thm:error-on-strip} as follows: 
\begin{align*}
\|\partial_n (I^k_h u - u_h)\|_{L^2(\Gamma)}
&\lesssim h^{-1/2} \|\nabla(I^k_h u - u_h)\|_{L^2(S_h)} 
 \lesssim h^{-3/2} \|I^k_h u - u_h\|_{L^2(S_h)}  \\
&\lesssim h^{-3/2} \|u - I^k_h u\|_{L^2(S_h)} + h^{-3/2} \|u - u_h\|_{L^2(S_h)}. 
\end{align*}
The term $\|u - I^k_h u\|_{L^2(S_h)}$ can be controlled with the approximation properties 
of $I^k_h$ in the desired fashion: 
$\|u - I^k_h u\|_{L^2(S_h)} \lesssim h \|\nabla u\|_{L^2(S_{2h})} 
\lesssim h^{3/2} \|\nabla u\|_{B^{1/2}_{2,1}(\Omega)}$. 
The contribution $\|u - u_h\|_{L^2(S_h)}$ is 
estimated with the aid of Theorem~\ref{thm:error-on-strip}.  
\qed
\end{proof}
\section{Extension of the results of \protect{\cite{melenk-wohlmuth12}}}
\label{sec:melenk-wohlmuth12}
The arguments of the present paper are similar to those underlying 
\cite{melenk-wohlmuth12}, in spite of the fact that we did not employ
the anisotropic norms that we introduced in \cite{melenk-wohlmuth12}
but instead worked with weighted Sobolev spaces. 
A feature of the analysis here that was not 
present in \cite{melenk-wohlmuth12} is our FEM error analysis in 
Lemma~\ref{lemma:weighted-L2-estimate} for a weighted $L^2$-estimate, which, in 
turn, relies on the regularity assertions of Lemma~\ref{lemma:regularity-weighted-rhs} 
for problems with data in weighted spaces. 
This additional technical issue was circumvented
in \cite{melenk-wohlmuth12} by assuming convexity of $\Omega$ so that
optimal order $L^2$-estimates could be cited from the literature. The present 
analysis provides the necessary technical tools to remove this simplification in 
\cite{melenk-wohlmuth12}, where a more complex mortar setting is analyzed. 
It is possible to make use of weighted $L^2$-estimates similar to those of 
Lemma~\ref{lemma:weighted-L2-estimate} in the setting of \cite{melenk-wohlmuth12}. 
For that, regularity results of the type provided in 
Lemma~\ref{lemma:regularity-weighted-rhs} have to be used. 
The outcome of this refinement is that the main results of \cite{melenk-wohlmuth12},
namely, \cite[Thm.~{2.1}]{melenk-wohlmuth12}, which provides $L^2$-estimates 
on strips of width $O(h)$ around the skeleton,  and \cite[Thm.~{2.5}]{melenk-wohlmuth12},
which provides optimal order approximations for the mortar variable, hold true
if the geometry is such that Assumption~\ref{assumption:shift-theorem} is valid. 
We will not provide the details of the arguments here and refer to 
\cite[Appendix~\appB]{melenk-wohlmuth14a} instead. Nevertheless, for future reference
we record the end result: 

\begin{theorem}
\label{thm:generalize-melenk-wohlmuth12-vorn}
In \cite[Thms.~{2.1}, {2.5}]{melenk-wohlmuth12}, the assumption
of convexity of $\Omega$ can be replaced with
\cite[Assumption~{(5.2)}]{melenk-wohlmuth12}. 
\end{theorem}
\section{Numerical results}
\label{sec:numerics}
We consider the simple model equation $- \Delta u=f$ in $\Omega
\subset {\mathbb R}$, $d \in \{2,3\}$ with inhomogeneous Dirichlet
boundary conditions. These are realized numerically by nodal interpolation 
of the prescribed exact solution $u$, and the data $f$ is also computed 
from $u$.
In the case of a non-smooth solution, we use a suitable quadrature
formula on finer meshes to guarantee that the $L^2$-error is
accurately evaluated. 
\subsection{Two-dimensional results}
We use a sequence of uniformly refined triangular meshes, where each element 
is split into four triangles.
\subsubsection{Lowest order discretization}
We consider two typical domains for reentrant corners. We start
with the L-shaped domain $(-1,1)^2 \setminus [0,1] \times [-1,0]$ and
then consider a slit domain 
$(-1,1)^2 \setminus \left( (0,1) \times \{0\}\right)$. 
In both cases, the prescribed solution is given in polar
coordinates by
$ u(r,\phi) = r^\alpha \sin (a \phi)$ where $\alpha$, $a$ are given
parameters. For non-integer $\alpha$, we have $u \in B^{1+\alpha}_{2,\infty}(\Omega)$ by 
\cite[Thm.~{2.1}]{babuska-osborn91}. 
Moreover, we test the influence of the position $(x_0,y_0)$ of the weak
singularity at $r=0$ by defining $r^2 := (x-x_0)^2 + (y-y_0)^2$.
We note that irrespective of the location $(x_0,y_0)$ of the singularity
on the boundary $\Gamma$, we have $u \in B^{1+\alpha}_{2,\infty}(\Omega) 
\subset H^{1+\alpha-\varepsilon}(\Omega)$ for any $\varepsilon >0$. 

\begin{table}[ht]
\begin{center}\begin{tabular}{|r||c|c||c|c||c|c|}
\hline
&\multicolumn{2}{|c||}{$(x_0,y_0) = (0,0)$}
&\multicolumn{2}{|c||}{$(x_0,y_0) = (0.5,0)$}
&\multicolumn{2}{|c|}{$(x_0,y_0) = (0,1)$}  \\
&\multicolumn{2}{|c||}{$a = \pi/2$}
&\multicolumn{2}{|c||}{$a=\pi$}
&\multicolumn{2}{|c|}{$a=\pi$}  \\
\hline
DOFs&$L^2$-error  & rate &$L^2$-error  & rate & $L^2$-error & rate  \\ \hline
81     		&6.1585e-03  &  -         &6.8141e-03    & -          &6.2506e-03       & -   \\
289    		&2.6986e-03  &1.19     &   2.5648e-03 &1.41     &      2.1211e-03 &1.56    \\
1.089   		&1.1123e-03  &1.28     &   8.8428e-04 &1.54     &      6.7413e-04 &1.65    \\
4.225   		&4.4037e-04  &1.34     &   2.9202e-04 &1.60     &      2.0903e-04 &1.69    \\
16.641  	&1.7107e-04  &1.36     &   9.4164e-05 &1.63     &      6.4027e-05 &1.71    \\
66.049  	&6.5689e-05  &1.38     &   2.9909e-05 &1.65     &      1.9471e-05 &1.72    \\
263.169 	&2.5030e-05  &1.39     &   9.4012e-06 &1.67     &      5.8930e-06 &1.72    \\
1.050.625	&9.4877e-06  &1.40     &   2.9328e-06 &1.68     &      1.7774e-06 &1.73    \\
4.198.401	&3.5834e-06  &1.40     &   9.0968e-07 &1.69     &      5.3475e-07 &1.73    \\
  \hline
\end{tabular}
   \caption{L-shaped domain, $k=1$: Influence of the position of singularity for $\alpha = 0.75$.}
\label{tab:lshape1}
\end{center}
\end{table}

For the L-shaped domain, the shift parameter $s_0$ can be taken to be any 
$s_0 < 2/3$. 
{}From the theoretical results in Section~\ref{sec:FEM-L2}, we therefore expect the error
decay to have a rate of at least $
\min(2,1+\alpha - 1/3)$ uniformly in the position $(x_0,y_0)$ of the singularity. 
Table~\ref{tab:lshape1} shows the numerical results
for $\alpha = 0.75$ and $a=2/3\pi$, for which 
$\min(2,1+\alpha - 1/3) = 1.417$. As it can be seen for $(x_0,y_0) =(0,0)$,
we observe a good agreement with Theorem~\ref{thm:optimal-L2}.
However for the locations $(x_0,y_0)
=(0.5,0)$ and  $(x_0,y_0)
=(0,1)$, the rates are substantially better.
This can be explained by the more refined analysis of Section~\ref{sec:L2-error-polygon}. Using
  Corollary~\ref{cor:2D-L2-concrete}, we expect an improved convergence
    rate of $1.75$ for these
cases.

\begin{table}[ht]
\begin{center}\begin{tabular}{|r||c|c||c|c||c|c|}
\hline
&\multicolumn{2}{|c||}{$\alpha = 10/9$}
&\multicolumn{2}{|c||}{$\alpha =4/3$}
&\multicolumn{2}{|c|}{$\alpha = 3/2$}  \\
\hline
DOFs		&$L^22$-error  & rate &$L^2$-error  & rate & $L^2$-error & rate  \\ \hline
81     		&6.5660e-03  &  -       &   8.6776e-03   &  -     	 & 8.9932e-03  &  - 		 		\\
289    		&2.3309e-03  &1.49   & 	2.8523e-03  & 1.61    & 2.8151e-03  &1.68  			 \\
1.089   		&7.3413e-04  &1.67   & 	8.2870e-04  & 1.78    &	7.8034e-04  &1.85			 \\
4.225   		&2.2257e-04  &1.72   & 	2.3073e-04  & 1.84 	 &	2.0751e-04  &1.91			 \\
16.641  	&6.5650e-05  &1.76   & 	6.2539e-05  & 1.88    &	5.3910e-05  &1.94			 \\
66.049  	&1.9056e-05  &1.78   & 	1.6688e-05  & 1.91    &	1.3835e-05  &1.96 			 \\
263.169 	&5.4810e-06  &1.80   & 	4.4099e-06  & 1.92    &	3.5256e-06  &1.97		 	\\
1.050.625	&1.5690e-06  &1.80   & 	1.1580e-06  & 1.93    &	8.9467e-07  &1.98			 \\
4.198.401	&4.4822e-07  &1.81   &   3.0279e-07  & 1.94   	 & 2.2641e-07  &1.98     		\\
 \hline
\end{tabular}
   \caption{L-shaped domain, $k=1$: Influence of exponent $\alpha $ for
     $a=2/3 \pi$ and $(x_0,y_0) =(0,0)$.}
\label{tab:lshape2}
\end{center}
\end{table}

Table~\ref{tab:lshape2} shows the results for $(x_0,y_0) =(0,0) $ and
$\alpha \in \{10/9,4/3,3/2\}$. From Theorem~\ref{thm:optimal-L2}, we
expect convergence rates of $ 1.78$, $2$, and $2$, respectively. The
observed numerical rates of $ 1.81$, $1.94$, and $1.98$ are quite close.

The situation is similar for the slit domain where the 
regularity of the dual problem is even further reduced. It corresponds
to a limiting case of our theory, which, strictly speaking, we did not cover, 
since the parameter $s_0$ of Assumption~\ref{assumption:shift-theorem} may be taken to be any
$s_0 < 1/2$. Nevertheless, one expects from Theorem~\ref{thm:optimal-L2} 
a convergence rate close to $\min\{2,1+\alpha-1/2\}$. For $\alpha = 0.75$ this is $1.25$, which is  visible in Table~\ref{tab:slit-shape1} for the case $(x_0,y_0) = (0,0)$. 
Again, the better convergence behavior for $(x_0,y_0) = (0.5,0)$ and 
$(x_0,y_0) = (0,1)$ can be explained by the theory of 
Corollary~\ref{cor:2D-L2-concrete}, which predicts $1+\alpha = 1.75$. 
Table~\ref{tab:slit-shape2} shows the results for $(x_0,y_0) =(0,0) $ and
$\alpha \in \{10/9,4/3,3/2\}$. From  Theorem~\ref{thm:optimal-L2}, we
expect convergence rates of $ 1.61$, $1.83$ and $2$, respectively. The
observed numerical rates of $ 1.65$, $1.86$, and $1.96$ are reasonably close to
these predictions.

\begin{table}[ht]
\begin{center}\begin{tabular}{|r||c|c||c|c||c|c|}
\hline
&\multicolumn{2}{|c||}{$(x_0,y_0) = (0,0)$}
&\multicolumn{2}{|c||}{$(x_0,y_0) = (0.5,0)$}
&\multicolumn{2}{|c|}{$(x_0,y_0) = (0,1)$}  \\
&\multicolumn{2}{|c||}{$a = \pi/2$}
&\multicolumn{2}{|c||}{$a=\pi$}
&\multicolumn{2}{|c|}{$a=\pi$}  \\
\hline
DOFs&$L^2$-error  & rate &$L^2$-error  & rate & $L^2$-error & rate  \\ \hline
97     		&6.1391e-03  &  -         &    1.1088e-02& -          & 1.0692e-02       & -        \\
348    		&2.8187e-03  &  1.12   &    4.1329e-03&  1.42   &      3.8553e-03  & 1.47   \\
1.315   		&1.2351e-03  &  1.19   &    1.4164e-03&  1.54   &      1.3388e-03  & 1.53   \\
5.109   		&5.3338e-04  &  1.21   &    4.7830e-04&  1.57   &      4.4562e-04  & 1.59   \\
20.137  	&2.2846e-04  &  1.22   &    1.4725e-04&  1.70   &      1.4420e-04  & 1.63   \\
79.953  	&9.7267e-05  &  1.23   &    4.6683e-05&  1.66   &      4.5843e-05  & 1.65   \\
318.625 	&4.1233e-05  &  1.24   &    1.4761e-05&  1.66   &      1.4401e-05  & 1.67   \\
1.272.129	&1.7428e-05  &  1.24   &    4.3773e-06&  1.75   &      4.4861e-06  & 1.68   \\
5.083.777	&7.3524e-06  &  1.25   &    1.3285e-06&  1.72   &      1.3889e-06  & 1.69   \\
  \hline
\end{tabular}
   \caption{Slit domain, $k=1$: Influence of the position of singularity for $\alpha = 0.75$.}
\label{tab:slit-shape1}
\end{center}
\end{table}

\begin{table}[ht]
\begin{center}\begin{tabular}{|r||c|c||c|c||c|c|}
\hline
&\multicolumn{2}{|c||}{$\alpha = 10/9$}
&\multicolumn{2}{|c||}{$\alpha =4/3$}
&\multicolumn{2}{|c|}{$\alpha = 3/2$}  \\
\hline
DOFs		&$L^2$-error  & rate &$L^2$-error  & rate & $L^2$-error & rate  \\ \hline
97     		&5.7534e-03  &  -       		&  7.3549e-03  &  -     	  	&  7.5901e-03  &  -     \\
348    		&1.9412e-03  &  1.57    	&  2.2414e-03  &  1.71    	&  2.1664e-03  & 1.81  \\
1.315   		&6.2583e-04  &  1.63    	&  6.4849e-04  &  1.79    	&  5.8638e-04  & 1.89  \\
5.109   		&1.9689e-04  &  1.67    	&  1.8251e-04  &  1.83    	& 	 1.5450e-04  & 1.92  \\
20.137  	&6.1446e-05  &  1.68    	&  5.0718e-05  &  1.85    	&  4.0197e-05  & 1.94  \\
79.953  	&1.9191e-05  &  1.68    	&  1.4021e-05  &  1.85    	& 	 1.0396e-05  & 1.95  \\
318.625 	&6.0229e-06  &  1.67    	&  3.8699e-06  &  1.86    	&  2.6803e-06  & 1.96  \\
1.272.129	&1.9023e-06  &  1.66   	&  1.0682e-06  &  1.86    	& 	 6.8978e-07  & 1.96  \\
5.083.777	&6.0474e-07  &  1.65  	&  2.9514e-07  &  1.86 		&   1.7730e-07  & 1.96  \\
  \hline
\end{tabular}
   \caption{Slit domain, $k=1$: Influence of  exponent $\alpha $ for
     $a=1/2 \pi$ and  $(x_0,y_0) =(0,0)$}
\label{tab:slit-shape2}
\end{center}
\end{table}
 \subsubsection{Second order finite elements}
In this subsection, we test the performance of quadratic finite
elements for the L-shaped domain. We use the same type of solution as
before and vary the parameter $\alpha$ for $(x_0,y_0) = (0,0)$, i.e., the
re-entrant corner. Here we expect from our theory a
convergence rate of $\min(3 , \alpha+1-1/3)$. For $\alpha \in \{2.175, 2.275, 2.375\}$, the observed numerical rates, which are visible in Table \ref{tab:quadratic}, are very close to the theoretically predicted ones.

\begin{table}[ht]
\begin{center}\begin{tabular}{|r||c|c||c|c||c|c|}
\hline
&\multicolumn{2}{|c|}{$\alpha = 2.175$}
&\multicolumn{2}{|c|}{$\alpha =2.275$}
&\multicolumn{2}{|c|}{$\alpha = 2.375$}  \\
\hline
DOFs		&L2 error  & rate &L2 error  & rate & L2 error  & rate  \\ \hline
289   		&2.7565e-04  & -         &		    2.4570e-04 & -  			&  2.2177e-04    & -   	    \\
1.089  		&5.1121e-05  &2.43    &         4.1696e-05  &2.56      &  3.3912e-05    &  2.71\\
4.225  		&7.5320e-06  &2.76    &         5.7319e-06  &2.86      & 4.3221e-06     &  2.97\\
16.641 		&1.1051e-06  &2.77    &         7.8407e-07  &2.87      &  5.4888e-07    &  2.98\\
66.049 		&1.5938e-07  &2.79    &         1.0553e-07  &2.89      &  6.8762e-08    &  3.00\\
263.169	&2.2723e-08  &2.81    &         1.4044e-08  &2.91      &  8.5292e-09    &  3.01\\
1.050.625	&3.2138e-09  &2.82    &         1.8538e-09  &2.92      &  1.0497e-09    &  3.02\\
\hline
\end{tabular}
  \caption{L-shaped domain, $k=2$: Influence of  $\alpha $ for
     $a=2/3 \pi$ and $(x_0,y_0) =(0,0)$.}
\label{tab:quadratic}
\end{center}
\end{table}

\subsection{Three-dimensional results}
In the three dimensional setting, we consider a Fichera corner 
$\Omega := (-1,1)^3 \setminus [0,1]^3$ and prescribe 
the smooth solution $u(x,y,z) := \sin((x+y)\pi) \cos(2\pi z)$. The inhomogeneous 
Dirichlet conditions are realized by nodal interpolation. The discretization
is based on trilinear finite elements on hexahedra and uniform refinements. 
Although the dual problem lacks full regularity, Theorem~\ref{thm:optimal-L2} 
asserts that this can be compensated by extra $s_0$ regularity of the 
primal solution to maintain full second order convergence in $L^2$. 

\begin{table}[ht]
\begin{center}
\begin{tabular}{|r|c|c|}
\hline
DOF & $L^2$-error & rate\\ \hline 
316  &   0.075444 &    - \\
3.032  &   0.017182   &  1.96 \\
26.416  &   0.0039376  &   2.04 \\
220.256  &   0.00094597 &    2.02 \\
1.798.336   &  0.00023208  &   2.01 \\
14.532.992  &   5.7491e-05   &  2.00 \\ \hline
\end{tabular}
\caption{
\label{tab:fichera}
Fichera corner, $k=1$: $L^2$-error for a smooth solution.}
\end{center}
\end{table}

Table~\ref{tab:fichera} shows that we observe numerically already for 
coarse discretizations the predicted convergence order two, and the theoretical
results are confirmed. 

\iftechreport
\medskip
We point the reader to Appendix~\ref{sec:detailed-numerics} for 
a further numerical results. 
\fi


\appendix
\normalsize
\include{appendix_besicovitch}
\iftechreport
\include{appendix}
\include{appendix_lemma.5.4}

\include{appendix_wahlbin}
\include{detailed_numerics}
\fi
\bibliography{nummech}
\bibliographystyle{spmpsci}      

\end{document}

%% file: appendix_besicovitch.tex
\section{Coverings} 
\label{sec:besicovitch}
In this appendix, the distance $\operatorname*{dist}(x,M)$ for 
some set $M$ appears frequently. For notational convenience, 
we set $\operatorname*{dist}(x,\emptyset) = 1$ to include the 
degenerate case $M = \emptyset$. 

We quote from \cite[Lemma~{A.1}]{melenk-wohlmuth12}:
\begin{lemma}
\label{lemma:besicovitch-balls}
Let $\Omega\subset\BbbR^d$ be bounded open 
and $M=\overline{M} $ be a closed set. 
Fix $c \in (0,1)$ and $\varepsilon \in (0,1)$ such that 
\begin{equation}
\label{eq:c-varepsilon-condition}
1 - c(1+ \varepsilon) =: c_0 > 0. 
\end{equation}
For each $x \in \Omega$, let $B_x:= \overline{B}_{c\operatorname*{dist}(x,M)}(x)$
be the closed ball of radius $c \operatorname*{dist}(x,M)$ centered at $x$, 
and let 
$\widehat B_x:= 
\overline{B}_{(1+\varepsilon) c \operatorname*{dist}(x,M)}(x)$ denote 
the stretched (closed) ball of radius $(1+\varepsilon)c \operatorname*{dist}(x,M)$ also centered at $x$. 

Then there exists a countable set $x_i \in \Omega$, $i\in\BbbN$, 
and a constant $N \in \BbbN$ depending solely on the spatial dimension 
$d$ with the following properties: 
\begin{enumerate}
\item 
(covering property) $\cup_{i \in \BbbN} B_{x_i} \supset \Omega$; 
\item
(finite overlap on $\Omega$) for each $x \in \Omega$, there holds 
$\operatorname*{card} \{i\,|\, x \in \widehat B_{x_i} \} \leq N$. 
\end{enumerate}
\end{lemma}
\begin{proof}
\cite[Lemma~{A.1}]{melenk-wohlmuth12} assumed that $M \subset \overline{\Omega}$. However, 
an inspection of the proof shows that this is not necessary. 
\end{proof}
Before we proceed with variants of the covering result
of Lemma~\ref{lemma:besicovitch-balls}, we introduce the notation
of sectorial neighborhoods relative a singular set $M$: 
\begin{definition}[sectorial neighborhood] 
Let $e$, $M \subset \BbbR^d$  and $\widetilde c > 0$. Then 
$$
S_{e,M,\widetilde c}:= \cup_{x \in e} B_{\widetilde c \operatorname*{dist}(x,M)}(x) 
$$
is a {\em sectorial neighborhood of the set $e$ relative 
to the singular set $M$}. 
\end{definition} 
We are interested in coverings of lower-dimensional
manifolds by balls whose centers are located on these 
manifolds: 
\begin{lemma} 
\label{lemma:besicovitch-centers-on-lower-dimensional-manifold}
Let $d \in \BbbN$ and $1 \leq d^\prime < d$. Let 
$\omega \subset \BbbR^{d^\prime}$ and let $\Omega \subset \BbbR^d$
be the canonical embedding of $\omega$ into $\BbbR^d$, i.e., 
$\Omega:= \omega \times \{0\} \times \cdots \times \{0\} \subset \BbbR^d$. 
Assume the hypotheses and notation of Lemma~\ref{lemma:besicovitch-balls}. 
Then there are $\widetilde c> 0$, $N > 0$, and a collection of 
balls $B_{x_i}$, $i \in \BbbN$, as described in 
Lemma~\ref{lemma:besicovitch-balls} such that 
\begin{enumerate}[(i)]
\item 
\label{item:lemma:besicovitch-centers-on-lower-dimensional-manifold-i}
(covering property for $\Omega$) 
$\cup_{i \in \BbbN} B_{x_i} \supset \Omega$. 
\item 
\label{item:lemma:besicovitch-centers-on-lower-dimensional-manifold-ii}
(covering property for a sectorial neighborhood of $\Omega$) 
$\cup_{i \in \BbbN} B_{x_i} \supset S_{\Omega,M,\widetilde c}$. 
\item 
\label{item:lemma:besicovitch-centers-on-lower-dimensional-manifold-iii}
(finite overlap property on $\BbbR^d$) 
for each $x \in \BbbR^d$, there holds 
$\operatorname*{card} \{i\,|\, x \in \widehat B_{x_i} \} \leq N$. 
\end{enumerate}
\end{lemma}
\begin{proof}
We employ the result of Lemma~\ref{lemma:besicovitch-balls} for the lower-dimensional
manifold $\omega$ noting that $B_{x} \cap \omega$ is a ball in $\BbbR^{d^\prime}$. 
In order to 
be able to ensure the covering condition for the sectorial neighborhood 
of $\Omega$ stated in 
(\ref{item:lemma:besicovitch-centers-on-lower-dimensional-manifold-iii}), 
we introduce the auxiliary 
balls $B^\prime_{x}:= \overline{B}_{c/2 \operatorname*{dist}(x,M)}(x)$ of half the radius. 
Applying Lemma~\ref{lemma:besicovitch-balls} with these balls $B^\prime_{x}$ and the stretched
balls $\widehat B_x$ therefore produces a collection of centers $x_i \in \Omega$, 
$i \in \BbbN$, such that 
\begin{enumerate}
\item 
$B^\prime_{x_i} \cap \Omega$ covers $\Omega$; 
\item 
for the stretched balls $\widehat B_{x_i}$, we have a finite overlap property  on $\Omega$: 
\begin{equation}
\label{eq:lemma:besicovitch-semiballs-50}
\forall x \in \Omega: \quad 
\operatorname*{card}\{i\,|\, x \in \widehat B_{x_i} \} \leq N. 
\end{equation}
\end{enumerate}
We next see that the balls $\widehat B_{x_i}$ even have the following, stronger 
finite overlap property: 
\begin{equation}
\label{eq:lemma:besicovitch-semiballs-100}
\forall x \in \BbbR^d: \quad 
\operatorname*{card}\{i\,|\, x \in \widehat B_{x_i} \} \leq N. 
\end{equation}
To see this, define the infinite cylinders 
$\widehat C_{x_i}:= \{x\,|\, \pi_{d^\prime}(x) \in \widehat B_{x_i}\cap \Omega\}$, where
$\pi_{d^\prime}$ is the canonical projection onto the 
hyperplane $\{x = (x_1,\ldots,x_d) \in \BbbR^d\,|\, x_{d^\prime+1} = \cdots = x_d = 0\}$. 
Clearly, $\widehat B_{x_i} \subset \widehat C_{x_i}$. These infinite cylinders have a finite
overlap property by (\ref{eq:lemma:besicovitch-semiballs-50}) as can be seen by writing
any $x \in \BbbR^d$ in the form $x = (\pi_{d^\prime}(x),x^{\prime})$ for some $x^\prime \in \BbbR^{d - d^\prime}$
and then noting that $x \in \widehat C_{x_i}$ implies $\pi_{d^\prime}(x) \in \widehat B_{x_i} \cap \Omega$. 

Is remains to see that the balls $B_{x_i}$ cover a sectorial neighborhood of $\Omega$. 
To that end, we note that the balls $B^\prime_{x_i}$ cover $\Omega$. Furthermore, 
for  each $x \in \Omega$, we pick $x_i$ such that $x \in B_{x_i}^\prime \subset B_{x_i}$. 
Since the radius of $B_{x_i}$ is twice that of $B^\prime_{x_i}$, we even have 
$B_{c/2 \operatorname*{dist}(x_i,M)}(x) \subset B_{x_i}$. Furthermore, by $c \in (0,1)$, 
we have $0 < (1 - c/2)\operatorname*{dist}(x_i,M) \leq \operatorname*{dist}(x,M) 
\leq (1 + c/2) \operatorname*{dist}(x_i,M)$. Therefore, there is $\widetilde c > 0$ such that 
$B_{\widetilde c \operatorname*{dist}(x,M)}(x) \subset B_{x_i}$ and thus 
$$
\cup_{x \in \Omega} B_{\widetilde c \operatorname*{dist}(x,M)}(x) \subset \cup_i B_{x_i}.  
$$ 
\qed
\end{proof}
We next show covering theorems for polygons and polyhedra. In the interest of 
clarity of presentation, we formulate two separate results. Before doing so, we 
point out that balls with center located on the boundary of the polygon/polyhedron
$\Omega$ will feature importantly so that the intersection of this ball with $\Omega$
will be of interest. We therefore introduce the following notions: 
\begin{definition}[solid angles and dihedral angles]
\label{def:solid-angle}
\begin{enumerate}
\item 
Let $\Omega \subset \BbbR^2$ be a Lipschitz polygon. Let $A$ be a vertex where
the edges $e_1$, $e_2$ meet. We say that the set $B_{\varepsilon}(A) \cap \Omega$
is a solid angle, if $\partial (B_{\varepsilon}(A) \cap \Omega) \cap \partial\Omega$ 
is contained in $\{A\} \cup e_1 \cup e_2$. 
\item 
Let $\Omega \subset \BbbR^3$ be a Lipschitz polyhedron. Let $A$ be a vertex of $\Omega$. 
We say that the set $B_\varepsilon(A) \cap \Omega$ is a {\em solid angle}, if 
$\partial (B_\varepsilon(A) \cap \Omega) \cap \partial \Omega$ is contained in
the union of $\{A\}$ and the edges and faces meeting at $A$. 
\item 
Let $\Omega \subset \BbbR^3$ be a Lipschitz polyhedron. Let $e$ be an edge of $\Omega$, 
which is shared by the faces $f_1$, $f_2$. 
Let $x \in e$. We say that the set $B_\varepsilon(x) \cap \Omega$ is a {\em dihedral angle}, if 
$\partial (B_\varepsilon(x) \cap \Omega) \cap \partial \Omega$ is contained 
in $e \cup f_1 \cup f_2$. 
\end{enumerate}
\end{definition}
\begin{theorem}
\label{thm:covering-2D}
Let $\Omega \subset \BbbR^2$ be a bounded Lipschitz polygon with 
vertices $A_j$, $j=1,\ldots,J$, and edges ${\mathcal E}$. Let $M \subset \{A_1,\ldots, A_J\}$.  
Set ${\mathcal A}^\prime:=\{A_1,\ldots,A_J\} \setminus M$ and fix $\varepsilon \in (0,1)$. 
\begin{enumerate}[(i)] 
\item 
\label{item:thm:covering-2D-i}
There is a sectorial neighborhood  
$S_{{\mathcal A}^\prime,M,\widetilde c}$ of the  vertices ${\mathcal A}^\prime$ 
and a constant $c \in (0,1)$ such that 
$S_{{\mathcal A}^\prime,M,\widetilde c}$ 
is covered by balls $B_i:= \overline{B}_{c \operatorname*{dist}(x_i,M)}(x_i)$ with 
centers $x_i \in {\mathcal A}^\prime$. Furthermore, the stretched
balls 
$\widehat B_i:= \overline{B}_{(1+\varepsilon) c \operatorname*{dist}(x_i,M)}(x_i)$ are solid
angles and satisfy a finite overlap property on $\BbbR^2$. 
\item 
\label{item:thm:covering-2D-ii}
Fix a sectorial neighborhood ${\mathcal U} := S_{{\mathcal A}^\prime,M,c^\prime}$ of the
vertices ${\mathcal A}^\prime$. 
For each edge $e \in {\mathcal E}$, there is a sectorial neighborhood  
$S_{e,M,\widetilde c}$ and a constant $c \in (0,1)$ such that 
$S_{e,M,\widetilde c} \setminus {\mathcal U}$ is covered
by balls $B_i = \overline{B}_{c \operatorname*{dist}(x_i,M)}(x_i)$ whose centers $x_i$ are located on $e$. 
Furthermore, the stretched balls $\widehat B_i = \overline{B}_{(1+\varepsilon)c \operatorname*{dist}(x_i,M)}(x_i) $ 
satisfy a finite overlap property on $\BbbR^2$ and are such that each 
$\widehat B_i \cap \Omega$ is a half-disk. 
\item 
\label{item:thm:covering-2D-iii}
Fix a sectorial neighbood ${\mathcal U}:= S_{{\mathcal E},M,c^\prime}$ of the edges ${\mathcal E}$. There 
is $c \in (0,1)$ such that $\Omega\setminus {\mathcal U}$ is covered by balls 
$B_i = \overline{B}_{c \operatorname*{dist}(x_i,M)}(x_i)$ such that 
the stretched balls $\widehat B_i = \overline{B}_{(1+\varepsilon) c \operatorname*{dist}(x_i,M)}(x_i)$ 
are completely contained in $\Omega$ and satisfy a finite overlap property on $\BbbR^2$. 
\end{enumerate}
\end{theorem}
\begin{proof}
The assertion (\ref{item:thm:covering-2D-i}) is almost trivial and only included to emphasize the structure
of the arguments. Assertions (\ref{item:thm:covering-2D-ii}), (\ref{item:thm:covering-2D-iii}) 
follow from suitable applications of Lemmas~\ref{lemma:besicovitch-centers-on-lower-dimensional-manifold}
and \ref{lemma:besicovitch-balls}.
\qed
\end{proof}
The 3D variant of Theorem~\ref{thm:covering-2D} is formulated in Theorem~\ref{thm:covering-3D}. We emphasize
that the ``singular'' set $M$ need not be the union of {\em all} edges and vertices but can be just a subset. 
We also emphasize that it is not necessarily related to the notion of ``singular set'' in 
Definition~\ref{def:H2-regular-edges}, although it is used in this way. 
The key property of the covering balls is again such that the centers are either a) in $\Omega$ (in which 
case the stretched ball is contained in $\Omega$); or b) on a face (in which case the stretched ball $\widehat B_i$
is such that $\widehat B_i \cap \Omega$ is a half-ball);  or c) on an edge in which case 
$\widehat B_i \cap \Omega$ is a dihedral angle (see Definition~\ref{def:solid-angle}); or d) 
in a vertex in which case $\widehat B_i \cap \Omega$ is a solid angle (see Definition~\ref{def:solid-angle}).
\begin{theorem}
\label{thm:covering-3D}
Let $\Omega \subset\BbbR^3$ be a Lipschitz polyhedron with faces ${\mathcal F}$, edges ${\mathcal E}$, 
and vertices ${\mathcal A}$. Let $M_{\mathcal A} \subset {\mathcal A}$ 
and $M_{\mathcal E} \subset {\mathcal E}$. Let $M = \overline{M} = \overline {M_{\mathcal A} \cup M_{\mathcal E}}$
and fix $\varepsilon \in (0,1)$. 
Let ${\mathcal A}^\prime:=\{ A \in {\mathcal A}\,|\, A \not\in M\}$ be the vertices not in $M$ 
and 
${\mathcal E}^\prime:=\{e \in {\mathcal E}\,|\, \overline{e} \cap M = \emptyset\}$ be the 
edges not abutting $M$. Then:
\begin{enumerate}[(i)] 
\item (non-singular vertices)
\label{item:thm:covering-3D-i}
There is a sectorial neighborhood  
$S_{{\mathcal A}^\prime,M,\widetilde c}$ of the vertices in ${\mathcal A}^\prime$ 
and a constant $c \in (0,1)$ such that 
$S_{{\mathcal A}^\prime,M,\widetilde c}$ 
is covered by balls $B_i:= \overline{B}_{c \operatorname*{dist}(x_i,M)}(x_i)$ with 
centers $x_i \in {\mathcal A}^\prime$. Furthermore, the stretched
balls 
$\widehat B_i:= \overline{B}_{(1+\varepsilon) c \operatorname*{dist}(x_i,M)}(x_i)$ are solid angles 
and satisfy a finite overlap property on $\BbbR^3$. 
\item (non-singular edges)
\label{item:thm:covering-3D-ii}
Fix a sectorial neighborhood ${\mathcal U} := S_{{\mathcal A}^\prime,M,c^\prime}$ of ${\mathcal A}^\prime$. 
For each edge $e \in {\mathcal E}^\prime$, there is a sectorial neighborhood  
$S_{e,M,\widetilde c}$ and a constant $c \in (0,1)$ such that 
$S_{e,M,\widetilde c} \setminus {\mathcal U}$ is covered
by balls $B_i = \overline{B}_{c \operatorname*{dist}(x_i,M)}(x_i)$ whose centers $x_i$ are located on $e$. 
Furthermore, the stretched balls $\widehat B_i = \overline{B}_{(1+\varepsilon)c \operatorname*{dist}(x_i,M)}(x_i) $ 
satisfy a finite overlap property on $\BbbR^3$ 
and $\widehat B_i \cap \Omega$ is a dihedral angle. 
\item (faces)
\label{item:thm:covering-3D-iii}
Fix a sectorial neighbood ${\mathcal U}:= S_{{\mathcal E},M,c^\prime}$ of ${\mathcal E}$. 
There is a sectorial neighborhood $S_{{\mathcal F},M,\widetilde c}$ 
and a constant $c \in (0,1)$ such that $S_{{\mathcal F},M,\widetilde c}\setminus {\mathcal U}$ 
is covered by balls $B_i = \overline{B}_{c \operatorname*{dist}(x_i,M)}(x_i)$ 
with centers $x_i \in \partial\Omega$. 
Furthermore, the stretched balls $\widehat B_i = \overline{B}_{(1+\varepsilon) c \operatorname*{dist}(x_i,M)}(x_i)$ 
satisfy a finite overlap property on $\BbbR^3$ and $\widehat B_i \cap \Omega$ is a half-ball. 
\item (interior)
\label{item:thm:covering-3D-iv}
Fix a sectorial neighbood ${\mathcal U}:= S_{{\mathcal F},M,c^\prime}$ of ${\mathcal F}$, where 
${\mathcal F}$ is the set of faces. 
Then there is $c \in (0,1)$ such that $\Omega\setminus {\mathcal U}$ 
is covered by balls $B_i = \overline{B}_{c \operatorname*{dist}(x_i,M)}(x_i)$ 
with centers $x_i \in \Omega$. 
Furthermore, the stretched balls $\widehat B_i = \overline{B}_{(1+\varepsilon) c \operatorname*{dist}(x_i,M)}(x_i)$ 
satisfy a finite overlap property on $\BbbR^3$ and $\widehat B_i \subset \Omega$.  
\end{enumerate}
\end{theorem}
\begin{proof}
Follows from Lemmas~\ref{lemma:besicovitch-centers-on-lower-dimensional-manifold}
and \ref{lemma:besicovitch-balls}.
\qed 
\end{proof}

%% file: appendix.tex
\iftechreport
\renewcommand\thesection{\Alph{section}}
\section{Details for the extension of the results of \protect{\cite{melenk-wohlmuth12}}.} 
\label{app:melenk-wohlmuth12}
The arguments used in the proof of Theorem~\ref{thm:error-on-strip}
rely on techniques developed in \cite{melenk-wohlmuth12}. A
feature of the analysis here that was not present in \cite{melenk-wohlmuth12} is
the bidual problem with right-hand side $\widetilde\delta_\Gamma^{-1} (w - w_h)$ that 
allowed us to estimate $w - w_h$  in a weighted space. This technical issue was circumvented 
in \cite{melenk-wohlmuth12} by assuming convexity of $\Omega$ so that 
optimal order $L^2$-estimates could be cited from the literature. The corresponding 
bidual problem can be analyzed in the mortar setting of \cite{melenk-wohlmuth12} as well. 
The end result is then Theorem~\ref{thm:generalize-melenk-wohlmuth12}, which states
that the convexity assumption in \cite{melenk-wohlmuth12} can be relaxed to the 
validity of Assumption~\ref{assumption:shift-theorem} for the Poisson problem, 
i.e., \cite[(5.2)]{melenk-wohlmuth12}. 

In the interest of brevity, we employ in this appendix the notation of 
\cite{melenk-wohlmuth12} and assume the reader's familiarity with 
\cite{melenk-wohlmuth12}. 

It will be useful to write $H^s_{pw}(\Gamma)$ for 
the space given by the broken Sobolev norm on the skeleton $\Gamma$, i.e., 
the Sobolev norm is understood facewise. Furthermore, we will write 
$\|\cdot\|_{H^1}$ for the broken $H^1$-norm, i.e., 
$\|\cdot\|^2_{H^1} = \sum_i \|\cdot\|^2_{H^1(\Omega_i)}$. 
We also introduce the $L^2$-projection 
$\Pi^{L^2}_{M_h}:L^2(\Gamma) \rightarrow M_h$ and recall that, since 
$M_h$ is a product space based on the faces, it inherits from 
\cite[(A2)]{melenk-wohlmuth12} the approximation  
property 
\begin{equation}
\label{eq:approximation-property-L2-projection}
\|z - \Pi^{L^2}_{M_h} z\|_{L^2(\Gamma)} \leq C h^s \|z\|_{H^s_{pw}(\Gamma)}, 
\qquad s \in [0,k]. 
\end{equation}

The main result is: 
\begin{theorem}
\label{thm:generalize-melenk-wohlmuth12}
In \cite[Thms.~{2.1}, {2.5}]{melenk-wohlmuth12}, the assumption
of convexity of $\Omega$ can be replaced with 
\cite[Assumption~{(5.2)}]{melenk-wohlmuth12}. 
\end{theorem}
\begin{proof}
We will only sketch the modifications entailed by the weakenend regularity assumptions.

{\bf Proof of \cite[Thm.~{2.1}]{melenk-wohlmuth12}:} The starting point 
is the error representation \cite[(6.2)]{melenk-wohlmuth12}, which consists 
of three terms:  
\begin{equation}
\label{eq:thm:2.1-10}
a(w - w_h,u - P_h u) + b(w - w_h,\lambda - \mu_h) + b(u - P_h u,\lambda_w - \tilde \mu_h), 
\end{equation}
where $\mu_h$, $\tilde \mu_h \in M_h$ are arbitrary. 
The first term in (\ref{eq:thm:2.1-10}) can be estimated as in \cite[Proof of Thm.~{2.1}]{melenk-wohlmuth12} 
in view of the generalization of \cite[Lemma~{5.5}]{melenk-wohlmuth12} given below
as Lemma~\ref{lemma:lemma:5.5}. The third term in (\ref{eq:thm:2.1-10}) can again be estimated
as in \cite[(6.4)]{melenk-wohlmuth12} since Lemma~\ref{lemma:approximation-of-w-and-lambda_w}
below provides the estimate $\inf_{\tilde \mu_h \in M_h} \|\lambda_w - \tilde \mu_h\|_{L^2(\Gamma)} 
\lesssim h^{1/2} \|v\|_{L^2(\Omega)}$. The second term in (\ref{eq:thm:2.1-10}) requires a modification
of the procedure in \cite[(6.3)]{melenk-wohlmuth12}. Taking 
$\mu_h = \Pi^{L^2}_{M_h} \lambda$ in \cite[(6.3)]{melenk-wohlmuth12} yields 
\begin{eqnarray*}
b(w - w_h,\lambda - \Pi^{L^2}_{M_h} \lambda) &=& 
\int_\Gamma [w - w_h] (\lambda - \Pi^{L^2}_{M_h} \lambda) \\
&= &
\int_\Gamma ([w - w_h]  - \Pi^{L^2}_{M_h} [w - w_h]) (\lambda - \Pi^{L^2}_{M_h} \lambda)  \\
& \lesssim & h^{1/2} \|[w - w_h]\|_{H^{1/2}_{pw}(\Gamma)} \|\lambda- \Pi^{L^2}_{M_h} \|_{L^2(\Gamma)}
\\
&\lesssim & h^{1/2} \|w - w_h\|_{H^1} \|\lambda - \Pi^{L^2}_{M_h} \lambda\|_{L^2(\Gamma)}  \\
&\lesssim & h^{1/2+k} \|w - w_h\|_{H^1} \|\lambda\|_{H^k_{pw}(\Gamma)}; 
\end{eqnarray*}
the last step followed from (\ref{eq:approximation-property-L2-projection}). 
The proof is completed with the aid of Lemma~\ref{lemma:approximation-of-w-and-lambda_w} and 
the trace estimate $\|\lambda\|_{H^k_{pw}(\Gamma)} \lesssim \|u\|_{B^{k+3/2}_{2,1}(\Omega)}$. 

{\bf Proof of \cite[Thm.~{2.5}]{melenk-wohlmuth12}:} The proof stands as given in \cite{melenk-wohlmuth12}.
\qed
\end{proof}
 
\begin{lemma}
\label{lemma:approximation-of-w-and-lambda_w}
Assume that $\Omega$ satisfies \cite[Assumption~{(5.2)}]{melenk-wohlmuth12}. 
Then, for $v \in L^2(S_h) \subset L^2(\Omega)$ and $w:= T^D(v)$ and the corresponding
Lagrange multiplier $\lambda_w$ defined facewise by $\lambda_w|_{\gamma_l} = - \partial_n w|_{\Omega_{s(l)}}$
and the corresponding mortar approximation $w_h$ of $w$ there holds:  
\begin{eqnarray}
\label{eq:lemma:5.5-100}
\sqrt{\sum_{i} \|w - w_h\|^2_{H^1(\Omega_i)}} = \|w - w_h\|_{H^1}  &\lesssim & h \|v\|_{L^2(\Omega)}, \\
\label{eq:lemma:5.5-101}
\|\lambda_w - \Pi^{L^2}_{M_h} \lambda_w \|_{L^2(\Gamma)} &\lesssim & h^{1/2} \|v\|_{L^2(\Omega)}. 
\end{eqnarray}
\end{lemma}
\begin{proof}
We start with the proof of (\ref{eq:lemma:5.5-100}). 
It results from standard convergence theory for mortar methods 
as follows. \cite[Assumption~{(5.2)}]{melenk-wohlmuth12} 
provides $w \in H^{3/2+\varepsilon}(\Omega)$ for some $\varepsilon > 0$
together with $\|w\|_{H^{3/2+\varepsilon}(\Omega)} \leq C \|v\|_{H^{-1/2+\varepsilon}(\Omega)}$. 
The Lagrange multiplier $\lambda_w$ is given facewise by the expression 
$\lambda_w|_{\gamma_l} = - \partial_{n} w|_{\Omega_{s(l)}} \in H^{\varepsilon}(\gamma_l)$ 
together with the estimate $\|\lambda_w\|_{H^\varepsilon_{pw}(\Gamma)} \lesssim \|v\|_{H^{-1/2+\varepsilon}(\Omega)}$. 
The standard convergence theory for mortar methods 
(as worked out, e.g., in \cite[Prop.~{2.3}]{flemisch-melenk-wohlmuth05}) 
then gives 
$$
\|w - w_h\|_{H^1} \lesssim 
h^{1/2+\varepsilon} \|w\|_{H^{3/2+\varepsilon}(\Omega)}
\lesssim h^{1/2+\varepsilon} \|v\|_{H^{-1/2+\varepsilon}(\Omega)}. 
$$
The proof of (\ref{eq:lemma:5.5-100}) is complete if we can show that 
\begin{equation}
\label{eq:lemma:5.5-120}
\|v\|_{H^{-1/2+\varepsilon}(\Omega)} \lesssim h^{1/2-\varepsilon} 
\|v\|_{L^2(\Omega)}. 
\end{equation}
This last estimate exploits $\supp v \subset S_h$ and 
follows by interpolation
arguments similar to those employed in the proof of
\cite[Lemma~{5.2}]{melenk-wohlmuth12}: 
Define $\theta = 1-2 \varepsilon$ (we assume $\varepsilon < 1/2$). 
Then $H^{-1/2+\varepsilon}(\Omega) = (H^{1/2 - \varepsilon}(\Omega))^\prime = 
\Bigl(\bigl(L^2(\Omega),B^{1/2}_{2,1}(\Omega)\bigr)_{\theta,2}\Bigr)^\prime
= \bigl((B^{1/2}_{2,1}(\Omega))^\prime, L^2(\Omega)\bigr)_{\theta,2}$, so that the interpolation 
inequality yields 
$\|v\|_{H^{-1/2+\varepsilon}(\Omega)} \lesssim 
\|v\|^\theta_{(B^{1/2}_{2,1}(\Omega))^\prime} \|v\|_{L^2(\Omega)}^{1-\theta}$.
The argument is completed by noting in view of $\supp v \subset S_h$ that 
\begin{align*}
\|v\|_{L^2(\Omega)} \leq \|v\|_{L^2(\Omega)} 
\qquad \mbox{ and } \qquad   
\|v\|_{(B^{1/2}_{2,1}(\Omega))^\prime} \leq C \sqrt{h} \|v\|_{L^2(\Omega)},  
\end{align*}
so that $\|v\|_{H^{-1/2+\varepsilon}(\Omega)} \lesssim h^{(1-\theta)/2} \|v\|_{L^2(\Omega)} 
\lesssim h^{1/2-\varepsilon} \|v\|_{L^2(\Omega)}$. 
This proves (\ref{eq:lemma:5.5-120}). 

The bound (\ref{eq:lemma:5.5-101}) follows from 
(\ref{eq:approximation-property-L2-projection}) and 
$
\|\lambda_w  - \Pi^{L^2}_{M_h} \lambda_w\|_{L^2(\Gamma)} \leq C h^\varepsilon \|\lambda_w\|_{H^\varepsilon_{pw}(\Gamma)} 
\lesssim h^\varepsilon \|w\|_{H^{3/2+\varepsilon}(\Omega)} \lesssim h^\varepsilon \|v\|_{H^{-1/2+\varepsilon}(\Omega)}. 
$
An appeal to (\ref{eq:lemma:5.5-120}) finishes the proof.
\qed
\end{proof}

We generalize \cite[Lemma~{5.5}]{melenk-wohlmuth12}:  
\begin{lemma}
\label{lemma:lemma:5.5} 
[generalizations of \protect{\cite[Lemma~{5.5}]{melenk-wohlmuth12}}]
Assume that $\Omega$ satisfies \cite[Assumption~{(5.2)}]{melenk-wohlmuth12}. 
Then, for $v \in L^2(S_h) \subset L^2(\Omega)$ and $w:= T^D(v)$ and
the mortar approximation $w_h$ of $w$, there holds 
$$
\|\nabla (w- w_h)\|_{L^2(\Gamma; L^1)} \leq C h^{3/2} ( 1 + \delta_{k,1} |\ln h|) \|v\|_{L^2(\Omega)}. 
$$
\end{lemma}
\begin{proof}
Inspection of the proof of \cite[Lemma~{5.5}]{melenk-wohlmuth12} shows that we have 
to estimate the following two terms: 
$$
T_1:= \int_{\tau = 0}^{\tilde c h} \|\nabla (w - w_h)\|_{L^2(\gamma_\tau)} 
\qquad \mbox{ and } \qquad 
T_2:= \int_{\tau =  \tilde c h}^D \|\nabla (w - w_h)\|_{L^2(\gamma_\tau)}. 
$$
{\bf Estimating $\mathbf {T_1}$:} 
Inspection
of the proof of \cite[Lemma~{5.5}]{melenk-wohlmuth12} shows that 
\begin{eqnarray*}
T_1 &\leq& 
\sqrt{h} \|\nabla (w - w_h)\|_{L^2(S(0,\tilde c h))} \leq  
\sqrt{h} \|\nabla (w - w_h)\|_{L^2(\Omega_i)}.
\end{eqnarray*} 
We conclude together with (\ref{eq:lemma:5.5-100})
$$
T_1 \leq C h^{3/2} \|v\|_{L^2(\Omega)}. 
$$
We now turn to estimating $T_2$. 
As in the proof of \cite[Lemma~{5.5}]{melenk-wohlmuth12}, we consider
the lowest order case $k = 1$ and the higher order cases $k > 1$  separately. 

{\bf Estimating $\mathbf {T_2}$ for $\mathbf {k = 1}$:} Inspection of the proof of 
\cite[Lemma~{5.5}]{melenk-wohlmuth12} gives (cf.~\cite[eqn.~(5.13)]{melenk-wohlmuth12})
\begin{align}
\label{eq:lemma:lemma:5.5-10} 
& \int_{\tau = \tilde c h}^D \|\nabla (w-  w_h) \|_{L^2(\gamma_\tau)}\,d\tau  \\
\nonumber 
& \quad \lesssim |\ln h|^{1/2} \Bigl(  \|\delta_\Gamma^{1/2} \nabla (w - I^1_h w)\|_{L^2(\cyl^\prime\setminus S_{c_1 h})}
+ \|\delta_\Gamma^{-1/2} (w - w_h)\|_{L^2(\cyl^\prime\setminus S_{c_1 h})}
\Bigr). 
\end{align} 
The term $w - I^1_h w$ is estimated as in the proof of \cite[Lemma~{5.5}]{melenk-wohlmuth12}
with the aid of (weighted) $H^2$-regularity asserted in \cite[Lemma~{5.4}]{melenk-wohlmuth12} 
and the estimate $\|w\|_{B^{3/2}_{2,\infty}(\Omega)} \lesssim \sqrt{h} \|v\|_{L^2(\Omega)}$ 
of \cite[Lemma~{5.2}]{melenk-wohlmuth12}. In total, we get 
$$
\|\delta_\Gamma^{1/2} \nabla (w - I^1_h w)\|_{L^2(\cyl^\prime\setminus S_{c_1 h})} \lesssim 
|\ln h|^{1/2} h^{3/2} \|v\|_{L^2(\Omega)}. 
$$
The second contribution of the right-hand side of (\ref{eq:lemma:lemma:5.5-10}) 
has to be treated with more care than in the proof of \cite[Lemma~{5.5}]{melenk-wohlmuth12}, 
where the convexity of $\Omega$ was conveniently exploited in order to control $\|w - w_h\|_{L^2(\Omega)}$; 
more precisely, we {\em do not} have full $H^2$-regularity but only the limited shift theorem
of \cite[Assumption~{(5.2)}]{melenk-wohlmuth12}. In order to control the term 
$\displaystyle 
\|\delta_\Gamma^{-1/2} (w - w_h)\|_{L^2(\cyl^\prime\setminus S_{c_1 h})}
$
appearing on the right-hand side of (\ref{eq:lemma:lemma:5.5-10}), 
we proceed by yet another duality argument.  
Let $\psi:= T^D(\widetilde \delta_\Gamma^{-1} (w - w_h))$, 
where $\widetilde \delta_\Gamma$ is the 
regularized distance function $\widetilde \delta_\Gamma \sim h + \delta_\Gamma$. 
Note that $\widetilde \delta_\Gamma \sim \delta_\Gamma$ on $\cyl^\prime \setminus S_{c_1 h}$. 
Denote by $\lambda_\psi$ the Lagrange multiplier for $\psi$, i.e., 
$\lambda_\psi|_{\gamma_l}  = -\partial_{n} \psi|_{\Omega_{s(l)}}$. From 
Lemma~\ref{lemma:regularity-weighted-rhs-appendix} we have for some $\varepsilon > 0$ 
given by the stipulated shift theorem (\cite[Assumption~{(5.2)}]{melenk-wohlmuth12}) 
\begin{eqnarray}
\label{eq:lemma:5.5-150}
\|\psi\|_{B^{3/2}_{2,\infty}(\Omega)} &\lesssim & |\ln h|^{1/2} \|\widetilde \delta_\Gamma^{-1/2} (w - w_h)\|_{L^2(\Omega)}, \\
\|\lambda_\psi\|_{H^\varepsilon_{pw}(\Gamma)}  + 
\|\psi\|_{H^{3/2+\varepsilon}(\Omega)} & \lesssim & h|^{-\varepsilon} \|\widetilde \delta_\Gamma^{-1/2} (w - w_h)\|_{L^2(\Omega)}. 
\end{eqnarray}
The pair $(\psi,\lambda_\psi)$ solves the following saddle point problem: 
\begin{subequations}
\label{eq:lemma:5.5-200}
\begin{align}
\label{eq:lemma:5.5-200a}
a(z,\psi) + b(z,\lambda_\psi) &= (z, \widetilde \delta^{-1} (w - w_h)) \qquad \forall 
z \in \{ z \in \prod_i H^1(\Omega_i)\,|\, z|_{\partial\Omega} = 0\}, \\
\label{eq:lemma:5.5-200b}
b(\psi,q) &= 0 \qquad \forall q \in \prod_{i} H^{1/2}_{00}(\gamma_i). 
\end{align}
\end{subequations}
The Galerkin orthogonality satisfied by $w - w_h$ reads for arbitrary $I \psi \in V_h$ 
\begin{eqnarray*}
a(w - w_h , I \psi) + b(I \psi,\lambda_w) &=& 0. 
\end{eqnarray*}
Hence, we get by taking $z = w - w_h$ in (\ref{eq:lemma:5.5-200a})
\begin{eqnarray}
\label{eq:lemma:5.5-1000}
\|\widetilde \delta_\Gamma^{-1/2} (w - w_h) \|^2_{L^2(\Omega)} &=& 
a(w - w_h,\psi) + b(w - w_h,\lambda_\psi) \\
\nonumber 
&= &
a(w - w_h,\psi - I\psi) - b(I\psi,\lambda_w) + b(w - w_h,\lambda_\psi)  . 
\end{eqnarray}
We estimate $|a(w - w_h,\psi - I \psi)| \lesssim \|w - w_h\|_{H^1} \|\psi - I \psi\|_{H^1}$. 
Since $I \psi \in V_h$ and $[\psi ]  = 0$: 
\begin{eqnarray*}
|b(I \psi,\lambda_w)| &= & b(\psi - I \psi,\lambda_w - \Pi^{L^2}_{M_h} \lambda_w)| 
\lesssim \|\psi - I \psi\|_{L^2(\Gamma)} \|\lambda_w - \Pi^{L^2}_{M_h} \lambda_w\|_{L^2(\Gamma)} \\
&\leq & \|\psi - I \psi\|_{L^2(\Omega)}^{1/2} \|\psi - I \psi\|_{H^1}^{1/2}   h^\varepsilon \|\lambda_w\|_{H^\varepsilon_{pw}(\Gamma)},
\end{eqnarray*}
where, in the last step, we employed the multiplicative trace inequality and the approximation
property (\ref{eq:approximation-property-L2-projection}). 
For the term $b(w - w_h,\lambda_\psi)$, we employ again that $[w] = 0$ and that $w_h \in V_h$ 
to get 
\begin{eqnarray*}
b(w - w_h,\lambda_\psi) & = &
\int_\Gamma [w - w_h] (\lambda_\psi - \Pi^{L^2}_{M_h} \lambda_\psi) \\
&= &
\int_\Gamma ([w - w_h]- \Pi^{L^2}_{M_h} [w - w_h]) (\lambda_\psi - \Pi^{L^2}_{M_h} \lambda_\psi)  \\
&\lesssim & h^{1/2} \|[w - w_h]\|_{H^{1/2}_{pw}(\Gamma)} h^\varepsilon \|\lambda_\psi\|_{H^\varepsilon_{pw}(\Gamma)}.
\end{eqnarray*}
Noting $\|[w - w_h]\|_{H^{1/2}_{pw}(\Gamma)}\lesssim \|w - w_h\|_{H^1}$ 
we obtain by inserting the above estimates in 
(\ref{eq:lemma:5.5-1000})
\begin{align*}
 &\|\widetilde \delta^{-1/2} (w - w_h)\|^2_{L^2(\Omega)}  
\lesssim \\
& \quad  \|w - w_h\|_{H^1} \|\psi - I \psi\|_{H^1}   
+ \|\psi - I \psi\|^{1/2}_{L^2} \|\psi - I \psi\|^{1/2}_{H^1} h^\varepsilon \|\lambda_w\|_{H^\varepsilon_{pw}(\Gamma)}
\\ 
& \quad \mbox{}
+ h^{1/2} \|w - w_h\|_{H^1} h^\varepsilon \|\lambda_\psi\|_{H^\varepsilon_{pw}}. 
\end{align*}
The terms involving $\psi - I \psi$ are estimate with the aid of Lemma~\ref{lemma:approximation-from-constrained-space}
and the bound $\|\psi\|_{B^{3/2}_{2,\infty}(\Omega)}$ given in 
(\ref{eq:lemma:5.5-150});  
the terms $\|w - w_h\|_{H^1}$ are controlled in (\ref{eq:lemma:5.5-100}); using 
$\|\lambda_w\|_{H^\varepsilon_{pw}(\Gamma)} \lesssim \|v\|_{H^{-1/2+\varepsilon}(\Omega)}$ and 
(\ref{eq:lemma:5.5-150}) for $\|\lambda_\psi\|_{H^\varepsilon_{pw}(\Gamma)}$ we get 
\begin{align*}
&\|\widetilde \delta^{-1/2} (w - w_h)\|_{L^2(\Omega)}  \lesssim \\
& 
h^{3/2} |\ln h|^{1/2} \|v\|_{L^2(\Omega)} + h^{1} |\ln h|^{1/2} \|v\|_{L^2(\Omega)} 
h^\varepsilon \|v\|_{H^{-1/2+\varepsilon}(\Omega)} + 
h^{3/2} h^{\varepsilon} h^{-\varepsilon} \|v\|_{L^2(\Omega)}. 
\end{align*}
Now the result follows from (\ref{eq:lemma:5.5-120}). 

{\bf Estimating $\mathbf {T_2}$ for $\mathbf {k > 1}$:} We proceed similarly to 
the case $k = 1$. The difference is that (cf.~\cite[(5.12)]{melenk-wohlmuth12}) we need to 
control
$$
h^{-\varepsilon} \|\delta_\Gamma^{-1/2+\varepsilon} (w -w_h)\|_{L^2(\cyl^\prime \setminus S_{\tilde c_1 h})}. 
$$
As in the case $k = 1$, we set up a dual problem with solution 
$\psi = T^D(\widetilde \delta_\Gamma^{-1+2\varepsilon} (w - w_h))$ 
and corresponding Lagrange multiplier $\lambda_\psi$. 
By Lemma~\ref{lemma:regularity-weighted-rhs-appendix} (and trace estimates) we have the regularity assertions 
\begin{equation}
\label{eq:lemma:5.5-600}
\|\lambda_\psi\|_{H^\varepsilon_{pw}(\Gamma)} + 
\|\psi\|_{H^{3/2+\varepsilon}(\Omega)} \lesssim \|\widetilde \delta_\Gamma^{-1/2+\varepsilon} (w- w_h)\|_{L^2(\Omega)}. 
\end{equation}
Proceeding in the same manner as in the case $k = 1$, we arrive at 
\begin{align*}
& \|\widetilde \delta^{-1/2+\varepsilon} (w - w_h)\|^2_{L^2(\Omega)}   \\
& \lesssim \|w - w_h\|_{H^1} \|\psi - I \psi\|_{H^1}  
+ \|\psi - I \psi\|^{1/2}_{L^2} \|\psi - I \psi\|^{1/2}_{H^1} h^\varepsilon \|\lambda_w\|_{H^\varepsilon_{pw}(\Gamma)}  
\\ 
& \quad \mbox{} 
+ h^{1/2} \|w - w_h\|_{H^1} h^\varepsilon \|\lambda_\psi\|_{H^\varepsilon_{pw}}. 
\end{align*}
The regularity assertions (\ref{eq:lemma:5.5-600}) as well as the approximation properties 
of Lemma~\ref{lemma:approximation-from-constrained-space} yield 
$
\|\widetilde \delta^{-1/2+\varepsilon} (w - w_h)\|_{L^2(\Omega)} 
\lesssim h^{3/2+\varepsilon} \|v\|_{L^2(\Omega)}
$
and hence 
$$
h^{-\varepsilon} \|\delta_\Gamma^{-1/2+\varepsilon} (w -w_h)\|_{L^2(\cyl^\prime \setminus S_{\tilde c_1 h})}
\lesssim h^{3/2} \|v\|_{L^2(\Omega)}. 
$$
\qed
\end{proof}
\begin{lemma}[Generalization of \protect{\cite[Lemma~{5.2}]{melenk-wohlmuth12}}]
\label{lemma:regularity-weighted-rhs-appendix}
Let $\widetilde \delta_\Gamma$ be the regularized distance function.  
Then for the operator $T^D$ we have 
for $w:= T^D(\widetilde \delta_\Gamma^{-1} v)$ (and $\varepsilon > 0$
sufficiently small): 
\begin{eqnarray}
\label{eq:lemma:regularity-weighted-rhs-appendix-10}
\| w\|_{B^{3/2}_{2,\infty}(\Omega)} &\leq& 
C |\ln h|^{1/2} \|\widetilde \delta_\Gamma^{-1/2} v\|_{L^2(\Omega)},  \\
\label{eq:lemma:regularity-weighted-rhs-appendix-15}
\| w\|_{H^{3/2+\varepsilon}(\Omega)} &\leq& 
C_\varepsilon h^{-\varepsilon} \|\widetilde \delta_\Gamma^{-1/2} v\|_{L^2(\Omega)}.
\end{eqnarray}
For $w = T^D(\widetilde\delta_\Gamma^{-1+2\varepsilon} v)$ we have 
\begin{eqnarray}
\label{eq:lemma:regularity-weighted-rhs-appendix-20}
\|w\|_{H^{3/2+\varepsilon}(\Omega)} &\leq & C 
\|\widetilde \delta_\Gamma^{-1/2+\varepsilon} v\|_{L^2(\Omega)}. 
\end{eqnarray}
\end{lemma}
\begin{proof} 
The proof is done with the same arguments as those of Lemma~\ref{lemma:regularity-weighted-rhs}.
\qed
\end{proof}
We need an approximation result for the approximation from the constrained space $V_h$ for functions
that do not permit nodal interpolation. 
\begin{lemma}[approximation from constrained space] 
\label{lemma:approximation-from-constrained-space}
Let the constrained space $V_h$ be defined in \cite[(2.4b)]{melenk-wohlmuth12} and assume hypotheses 
\cite[(A1), (A2)]{melenk-wohlmuth12}. Define the operator $\widetilde P_h$ as in \cite[(4.2)]{melenk-wohlmuth12}
but replace the interpolation operator $I_k$ by a (subdomainwise) Scott-Zhang operator $I^{k,SZ}_h$ 
(cf.~\cite{scott-zhang90})
that conforms 
to the boundaries of the subdomains $\Omega_i$, $i=1,\ldots,M$. Then $\widetilde P_h$ is defined
on $H^s(\Omega) \cap H^1_0(\Omega)$ for $s \ge 1$, it maps into $V_h$, and has the approximation properties
\begin{align*}
 \sqrt{\sum_{i} \|v - \widetilde P_h v\|^2_{H^1(\Omega_i)}} &\leq C h^{s-1} \|v\|_{H^s(\Omega)},  \\
\|v - \widetilde P_h v\|_{L^2(\Omega)} &\leq C h^{s} \|v\|_{H^s(\Omega)}, 
\qquad 1 \leq s \leq k+1, \\
 \sqrt{\sum_{i} \|v - \widetilde P_h v\|^2_{H^1(\Omega_i)}} &\leq C h^{s-1} \|v\|_{B^s_{2,\infty}(\Omega)}, 
\\
\|v - \widetilde P_h v\|_{L^2(\Omega)} &\leq C h^{s} \|v\|_{B^s_{2,\infty}(\Omega)}, 
\qquad 1 < s < k+1,  \qquad s \not \in \BbbN. 
\end{align*}
\end{lemma}
\begin{proof}
We only show the first estimates as the estimates with $B^{s}_{2,\infty}(\Omega)$-regularity follow
by interpolation arguments.  
It suffices to study the contribution $E_k \Pi_h [I^{k,SZ} v]$ to the operator $\widetilde P_h$. We recall 
that by the multiplicative trace inequality $\|w\|^2_{L^2(\partial\Omega_i)}\lesssim 
\|w\|_{L^2(\Omega_i)} \|w\|_{H^1(\Omega_i)}$ and the simultaneous approximation properties of 
$I^{k,SZ}$ in $L^2$ and $H^1$  we have 
$\|v - I^{k,SZ} v\|_{L^2(\partial\Omega_i)} \leq C h^{s-1/2} \|v\|_{H^s(\Omega_i)}$. 
Exploiting 
the $L^2$-stability of the mortar projection $\Pi_h$, we get the bound 
\begin{align*}
\|E_k \Pi_h [I^{k,SZ} v]\|_{H^1(\Omega_i)} & \lesssim h^{-1} h^{1/2} \|[I^{k,SZ} v]\|_{L^2(\partial\Omega_i)}  \\
& \lesssim h^{-1/2} \|v - I^{k,SZ} v\|_{L^2(\partial\Omega_i)} \lesssim h^{s-1} \|v\|_{H^s(\Omega_i)}. 
\end{align*}
\qed
\end{proof}
\fi 

%% file: appendix_lemma.5.4.tex
\section{Details concerning Lemma~\ref{lemma:5.4}}
\label{sec:lemma:5.4}
The following lemma is an expanded version of Lemma~\ref{lemma:5.4}. 
It is closely related to Lemma~\ref{lemma:weighted-shift-theorem} with 
$\beta = 1/2$ there. The essential difference is that we replace the norm 
$\|z\|_{H^{3/2}(\Omega)}$ with the weaker norm $\|z\|_{B^{3/2}_{2,\infty}(\Omega)}$ 
at the expense of a factor $|\ln h|^{1/2}$.  
\begin{lemma}
\label{lemma:weighted-inverse-estimate}
Let the bounded Lipschitz domain $\Omega\subset\BbbR^d$, $d \in \{2,3\}$.
Assume that $w \in B^{3/2}_{2,\infty}(\Omega)$ is a solution of
$$
-\nabla \cdot (\A \nabla w) = v.  
$$
\begin{enumerate}[(i)]
\item 
\label{item:lemma:weighted-inverse-estimate-i}
There exist constants $C$ (depending only on the lower bound $\alpha_0$ for the eigenvalues of $\A$, 
the norm $\|\A\|_{C^{0,1}(\overline{\Omega})}$, and $\Omega$) and $c_1 >0$ (depending only on $\Omega$) 
such that 
with the distance function $\delta_\Gamma$
\begin{equation}
\label{eq:lemma:weighted-inverse-estimate-foo-10}
\|\sqrt{\delta_\Gamma} \nabla^2 w\|_{L^2(\Omega\setminus  S_{h})} 
\leq C \sqrt{|\ln h|} \|w\|_{B^{3/2}_{2,\infty}(\Omega)} 
+ C \|\sqrt{\delta_\Gamma} v\|_{L^2(\Omega\setminus S_{c_1 h})}.
\end{equation}
In particular, if $v|_{\Omega\setminus S_{c_1 h}} = 0$ then 
$\displaystyle 
\|\sqrt{\delta_\Gamma} \nabla^2 w\|_{L^2(\Omega\setminus  S_{h})} 
\leq C \sqrt{|\ln h|} \|w\|_{B^{3/2}_{2,\infty}(\Omega)}. 
$
\item 
\label{item:lemma:weighted-inverse-estimate-ii}
Let $c^\prime > 0$ be fixed. Assume $v|_{\Omega\setminus S_{c^\prime h}} = 0$. Then  
there exist $\widetilde c$, $c_2 > 0$ (depending only on $\Omega$) such for every $\alpha > 0$
\begin{align}
\label{eq:lemma:weighted-inverse-estimate-foo-20}
 \|\delta_\Gamma^\alpha \nabla^3 w\|_{L^2(\Omega\setminus  S_{\widetilde ch})} 
&\leq C_{A,\Omega} \Bigl[  \|\delta_\Gamma^{\alpha-1} \nabla^2 w\|_{L^2(\Omega\setminus S_{c_2 h})} 
\\
\nonumber 
& \quad 
+ \|\A\|_{C^{1,1}(\Omega)} \|\delta_\Gamma^\alpha \nabla w\|_{L^2(\Omega\setminus S_{c_2 h})}
\Bigr];
\end{align}
the constant $C_{A,\Omega}$ depends on the coercivity constant $\alpha_0$ and $\|\A\|_{C^{0,1}(\overline{\Omega})}$
as well as $\Omega$. 
\item 
\label{item:lemma:weighted-inverse-estimate-iii}
Let $c^\prime > 0$ be fixed. Assume $v|_{\Omega\setminus S_{c^\prime h}} = 0$. Assume that 
$w \in H^{3/2+\varepsilon}(\Omega)$ for some $\varepsilon \in (0,1/2)$. Then 
there exists $\widehat c > 0$ (depending only on $\Omega$ and $c^\prime$) such that 
$$
\|\nabla^2 w\|_{L^2(\Omega\setminus S_{\widehat c h})} \leq C_{A,\varepsilon} h^{-1/2+\varepsilon} \|w\|_{H^{3/2+\varepsilon}(\Omega)}, 
$$
where $C_{A,\varepsilon}$ depends on $\alpha_0$, $\|\A\|_{C^{0,1}(\overline{\Omega})}$, and $\varepsilon$. 
\end{enumerate}
\end{lemma}
\begin{proof}
{\ } \newline 
{\em Proof of (\ref{item:lemma:weighted-inverse-estimate-i}):}
We may restrict our attention to a local situation near a part of the boundary.
The boundary $\Gamma$ can locally be described by a graph $\phi$. That is, 
in a suitable coordinate system, we can define cylinders 
\begin{align*}
\cyl_\delta &= \{(x,\phi(x) + t)\,|\, \delta  < t < D+\delta, x \in B\},  \\
\cyl_\delta^\prime &= \{(x,\phi(x) + t)\,|\, \delta < t < D^\prime+\delta, x \in B^\prime\}, 
\end{align*}
where $\delta  < D < D^\prime$ and $B$, $B^\prime$ are two concentric balls with $B \subset\subset B^\prime$. 
Furthermore, for $\delta=0$ we assume $\cyl^\prime_0 \subset \Omega$. In particular, 
$\{(x,\phi(x))\,|\, x \in B^\prime\} \subset \Gamma$. We also note that 
$t \sim \operatorname*{dist}((x,\phi(x)+t),\Gamma)$. 

Let $\cyl_\delta^{\prime\prime}$ be a third cylinder of the form
$\cyl_\delta^{\prime\prime} = \{(x,\phi(x)+t)\colon \delta < t < D^{\prime\prime}+\delta, x \in B^{\prime\prime}\}$ 
where $B \subset\subset B^{\prime\prime} \subset \subset B^\prime$ and $D < D^{\prime\prime} < D^\prime$. 
Let $\chi \in C^\infty(\BbbR^d)$ be such that
$\chi|_{\cyl_0^{\prime\prime}} \equiv 1$
and $\chi|_{\Omega\setminus \cyl_0^\prime} \equiv 0$.
To simplify the notation, we assume that the functions $w$, $\A$, $v$
are given in a coordinate system commensurate with the coordinate
system describing the cylinders $\cyl_\delta$, $\cyl_\delta^\prime$, viz.,
$w$ evaluated at a point $(x,\phi(x) + t) \in \cyl_\delta^\prime$
is given by $w(x,\phi(x) + t)$.
A translation in the last variable defines the function
$\widetilde w$ by
$\widetilde w(x,\phi(x) + t):= w(x,\phi(x) + t+ 2\delta)$. We note
\begin{equation}
\label{eq:lemma:weighted-inverse-estimate-100}
-\nabla \cdot \left( \widetilde \A \nabla \widetilde w\right) = \widetilde v \qquad \mbox{ in }
\cyl^\prime_{-2 \delta}; 
\end{equation}
here (again in the coordinate system used 
to describe the cylinders) 
$\widetilde \A(x,\phi(x)+t) = \A(x,\phi(x)+t+2 \delta)$ and 
$\widetilde v(x,\phi(x)+t) = v(x,\phi(x)+t+2 \delta)$. 

{\em 1.~step:}
We show (if $\delta$ is sufficiently small) 
\begin{equation}
\label{eq:lemma:weighted-inverse-estimate-1}
\|\widetilde w\|_{H^{3/2}(\cyl_0^{\prime\prime})}\leq 
C_{A,\Omega} \left[ \sqrt{|\ln \delta|}\|w\|_{B^{3/2}_{2,\infty}(\Omega)} 
+ \delta \|\tilde v\|_{L^2(\cyl_{-\delta}^\prime)} \right],
\end{equation}
where $C_{A,\Omega}$ depends on $\Omega$, the coercivity constant $\alpha_0$ and $\|\A\|_{c^{0,1}(\overline{\Omega})}$.

Using the characterization of $H^{3/2}(\cyl_0^\prime) = (H^1(\cyl_0^\prime),H^2(\cyl_0^\prime))_{1/2,2}$
in terms of the $K$-functional, we write
(cf. also \cite[p.193, eqn. (7.4)]{devore93})
\begin{eqnarray}
\nonumber 
\|\chi \widetilde w\|^2_{H^{3/2}(\cyl_0^\prime)} &=& 
\int_{t=0}^1 \left(t^{-1/2} K(t,\chi \widetilde w)\right)^2\frac{dt}{t} \\
\label{eq:lemma:weighted-inverse-estimate-1000}
&= &
\int_{t=0}^\varepsilon 
\left(t^{-1/2} K(t,\chi \widetilde w)\right)^2\frac{dt}{t}  + 
\int_{t=\varepsilon}^1 
\left(t^{-1/2} K(t,\chi \widetilde w)\right)^2\frac{dt}{t} . 
\end{eqnarray}
The second integral
in (\ref{eq:lemma:weighted-inverse-estimate-1000})
can be estimated by
$$
\int_{t=\varepsilon}^1 \left(t^{-1/2} 
K(t,\chi \widetilde w)\right)^2\frac{dt}{t} 
\leq 
\int_{t=\varepsilon}^1 \frac{dt}{t} 
\sup_{t > 0} \left(t^{-1/2} K(t,\chi \widetilde w)\right)^2
 \leq \ln \varepsilon \|\chi \widetilde w\|^2_{B^{3/2}_{2,\infty}(\cyl_0^\prime)} .
$$
For the first integral in (\ref{eq:lemma:weighted-inverse-estimate-1000})
we employ interior regularity estimates for
solutions of second order elliptic equation with vanishing right-hand side. 
Specifically, (\ref{eq:lemma:weighted-inverse-estimate-100}) and interior regularity
(see, e.g., \cite[Thm.~{8.8} and proof]{gilbarg-trudinger77a}) 
give (here, we assume that $\delta$ is sufficiently small) 
$$
\|\chi \widetilde w\|_{H^2(\cyl_0^\prime)} \leq 
C_A \left[  \delta^{-1} \|\widetilde w\|_{H^1(\cyl_{-\delta}^\prime)} 
+ \|\widetilde v\|_{L^2(\cyl_{-\delta}^\prime)} \right], 
$$
where the constant $C_A$ depends only on the coercivity constant $\alpha_0$ and 
$\|\A\|_{C^{0,1}(\overline{\Omega})}$. 

Hence, estimating $K(t,\chi \widetilde w) = \inf_{v\in H^2} 
\|\chi \widetilde w - v\|_{H^1(\cyl_0^\prime)} + t \|v\|_{H^2(\cyl_0^\prime)}
\leq t \|\chi \widetilde w\|_{H^2(\cyl_0^\prime)}$, we obtain
$$
\int_{t=0}^\varepsilon t^{-2} K^2(t,\chi \widetilde w)\,dt 
\leq \varepsilon \|\chi \widetilde w\|^2_{H^2(\cyl_0^\prime)} 
\leq C_A \left[  \varepsilon \delta^{-2} \|\widetilde w\|^2_{H^1(\cyl_{-\delta}^\prime)} 
+ \varepsilon \|\widetilde v\|^2_{L^2(\cyl_{-\delta}^\prime)}\right].
$$
We conclude
\begin{align} 
\label{eq:lemma:weighted-inverse-estimate-10}
\|\chi \widetilde w\|^2_{H^{3/2}(\cyl_0^\prime)} 
& \leq  C \left[ \varepsilon \delta^{-2} \|\widetilde w\|^2_{H^1(\cyl_{-\delta}^\prime)} + 
\varepsilon \|\widetilde v\|^2_{L^2(\cyl_{-\delta}^\prime)} 
+ \ln \varepsilon \|\chi \widetilde w\|^2_{B^{3/2}_{2,\infty}(\cyl_0^\prime)}\right]\\
\nonumber 
& \leq  
C_{A,\Omega} \left\{ \left[ \varepsilon \delta^{-2} + \ln \varepsilon \right] 
\|w\|^2_{B^{3/2}_{2,\infty}(\cyl_0^\prime)} + \varepsilon \|\widetilde v\|^2_{L^2(\cyl_{-\delta}^\prime)}\right\},
\end{align}
where, in the last step we have employed that multiplication
by a smooth function and translation are bounded operations on Sobolev
(and therefore also Besov) spaces; the constant $C_{A,\Omega}$ depends on 
$\Omega$, $\alpha_0$, and $\|\A\|_{C^{0,1}(\overline{\Omega})}$. 
Selecting $\varepsilon = \delta^2$
shows $\|\chi \widetilde w\|_{H^{3/2}(\cyl_0^{\prime})} \leq 
C_{A,\Omega} \left[ \sqrt{|\ln \delta|} \|w\|_{B^{3/2}_{2,\infty}(\Omega)} 
+ \delta \|\widetilde v\|_{L^2(\cyl_{-\delta}^\prime)}\right]$
from which we get
(\ref{eq:lemma:weighted-inverse-estimate-1}) in view of the support
properties of $\chi$.

{\em Step 2:}
Let $z$ solve, for a fixed $\rho > 0$ and a parameter $r \leq 1$, the equation
$$
-\nabla \cdot \left(\widetilde\A \nabla z\right) =f
\qquad \mbox{ in a ball $B_{r(1+\rho)}$ of radius $r(1+\rho)$.}
$$
We claim:
For a constant $C_{A,\rho} > 0$ that depends solely on the coercivity constant $\alpha_0$,  
$\|\A\|_{C^{0,1}(\overline{\Omega})}$, and $\rho$  there holds 
\begin{eqnarray}
\label{eq:foobar-200}
\|\nabla^2 z\|_{L^2(B_r)} 
&\leq & C_{A,\rho}  \left[ \|f\|_{L^2(B_{r(1+\rho)})} + \|\nabla z\|_{L^2(B_{r(1+\rho)})} + r^{-1/2} |z|_{H^{3/2}(B_{r(1+\rho)})} 
\right], \qquad \\
\label{eq:foobar-210}
\|\nabla^2 z\|_{L^2(B_r)} 
&\leq & C_{A,\rho} \left[ \|f\|_{L^2(B_{r(1+\rho)})} + r^{-1} \|\nabla z\|_{L^2(B_{r(1+\rho)})} \right].
\end{eqnarray}
We point out that the $H^{3/2}$-seminorm in (\ref{eq:foobar-200}) is defined in terms of the Aronstein-Slobodeckij
norm for scaling reasons.
The bounds (\ref{eq:foobar-200}), (\ref{eq:foobar-210})  follow from interior regularity in the following way.

Scaling the ball $B_{r(1+\rho)}$ to a ball $\widehat B_{1+\rho}$ of 
radius $1 +\rho$ leads to an equation of the form 
\begin{equation}
\label{eq:local-laplace} 
-\nabla \cdot \left(\widehat \A \nabla \widehat z\right) = r^2 \widehat f
\qquad \mbox{ in a ball $\widehat B_{1+\rho}$ of radius $(1+\rho)$}, 
\end{equation}
where $\widehat\A$ and $\widehat f$ are the coefficient and the right-hand side in the scaled variables. 
We note that 
$$
\|\nabla^j \widehat \A \|_{L^\infty(\widehat B_{1+\rho})} \sim  r^j \|\nabla^j \widetilde \A\|_{L^\infty(B_{r(1+\rho)})}, 
\qquad j \in \BbbN_0.
$$
Then standard interior regularity
(see, e.g., \cite[Thm.~{8.8}]{gilbarg-trudinger77a}) gives 
\begin{equation}
\|\nabla^2 \widehat z\|_{L^2(\widehat B_1)} \leq C_{A,\rho}
\left[ r^2 \|\widehat f\|_{L^2(\widehat B_{1+\rho})} + \|\widehat z\|_{H^1(\widehat B_{1+\rho})} \right], 
\end{equation}
where the constant $C_{A,\rho}$ depends only on the coercivitiy constant of $\widehat \A$, the norm 
$\|\widehat \A\|_{C^{0,1}(B_{1+\rho})}$, and $\rho$. 
In view of $r \leq 1$ and the fact that the operator $\widetilde \A$ is obtained from $\A$ by an affine
change of variables (a translation and an orthogonal transformation) 
we bound $\|\widehat \A\|_{C^{0,1}(\widehat B_{1+\rho})} \leq 
C \|A\|_{C^{0,1}(\overline{\Omega})}$ and get 
\begin{equation}
\label{eq:foobar-1}
\|\nabla^2 \widehat z\|_{L^2(\widehat B_1)} \leq C_{A,\rho}
 \left[ r^2 \|\widehat f\|_{L^2(\widehat B_{1+\rho})} + 
\|\widehat z\|_{H^1(\widehat B_{1+\rho})} \right], 
\end{equation}
for a constant $C_{A,\rho}$ that depends only on $\rho$, the coercivity constant of $\A$, 
and $\|\A\|_{C^{0,1}(\overline{\Omega})}$. 
Since the constant functions are in the kernel of the operator $-\nabla \cdot (\widehat \A \nabla z)$ 
it is easy to conclude with a Poincar\'e inequality that (\ref{eq:foobar-1}) implies 
\begin{equation}
\label{eq:foobar-1000}
\|\nabla^2 \widehat z\|_{L^2(\widehat B_1)} \leq C_{A,\rho} \left[ r^2 \|\widehat f\|_{L^2(\widehat B_{1+\rho})} + 
\|\nabla \widehat z\|_{L^2(\widehat B_{1+\rho})} \right]. 
\end{equation}
Scaling this equation back to $B_{r(1+\rho)}$ yields the desired bound 
(\ref{eq:foobar-210}). For the proof of 
(\ref{eq:foobar-200}), we have to bring in the $H^{3/2}$-seminorm. 
Let $\pi \in {\mathcal P}_1$ be arbitrary. Then the function $\widehat z - \pi$ satisfies 
in view of the fact that $\nabla \pi$ is constant 
$$
-\nabla \cdot \left(\widehat \A \nabla (\widehat z- \pi)\right) = 
r^2 \widehat f + \nabla \cdot \left(\widehat \A \nabla \pi\right) = r^2 \widehat f+
(\nabla \cdot \widehat \A) \cdot \nabla \pi  =:\widetilde f. 
$$
(We employed the convention that the divergence operator $\nabla \cdot$ in the expression $\nabla \cdot \widehat \A$ 
acts on columns of $\widehat \A$). 
Applying (\ref{eq:foobar-1}) to this equation (and replacing $r^2 \widehat f$ with $\widetilde f$ and 
$\widehat z$ with $\widehat z - \pi$) yields 
\begin{align*}
\|\nabla^2 \widehat z\|_{L^2(\widehat B_1)} &\leq 
\|\widehat z - \pi\|_{H^2(\widehat B_1)} \leq C_{A,\rho} 
\left[ \|\widetilde f\|_{L^2(\widehat B_{1+\rho})}  + \| \widehat z -\pi\|_{H^1(\widehat B_{1+\rho})} \right] \\
&\leq C_{A,\rho}
\left[ r^2 \|\widehat f\|_{L^2(B_{r(1+\rho)})} + r \|\nabla \pi\|_{L^2(\widehat B_{1+\rho})}  
+  \| \widehat z -\pi \|_{H^{3/2}(\widehat B_{1+\rho})} \right] \\
&\leq C_{A,\rho}  
\left[ r^2 \|\widehat f\|_{L^2(B_{r(1+\rho)})} + r \|\nabla \widehat z\|_{L^2(\widehat B_{1+\rho})}  
+  \| \widehat z -\pi \|_{H^{3/2}(\widehat B_{1+\rho})} \right].  
\end{align*}
Infimizing over all $\pi \in {\mathcal P}_1$ yields 
\begin{equation}
\label{eq:foobar-10}
\|\nabla^2 \widehat z\|_{L^2(\widehat B_1)} 
\leq C_{A,\rho} \left[ r^2 \|\widehat f\|_{L^2(\widehat B_{1+\rho})} + r \|\nabla \widehat z\|_{L^2(\widehat B_{1+\rho})} + |\widehat z|_{H^{3/2}(\widehat B_{1+\rho})} \right].
\end{equation}
Scaling back to $B_{r(1+\rho)}$ yields (\ref{eq:foobar-200}), if we note the scaling properties of the 
Aronstein--Slobodeckij seminorm.

{\em 3.~step:} 
Applying the result of step~2 to the function $\widetilde w$  yields 
\begin{equation}
\label{eq:foobar-1010}
\|\nabla^2 \widetilde w \|_{L^2(B_r)} \leq 
C_{A,\rho} \left[ \|\nabla \widetilde w\|_{L^2(B_{r(1+\rho)})} + r^{-1/2} |\widetilde w|_{H^{3/2}(B_{(1+\rho)r})} 
+ \|\widetilde v\|_{L^2(B_{r(1+\rho)})}
\right]
\end{equation}
for all 
balls $B_r$ such that $B_{(1+\rho) r} \subset \cyl_{-2\delta}^\prime$. 
Using, for example, the Besicovitch covering theorem \cite{evans-gariepy}, we can
cover $\cyl_0$ by overlapping balls $B_{r_i}(x_i)$
with centers $x_i$ and radii $r_i 
\sim \delta_\Gamma(x_i)$ such that the stretched balls
$B_{r_i(1+\rho)}(x_i)$ have a finite overlap property (see
\cite[Lemma~A.1]{melenk-wohlmuth12} for details).
A covering argument and afterwards (\ref{eq:lemma:weighted-inverse-estimate-1})
then show
\begin{align*}
\label{eq:besicovic-argument}
\|\sqrt{\delta_{\Gamma}} \nabla^2 \widetilde w\|_{L^2(\cyl_0)}
& \leq C_{A,\Omega} \left[ \|\widetilde w\|_{H^{3/2}(\cyl^{\prime\prime}_0)} 
+ \|\sqrt{\delta_\Gamma} \nabla \widetilde w\|_{L^2(\cyl_{0}^\prime)} 
+ \|\sqrt{\delta_\Gamma} \widetilde v\|_{L^2(\cyl_{0}^\prime)} 
\right] \\
\nonumber 
& 
\leq C_{A,\Omega} \left[ \sqrt{|\ln \delta|} \|w\|_{B^{3/2}_{2,\infty}(\Omega)} 
+ \delta \|\widetilde v\|_{L^2(\cyl_{-\delta}^\prime)}  
+ \|\sqrt{\delta_\Gamma} \widetilde v\|_{L^2(\cyl_0^\prime)}
\right], 
\end{align*}
where the constant $C_{A,\Omega}$ depends only on the coercivity constant $\alpha_0$, 
$\|\A\|_{C^{0,1}(\overline{\Omega})}$, and $\Omega$. 

Since $\widetilde w$ and $\widetilde v$ are obtained by a translation, we arrive 
at
$$
\|\sqrt{\delta_{\Gamma}} \nabla^2 w\|_{L^2(\cyl_{2\delta})} 
\leq C_{A,\Omega} \left[ \sqrt{|\ln \delta|} \|w\|_{B^{3/2}_{2,\infty}(\Omega)}
+ \|\sqrt{\delta_\Gamma} v\|_{L^2(\Omega\setminus S_{c_2 \delta})}
\right]
$$
for a suitable $c_2$ that depends solely on the Lipschitz character of $\Gamma$. 
Taking $\delta \sim h$ produces the desired result. 

{\em Proof of (\ref{item:lemma:weighted-inverse-estimate-ii}):}
The estimate merely expresses
interior regularity for solutions of elliptic equations with vanishing right-hand side.
It follows from (\ref{eq:foobar-210}) and a covering argument. More precisely, as in Step~{2}, 
we start from 
$$
-\nabla \cdot \left( \A \nabla z \right) = f \quad \mbox{ in a ball $B_{r(1+\rho)}$ 
of radius $r (1+\rho)$}.
$$
Differentiating this equation once gives for $\alpha \in \BbbN^d$ with $|\alpha| = 1$
$$
-\nabla\cdot \left( \A \nabla D^\alpha z \right) = D^\alpha f  
+ \nabla \cdot \left( D^\alpha \A\right) \nabla z 
\quad \mbox{ in a ball $B_{r(1+\rho)}$ 
of radius $r (1+\rho)$},
$$
where again the divergence operator in the expression 
$\nabla \cdot \left( D^\alpha \A\right)$ acts on the columns of $D^\alpha \A$. 
We get from (\ref{eq:foobar-210}) by considering all $\alpha \in \BbbN^d$
with $|\alpha| = 1$ 
\begin{align*}
&\|\nabla^3 z\|_{L^2(B_r)} \leq \\
& \quad C_{A,\rho} \left[ 
\|\nabla f\|_{L^2(B_{r(1+\rho)})} + r^{-1} \|\nabla^2 z\|_{L^2(B_{r(1+\rho)})} 
+ \|\nabla^2 \A\|_{L^\infty(B_{r(1+\rho)})} \|\nabla z\|_{L^2(B_{r(1+\rho)})}\right].
\end{align*}
A covering argument then produces the claim since $f$ is assumed to satisfy $f \equiv 0$ on 
$\Omega\setminus S_{c^\prime h}$ in the statement (\ref{item:lemma:weighted-inverse-estimate-ii}). 
We also note that $\|\nabla^2 \A \|_{L^\infty(\Omega)}$ can be bounded by $\| \nabla \A\|_{W^{1,\infty}(\Omega)}$. 

{\em Proof of (\ref{item:lemma:weighted-inverse-estimate-iii}):} 
Our starting point is the above Step~2: We claim that for the function  $z$ satisfying 
(\ref{eq:local-laplace}) we have 
\begin{equation}
\label{eq:foobar-10000}
\|\nabla^2 z\|_{L^2(B_r)} \leq C_{A} \left[ \|f\|_{L^2(B_{r(1+\rho)})} + 
\|\nabla z\|_{L^2(B_{r(1+\rho)})} + r^{-1/2+\varepsilon} |z|_{H^{3/2+\varepsilon}(B_{r(1+\rho)})}\right], 
\end{equation}
where the $H^{3/2+\varepsilon}$-seminorm is again an Aronstein-Slobodeckij norm, and the 
constant $C_A$ depends only on the coercivity constant $\alpha_0$, $\|\A\|_{C^{0,1}(\overline{\Omega})}$, 
and $\varepsilon \in (0,1/2)$. 
To see this, we 
check the derivation of (\ref{eq:foobar-10}). We see that one can (marginally) modify the arguments
to obtain instead of (\ref{eq:foobar-10}) the estimate 
$$
\|\nabla^2 \widehat z\|_{L^2(\widehat B_{1+\rho})} 
\lesssim r^2 \|\widehat f\|_{L^2(\widehat B_{1+\rho})} + r\|\nabla \widehat z\|_{L^2(\widehat B_{1+\rho})} + 
|\widehat z|_{H^{3/2+\varepsilon}(\widehat B_{1+\rho})},
$$
where the implied constant depends on $\rho$, $\alpha_0$, $\|\A\|_{C^{0,1}(\overline{\Omega})}$, and $\varepsilon$.
Scaling back to $B_{r(1+\rho)}$ yields 
$$
r^2 \|\nabla^2 z\|_{L^2(B_{r(1+\rho)})} 
\lesssim r^2 \|f\|_{L^2(B_{r(1+\rho)})} + r r\|\nabla z\|_{L^2(B_{r(1+\rho)})} + 
r^{3/2+\varepsilon} |z|_{H^{3/2+\varepsilon}(\widehat B_{r(1+\rho)})},
$$
which leads to (\ref{eq:foobar-10000}). We now use $f \equiv 0$ on 
$\Omega\setminus S_{c^\prime h}$ and use these estimates with $w$ in place of $z$. 
A covering argument then gives 
$$
\|\nabla^2 w\|_{L^2(\Omega\setminus S_{\widehat c h})} \lesssim \|\nabla w\|_{L^2(\Omega \setminus S_{c^\prime h})} 
+ h^{-1/2+\varepsilon} \|w\|_{H^{3/2+\varepsilon}(\Omega)}.  
$$
This concludes the proof of 
(\ref{item:lemma:weighted-inverse-estimate-iii}). 
\qed
\end{proof}

%% file: appendix_wahlbin.tex
\newcommand{\tri}[1]{\|\!|#1|\!\|}
\section{Details of \cite[Sec.~{5.3}]{wahlbin95}}
By and large, we follow the arguments of \cite[Sec.~{5.3}]{wahlbin95} and adapt them
as needed.

In what follows, ${\mathcal T}$ is a quasi-uniform mesh with mesh size $h$. 
We consider the bilinear form 
$$
a(u,v):= \int_\Omega \A \nabla u \cdot \nabla v
$$
where the matrix $\A$ is sufficiently smooth and pointwise symmetric positive definite.

\subsection{The interior case}
In this subsection, we assume that  $V_h  = S^{k,1}({\mathcal T})$.

For balls $B_d$ of radius $d$, it will be convenient to introduce the notation
$$
\tri{u}_{1,B_d}:= |u|_{H^1(B_d)} + \frac{1}{d} \|u\|_{L^2(B_d)}.
$$
We start with making the notion of ``superapproximation'' more precise: 
\begin{lemma}[superapproximation on balls]
\label{lemma:superapproximation-ball}
Let $B_d \subset \Omega$ be a ball of radius $d$. Let $\omega \in C^\infty$ with 
$\supp \omega \subset B_{d/2}$ and 
\begin{equation}
\label{eq:cut-off}
\|\nabla^j \omega\|_{L^\infty} \leq C d^{-j}, \qquad j=0,\ldots,k. 
\end{equation}
Then for every $u \in V_h$ the interpolant $I(\omega^2 u) \in S^{k,1}({\mathcal T}) \cap H^1_0(B_d)$
satisfies 
\begin{eqnarray}
|\omega^2 u - I(\omega^2 u)|_{H^1(B_d)} &\leq& C \frac{h}{d} \tri{u}_{1,B_d}, \\
\frac{1}{d}\|\omega^2 u - I(\omega^2 u)\|_{L^2(B_d)} &\leq& C \left(\frac{h}{d}\right)^2 \tri{u}_{1,B_d}.
\end{eqnarray}
We assume implicitly that $d > h$ is sufficiently large. 
\end{lemma}
\begin{proof}
Since $\supp \omega\subset B_{d/2}$ and $d$ is large compared to $h$, we have 
$\supp I (\omega^2 u) \subset B_d$. For each element $K$ we have the estimates 
$$
\|\omega^2 u - I(\omega^2 u) \|_{L^2(K)} + 
h \|\nabla( \omega^2 u - I(\omega^2 u)) \|_{L^2(K)} 
\leq C h^{k+1} \|\nabla^{k+1} (\omega^2 u)\|_{L^2(K)}. 
$$ 
Inductively, we see that 
$
\displaystyle \|\nabla^j (\omega^2) \|_{L^\infty} \leq C d^{-j}$ for $j=0,\ldots,k$.
Using the fact that $u$ is piecewise polynomial of degree $k$ we conclude 
$$
\|(\nabla^{k+1} (\omega^2 u)\|_{L^2(K)} 
\leq C \sum_{j=0}^k d^{-(k+1-j)} \|\nabla^j u\|_{L^2(K)}.
$$
An inverse estimate produces in view of $h/d \lesssim 1$ 
\begin{align*}
\|\omega^2 u - I(\omega^2 u) \|_{L^2(K)} + 
h \|\nabla( \omega^2 u - I(\omega^2 u)) \|_{L^2(K)} 
&\leq C h \frac{h}{d} 
\left[ \frac{1}{d} \|u\|_{L^2(B_d)} + |u|_{H^1(B_d)} \right]  \\
& \leq C h \frac{h}{d} \tri{u}_{1,B_d}. 
\end{align*}
\qed 
\end{proof}
The next result shows an inverse estimate for ``discrete harmonic'' functions:
\begin{lemma} 
\label{lemma:inverse-estimate-discrete-harmonic-ball}
Let $u_h \in V_h$ satisfy 
\begin{equation}
\label{eq:lemma:inverse-estimate-discrete-harmonic-ball-10}
a(u_h,v) = 0 \qquad \forall v \in V_h \quad \mbox{ with } \quad \supp v \subset B_d
\end{equation}
for a ball $B_d \subset \Omega$ of radius $d$ (implicitly assumed sufficiently large
compared to $h$).
Then 
$$
\|\nabla u_h\|_{L^2(B_{d/2})} \leq C d^{-1} \|u_h\|_{L^2(B_d)}. 
$$
\end{lemma}
\begin{proof}
Select a cut-off function $\omega$ with $\supp \omega \subset B_{d}$ and 
(\ref{eq:cut-off}) as well as $\chi \equiv 1$ on $B_{d/2}$. 

We write for arbitrary $\chi \in V_h$ with $\supp \chi \subset B_d$
\begin{align*}
\int_\Omega \omega^2 \A \nabla u_h \cdot \nabla u_h &= 
\int_\Omega \A \nabla u_h \cdot \nabla (\omega^2 u_h) - 
\int_\Omega 2 u_h \omega \nabla \omega \cdot \A \nabla u_h \\
 & = 
\int_\Omega \A \nabla u_h \cdot \nabla (\omega^2 u_h - \chi) - 
\int_\Omega 2 u_h \omega \nabla \omega \cdot \A \nabla u_h.
\end{align*}
We conclude from Lemma~\ref{lemma:superapproximation-ball} with $\chi = I(\omega^2 u_h)$ there
and Young's inequality 
\begin{align*}
\int_\Omega \omega^2 \A \nabla u_h \cdot \nabla u_h & 
\lesssim  \frac{h}{d} \|\nabla u_h\|_{L^2(B_d)} \tri{u_h}_{1,B_d} + \frac{1}{d^2}\|u_h\|^2_{L^2(B_d)}. 
\end{align*} 
We conclude 
$$
\|\nabla u_h\|_{L^2(B_{d/2})} \lesssim \sqrt{ \frac{h}{d} } \tri{u_h}_{1,B_d} + \frac{1}{d} \|u_h\|_{L^2(B_d)}. 
$$
Iterating the argument yields 
$$
\|\nabla u_h\|_{L^2(B_{d/4})} \lesssim \frac{h}{d}  \tri{u_h}_{1,B_d} + \frac{1}{d} \|u_h\|_{L^2(B_d)}. 
$$
Finally, a standard inverse estimate produces 
$$
\|\nabla u_h\|_{L^2(B_{d/4})} \lesssim \frac{1}{d} \|u_h\|_{L^2(B_d)}, 
$$
which is the desired final bound. 
\qed
\end{proof}
We now show the main result: 
\begin{theorem}
\label{thm:local-quasi-optimality-ball}
Let $u \in H^1(\Omega)$ and $u_h \in V_h$ be such that 
$$
a(u - u_h,v) = 0 \qquad \forall v \in V_h \quad \mbox{ with } \quad \supp v \subset B_d
$$
for a ball $B_d \subset \Omega$ of radius $d$ (implicitly assumed sufficiently large
compared to $h$). 
Then 
\begin{align}
\label{eq:thm:local-quasi-optimality-ball-1}
\|\nabla( u - u_h)\|_{L^2(B_{d/4})} 
& \leq C \inf_{\chi \in V_h} \tri{u - \chi}_{1,B_d} + C d^{-1} \|u - u_h\|_{L^2(B_d)}  \\
\label{eq:thm:local-quasi-optimality-ball-2}
& \leq C \inf_{\chi \in V_h} \|\nabla (u - \chi)\|_{L^2(B_d)} + C d^{-1} \|u - u_h\|_{L^2(B_d)}. 
\end{align}
\end{theorem}
\begin{proof}
The second bound (\ref{eq:thm:local-quasi-optimality-ball-2}) follows from 
(\ref{eq:thm:local-quasi-optimality-ball-1}) 
by an application of the (second) Poincar\'e inequality.

Let $\omega \in C^\infty$ with $\supp \chi \subset B_d$ be such that 
$\omega \equiv 1$ on $B_{d/2}$ and 
assume (\ref{eq:cut-off}). Define 
$$
\widetilde u:= \omega u
$$
and let $\widetilde u_h \in V_h \cap H^1_0(B_d)$ be its Galerkin 
approximation on $V_h \cap B^1_0(B_d)$, i.e., 
$$
a(\widetilde u - \widetilde u_h,v) = 0 \qquad \forall v \in V_h \cap H^1_0(B_d).
$$
Then it is classical (and easy to see) that 
$$
\int_{B_d} \nabla \widetilde u_h \cdot (\A \nabla \widetilde u_h) \leq 
\int_{B_d} \nabla \widetilde u \cdot (\A \nabla \widetilde u).  
$$
In particular, we get 
\begin{equation}
\label{eq:thm:local-quasi-optimality-ball-10}
|\widetilde u_h|_{H^1(B_d)} \leq C |\widetilde u|_{H^1(B_d)} 
\lesssim \tri{u}_{1,B_d}. 
\end{equation}
Next, we write 
\begin{equation}
\label{eq:thm:local-quasi-optimality-ball-20}
u - u_h = (\widetilde u - \widetilde u_h) + (\widetilde u_h - u_h) 
\quad \mbox{ in $B_{d/2}$} 
\end{equation}
and estimate each of the two terms separately. For the first one, we employ 
(\ref{eq:thm:local-quasi-optimality-ball-10}) to get 
\begin{equation}
\label{eq:thm:local-quasi-optimality-ball-30} 
\|\nabla (\widetilde u - \widetilde u_h)\|_{L^2(B_{d/2})} \lesssim  \|\nabla \widetilde u\|_{L^2(B_d)} 
\lesssim \tri{u}_{1,B_d}. 
\end{equation}
For the second term in (\ref{eq:thm:local-quasi-optimality-ball-20}), we observe that 
$\widetilde u_h - u_h$ is discrete harmonic in $B_{d/2}$ (in fact, almost in $B_d$) 
since 
$$
a(\widetilde u_h - u_h,v)  = 
a(\widetilde u - u,v)  = 0 \qquad \forall v \in V_h \cap H^1_0(B_{d/2}). 
$$
Therefore, Lemma~\ref{lemma:inverse-estimate-discrete-harmonic-ball} is applicable 
and yields in view of $\omega \equiv 1$ on $B_{d/2}$:
\begin{align*}
\|\nabla (\widetilde u_h - u_h)\|_{L^2(B_{d/4})} & \leq C d^{-1} \|\widetilde u_h - u_h\|_{L^2(B_{d/2})} \\
& \leq C d^{-1} \|\widetilde u_h - u\|_{L^2(B_{d/2})}
+ C d^{-1} \| u - u_h \|_{L^2(B_{d/2})}\\
& = C d^{-1} \|\widetilde u_h - \widetilde u\|_{L^2(B_{d/2})}
+ C d^{-1} \| u - u_h \|_{L^2(B_{d/2})}, 
\end{align*}
where in the last step we exploited $\widetilde u|_{B_{d/2}} = u|_{B_{d/2}}$ due to $\omega|_{B_{d/2}} \equiv 1$.
Since $\widetilde u_h$, $\widetilde u \in H^1_0(B_d)$, a Poincar\'e inequality together with 
(\ref{eq:thm:local-quasi-optimality-ball-10}) produces 
\begin{align}
\nonumber 
\|\nabla (\widetilde u_h - u_h)\|_{L^2(B_{d/4})} &\lesssim \|\nabla( \widetilde u_h - \widetilde u)\|_{L^2(B_{d})} 
+ d^{-1} \|u - u_h\|_{L^2(B_{d/2})}\\
\label{eq:thm:local-quasi-optimality-ball-40} 
&\lesssim \tri{u}_{1,B_d} 
+ d^{-1} \|u - u_h\|_{L^2(B_{d/2})}
\end{align}
Combining 
(\ref{eq:thm:local-quasi-optimality-ball-30}) and 
(\ref{eq:thm:local-quasi-optimality-ball-40}) yields 
again with $\omega \equiv 1$ on $B_{d/2}$
\begin{align}
\label{eq:thm:local-quasi-optimality-ball-50} 
\|\nabla (u - u_h)\|_{L^2(B_{d/4})} 
&\leq 
\|\nabla (\widetilde u - \widetilde u_h)\|_{L^2(B_{d/4})} + 
\|\nabla (\widetilde u_h - u_h)\|_{L^2(B_{d/4})}  \\
\nonumber 
& \lesssim \tri{u}_{1,B_d} + d^{-1} \|u - u_h\|_{L^2(B_d)}. 
\end{align}
The final step consists in noting for arbitrary $\chi \in V_h$ that 
$u - u_h = (u - \chi) + (\chi - u_h)$ so that an application of (\ref{eq:thm:local-quasi-optimality-ball-50})
applied to $u - \chi$ yields 
$$
\|\nabla( u - u_h)\|_{L^2(B_{d/4})} \lesssim \tri{u - \chi}_{1,B_d} + d^{-1} \|u - u_h\|_{L^2(B_d)}. 
$$
\qed 
\end{proof}
\subsection{The boundary case}
We now check to what extent the above results extend up to the boundary. We 
Let $\Gamma$ be (a part of) the boundary $\partial\Omega$. We denote by 
$$
D_d = B_d \cap \Omega
$$
semiballs of radius $d$, where the implict assumption is always that the center of $D_d$ lies on $\Gamma$. 
Another implicit assumption in the following is that for all semiball appearing in the following,
we assume 
$$
\partial D_d \cap \partial\Omega \subset \Gamma
$$
We define 
$
\tri{u}_{1,D_d}:= |u|_{H^1(D_d)} + \frac{1}{d} \|u\|_{L^2(D_d)}.
$
We employ a space $V_h$ which satisfies 
\begin{equation}
\label{eq:V_h-bdy}
V_h \subset S^{k,1}({\mathcal T}), 
\qquad v|_{\Gamma} = 0 \quad \forall v \in V_h.
\end{equation}
We employ the following observation: For a semiball $D_d\subset B_d$ (same center) 
as described above, $\omega \in C^\infty_0(B_d)$, and $u \in V_h$, we have 
$\omega u \in H^1_0(D_d)$. 

Compared to Lemma~\ref{lemma:superapproximation-semiballs}, the cut-off function $\omega$
may be $\equiv 1$ near parts of $\Gamma$: 
\begin{lemma}[superapproximation on semiballs]
\label{lemma:superapproximation-semiballs}
Assume (\ref{eq:V_h-bdy}). 
Let $D_d \subset \Omega$ be a semiball of radius $d$. Let $\omega \in C^\infty(\BbbR^n)$ with 
$\supp \omega \subset B_{d/2}$ and (\ref{eq:cut-off}). 
Then for every $u \in V_h$ the interpolant $I(\omega^2 u) \in S^{k,1}({\mathcal T}) \cap H^1_{0}(D_d)$
satisfies 
\begin{eqnarray}
|\omega^2 u - I(\omega^2 u)|_{H^1(D_d)} &\leq& C \frac{h}{d} \tri{u}_{1,D_d}, \\
\frac{1}{d}\|\omega^2 u - I(\omega^2 u)\|_{L^2(D_d)} &\leq& C \left(\frac{h}{d}\right)^2 \tri{u}_{1,D_d}.
\end{eqnarray}
We assume implicitly that $d > h$ is sufficiently large. 
\end{lemma}
\begin{proof}
Follows by the same arguments as in Lemma~\ref{lemma:superapproximation-ball}. 
\qed 
\end{proof}
The next result shows an inverse estimate for ``discrete harmonic'' functions:
\begin{lemma} 
\label{lemma:inverse-estimate-discrete-harmonic-semiball}
Assume (\ref{eq:V_h-bdy}). 
Let $u_h \in V_h$ satisfy 
\begin{equation}
\label{eq:lemma:inverse-estimate-discrete-harmonic-semiball-10}
a(u_h,v) = 0 \qquad \forall v \in V_h \quad  \mbox{ with } \quad \supp v \subset \overline{D_d} 
\end{equation}
for a semiball $D_d \subset \Omega$ of radius $d$ (implicitly assumed sufficiently large
compared to $h$).
Then 
$$
\|\nabla u_h\|_{L^2(D_{d/2})} \leq C d^{-1} \|u_h\|_{L^2(D_d)}. 
$$
\end{lemma}
\begin{proof}
Again, this follows by tracing the arguments in the proof of 
Lemma~\ref{lemma:inverse-estimate-discrete-harmonic-semiball}.
\qed 
\end{proof}
\begin{theorem}
\label{thm:local-quasi-optimality-semiball}
Assume (\ref{eq:V_h-bdy}). 
Let $u \in H^1(\Omega)$ with $u|_{\Gamma} = 0$ and $u_h \in V_h$ be such that 
$$
a(u - u_h,v) = 0 \qquad \forall v \in V_h \quad \mbox{ with } \quad \supp v \subset \overline{D_d}
$$
for a semiball $D_d \subset \Omega$ of radius $d$ (implicitly assumed sufficiently large
compared to $h$). 
Then 
\begin{align*}
\|\nabla( u - u_h)\|_{L^2(D_{d/4})} 
& \leq C \inf_{\chi \in V_h} \tri{u - \chi}_{1,D_d} + C d^{-1} \|u - u_h\|_{L^2(D_d)}  \\
& \leq C \inf_{\chi \in V_h} \|\nabla (u - \chi)\|_{L^2(D_d)} + C d^{-1} \|u - u_h\|_{L^2(D_d)}. 
\end{align*}
\end{theorem}
\begin{proof}
The second inequality follows again from a Poincar\'e inequality (the first one, this time). 
The first inequality follows again from tracing the arguments of the proof of 
Theorem~\ref{thm:local-quasi-optimality-ball}.  
\qed 
\end{proof}

%% file: detailed_numerics.tex
\newcommand{\tw}{0.45\textwidth}
\newcommand{\tws}{0.31\textwidth}
\newcommand{\twt}{0.25\textwidth}
\section{Detailed numerics}
\label{sec:detailed-numerics}
\subsection{Slit domain}
\label{sec:slit-domain} 
The geometry is a slit domain 
$$
\Omega = \Omega^S:= (-1,1)^2 \setminus [0,1) \times \{0\}. 
$$
The exact solution is given by 
$$
|x - x_0|^\alpha \sin(a \pi\phi)
$$
for different choices of the parameters $\alpha$ and $x_0$ (and $a$). 

The inhomogeneous Dirichlet boundary conditions are 
realized by nodal interpolation. The equation considered is
$$
-\Delta u = f. 
$$ 
Starting from a coarse mesh, we perform a sequence of uniform
(red) refinements. We consider a lowest order discretization, i.e.,  $k = 1$. 

Strictly speaking, the slit domain is not covered by our theory. Also not 
covered by our theory are the variational crimes associated with approximating
the inhomogeneous Dirichlet data. 
Nevertheless, we expect the convergence behavior detailed in Corollary~\ref{cor:2D-L2-concrete} 
to be an good description of the actual convergence behavior. 
We assume that the global regularity of the solution $u$ is described by 
$s = 1 + \alpha$ (actually, it is $1 + \alpha - \varepsilon$ for all $\varepsilon > 0$). 
Corollary~\ref{cor:2D-L2-concrete} then lets us expect for the two cases $x_0 = (0,0)$ and 
$x_0 \ne (0,0)$  the following convergence rates: 
\begin{eqnarray*}
\mbox{$x_0 = (0,0)$} &\Longrightarrow&  
\tau  = \min\{2,1+\alpha,-1+1/2+(1+\alpha)\} = \min\{2,1/2+\alpha\} \\
\mbox{$x_0 \ne (0,0)$} &\Longrightarrow&  
\tau  = \min\{2,1+\alpha,-1+1/2+\infty\} = \min\{2,1+\alpha\} 
\end{eqnarray*}

In the following Tables~\ref{tab:App_slit-shape1}--\ref{tab:App_slit-shape7}, we vary 
the parameter $\alpha$. In each table separately we vary the location. The locations under investigation are $(x_0,y_0)=(0,0)$, $(x_0,y_0)=(0.5,0)$ and $(x_0,y_0)=(0,1)$. We observe that the theoretical convergence rates are mostly achieved in our numerical simulations.

\begin{table}[ht]
\begin{center}\begin{tabular}{|r||c|c||c|c||c|c|}
\hline
&\multicolumn{2}{|c||}{$(x_0,y_0) = (0,0)$}
&\multicolumn{2}{|c||}{$(x_0,y_0) = (0.5,0)$}
&\multicolumn{2}{|c|}{$(x_0,y_0) = (0,1)$}  \\
&\multicolumn{2}{|c||}{$a = \pi/2$}
&\multicolumn{2}{|c||}{$a=\pi$}
&\multicolumn{2}{|c|}{$a=\pi$}  \\
\hline
DOFs   &$L^2$-error  & rate &$L^2$-error  & rate & $L^2$-error & rate  \\ \hline
97     &2.9124e-02  &  -         &    3.8405e-02&     -        &      3.0468e-02 & -        \\
348    &1.5745e-02  & 0.89    &    1.0451e-02&  1.88     &      1.2883e-02& 1.24    \\
1.315   &8.1422e-03  & 0.95    &    4.8926e-03&  1.10     &      5.2831e-03& 1.29    \\
5.109   &4.1322e-03  & 0.98    &    2.1508e-03&  1.19     &      2.0814e-03& 1.34    \\
20.137  &2.0799e-03  & 0.99    &    8.1046e-04&  1.41     &      7.9896e-04& 1.38    \\
79.953  &1.0430e-03  & 1.00    &    3.0969e-04&  1.39     &      3.0187e-04& 1.40    \\
318.625 &5.2221e-04  & 1.00    &    1.1780e-04&  1.39     &      1.1288e-04& 1.42    \\
1.272.129&2.6125e-04  & 1.00    &    4.1750e-05&  1.50     &      4.1903e-05& 1.43    \\
5.083.777&1.3066e-04  & 1.00    &    1.4985e-05&  1.48     &      1.5472e-05& 1.44     \\
\hline
\end{tabular}
   \caption{Slit domain, $k=1$: Influence of the position of singularity for $\alpha = 0.5$.}
\label{tab:App_slit-shape1}
\end{center}
\end{table}
\begin{table}[ht]
\begin{center}\begin{tabular}{|r||c|c||c|c||c|c|}
\hline
&\multicolumn{2}{|c||}{$(x_0,y_0) = (0,0)$}
&\multicolumn{2}{|c||}{$(x_0,y_0) = (0.5,0)$}
&\multicolumn{2}{|c|}{$(x_0,y_0) = (0,1)$}  \\
&\multicolumn{2}{|c||}{$a = \pi/2$}
&\multicolumn{2}{|c||}{$a=\pi$}
&\multicolumn{2}{|c|}{$a=\pi$}  \\
\hline
DOFs   &$L^2$-error  & rate &$L^2$-error  & rate & $L^2$-error & rate  \\ \hline
97     &1.1642e-02  &  -           &1.6577e-02    & -          &1.6031e-02       & -   \\
348    & 5.6371e-03  &  1.05     &   6.5153e-03 & 1.35    &      6.0796e-03 & 1.40   \\
1.315   & 2.6059e-03  &  1.11     &   2.3569e-03 & 1.47    &      2.2272e-03 & 1.45   \\
5.109   & 1.1835e-03  &  1.14     &   8.4203e-04 & 1.48    &      7.8285e-04 & 1.51   \\
20.137  & 5.3296e-04  &  1.15     &   2.7392e-04 & 1.62    &      2.6776e-04 & 1.55   \\
79.953  & 2.3888e-04  &  1.16     &   9.1902e-05 & 1.58    &      9.0042e-05 & 1.57   \\
318.625 & 1.0679e-04  &  1.16     &   3.0823e-05 & 1.58    &      2.9942e-05 & 1.59   \\
1.272.129& 4.7666e-05  &  1.16     &   9.6773e-06 & 1.67    &      9.8788e-06 & 1.60   \\
5.083.777& 2.1257e-05   & 1.16     &   3.1032e-06 & 1.64    &      3.2407e-06 & 1.61   \\
  \hline
\end{tabular}
   \caption{Slit domain, $k=1$: Influence of the position of singularity for $\alpha = 2/3$.}
\label{tab:App_slit-shape2}
\end{center}
\end{table}
\begin{table}[ht]
\begin{center}\begin{tabular}{|r||c|c||c|c||c|c|}
\hline
&\multicolumn{2}{|c||}{$(x_0,y_0) = (0,0)$}
&\multicolumn{2}{|c||}{$(x_0,y_0) = (0.5,0)$}
&\multicolumn{2}{|c|}{$(x_0,y_0) = (0,1)$}  \\
&\multicolumn{2}{|c||}{$a = \pi/2$}
&\multicolumn{2}{|c||}{$a=\pi$}
&\multicolumn{2}{|c|}{$a=\pi$}  \\
\hline
DOFs   &$L^2$-error  & rate &$L^2$-error  & rate & $L^2$-error & rate  \\ \hline
97     &6.1391e-03  &  -         &    1.1088e-02& -          & 1.0692e-02       & -        \\
348    &2.8187e-03  &  1.12   &    4.1329e-03&  1.42   &      3.8553e-03  & 1.47   \\
1.315   &1.2351e-03  &  1.19   &    1.4164e-03&  1.54   &      1.3388e-03  & 1.53   \\
5.109   &5.3338e-04  &  1.21   &    4.7830e-04&  1.57   &      4.4562e-04  & 1.59   \\
20.137  &2.2846e-04  &  1.22   &    1.4725e-04&  1.70   &      1.4420e-04  & 1.63   \\
79.953  &9.7267e-05  &  1.23   &    4.6683e-05&  1.66   &      4.5843e-05  & 1.65   \\
318.625 &4.1233e-05  &  1.24   &    1.4761e-05&  1.66   &      1.4401e-05  & 1.67   \\
1.272.129&1.7428e-05  &  1.24   &    4.3773e-06&  1.75   &      4.4861e-06  & 1.68   \\
5.083.777&7.3524e-06  &  1.25   &    1.3285e-06&  1.72   &      1.3889e-06  & 1.69   \\
  \hline
\end{tabular}
   \caption{Slit domain, $k=1$: Influence of the position of singularity for $\alpha = 0.75$.}
\label{tab:App_slit-shape3}
\end{center}
\end{table}
\begin{table}[ht]
\begin{center}\begin{tabular}{|r||c|c||c|c||c|c|}
\hline
&\multicolumn{2}{|c||}{$(x_0,y_0) = (0,0)$}
&\multicolumn{2}{|c||}{$(x_0,y_0) = (0.5,0)$}
&\multicolumn{2}{|c|}{$(x_0,y_0) = (0,1)$}  \\
&\multicolumn{2}{|c||}{$a = \pi/2$}
&\multicolumn{2}{|c||}{$a=\pi$}
&\multicolumn{2}{|c|}{$a=\pi$}  \\
\hline
DOFs   &$L^2$-error  & rate &$L^2$-error  & rate & $L^2$-error & rate  \\ \hline
97     &4.1949e-03  &  -         & 3.0111e-04   & -          & 2.9784e-04      & -   \\
348    &1.4605e-03  &  1.52   &    9.9257e-05& 1.60    &       9.3618e-05&1.67    \\
1.315   &4.8756e-04  &  1.58   &    2.9679e-05& 1.74    &       2.8033e-05&1.74    \\
5.109   &1.5909e-04  &  1.62   &    8.6201e-06& 1.78    &       8.0205e-06&1.81    \\
20.137  &5.1667e-05  &  1.62   &    2.2994e-06& 1.91    &       2.2266e-06&1.85    \\
79.953  &1.6874e-05  &  1.61   &    6.2533e-07& 1.88    &       6.0606e-07&1.88    \\
318.625 &5.5687e-06  &  1.60   &    1.6832e-07& 1.89    &       1.6270e-07&1.90    \\
1.272.129&1.8596e-06  &  1.58   &    4.3035e-08& 1.97    &       4.3240e-08&1.91    \\
5.083.777&6.2798e-07  &1.57     &    1.1235e-08& 1.94    &       1.1403e-08&1.92    \\
  \hline
\end{tabular}
   \caption{Slit domain, $k=1$: Influence of the position of singularity for $\alpha = 1.01$.}
\label{tab:App_slit-shape4}
\end{center}
\end{table}
\begin{table}[ht]
\begin{center}\begin{tabular}{|r||c|c||c|c||c|c|}
\hline
&\multicolumn{2}{|c||}{$(x_0,y_0) = (0,0)$}
&\multicolumn{2}{|c||}{$(x_0,y_0) = (0.5,0)$}
&\multicolumn{2}{|c|}{$(x_0,y_0) = (0,1)$}  \\
&\multicolumn{2}{|c||}{$a = \pi/2$}
&\multicolumn{2}{|c||}{$a=\pi$}
&\multicolumn{2}{|c|}{$a=\pi$}  \\
\hline
DOFs   &$L^2$-error  & rate &$L^2$-error  & rate & $L^2$-error & rate  \\ \hline
97     		&5.7534e-03  &  -         & 2.8888e-03    & -          &  2.8799e-03       & -   \\
348    		&1.9412e-03  &  1.57   &   9.1606e-04  & 1.66    &      8.6618e-04   &1.73    \\
1.315   		&6.2583e-04  &  1.63   &   2.6267e-04  & 1.80    &      2.4731e-04   &1.81    \\
5.109   		&1.9689e-04  &  1.67   &   7.2833e-05  & 1.85    &      6.7463e-05   &1.87    \\
20.137  	&6.1446e-05  &  1.68   &   1.8673e-05  & 1.96    &      1.7871e-05   &1.92    \\
79.953  	&1.9191e-05  &  1.68   &   4.8621e-06  & 1.94    &      4.6455e-06   &1.94    \\
318.625 	&6.0229e-06  &  1.67   &   1.2518e-06  & 1.96    &      1.1921e-06   &1.96    \\
1.272.129	&1.9023e-06  &  1.66   &   3.0913e-07  & 2.02    &      3.0311e-07   &1.98    \\
5.083.777	&6.0474e-07  &  1.65   &   7.7720e-08  & 1.99    &      7.6552e-08   &1.99    \\
  \hline
\end{tabular}
   \caption{Slit domain, $k=1$: Influence of the position of singularity for $\alpha = 10/9$.}
\label{tab:App_slit-shape5}
\end{center}
\end{table}


\begin{table}[ht]
\begin{center}\begin{tabular}{|r||c|c||c|c||c|c|}
\hline
&\multicolumn{2}{|c||}{$(x_0,y_0) = (0,0)$}
&\multicolumn{2}{|c||}{$(x_0,y_0) = (0.5,0)$}
&\multicolumn{2}{|c|}{$(x_0,y_0) = (0,1)$}  \\
&\multicolumn{2}{|c||}{$a = \pi/2$}
&\multicolumn{2}{|c||}{$a=\pi$}
&\multicolumn{2}{|c|}{$a=\pi$}  \\
\hline
DOFs   &$L^2$-error  & rate &$L^2$-error  & rate & $L^2$-error & rate  \\ \hline
97     &7.3549e-03  &  -          & 6.3401e-03   & -           &6.3849e-03       & -   \\
348    &2.2414e-03  &  1.71    &    1.8792e-03&  1.75    &      1.7790e-03 & 1.84   \\
1.315   &6.4849e-04  &  1.79    &    5.0365e-04&  1.90    &      4.6905e-04 & 1.92   \\
5.109   &1.8251e-04  &  1.83    &    1.3007e-04&  1.95    &      1.1878e-04 & 1.98   \\
20.137  &5.0718e-05  &  1.85    &    3.1798e-05&  2.03    &      2.9443e-05 & 2.01   \\
79.953  &1.4021e-05  &  1.85    &    7.8665e-06&  2.02    &      7.2227e-06 & 2.03   \\
318.625 &3.8699e-06  &  1.86    &    1.9356e-06&  2.02    &      1.7642e-06 & 2.03   \\
1.272.129&1.0682e-06  &  1.86    &    4.6924e-07&  2.04    &      4.3055e-07 & 2.03   \\
5.083.777&2.9514e-07  &   1.86   &    1.1524e-07&  2.03    &      1.0519e-07 & 2.03   \\
  \hline
\end{tabular}
   \caption{Slit domain, $k=1$: Influence of the position of singularity for $\alpha = 4/3$.}
\label{tab:App_slit-shape6}
\end{center}
\end{table}


\begin{table}[ht]
\begin{center}\begin{tabular}{|r||c|c||c|c||c|c|}
\hline
&\multicolumn{2}{|c||}{$(x_0,y_0) = (0,0)$}
&\multicolumn{2}{|c||}{$(x_0,y_0) = (0.5,0)$}
&\multicolumn{2}{|c|}{$(x_0,y_0) = (0,1)$}  \\
&\multicolumn{2}{|c||}{$a = \pi/2$}
&\multicolumn{2}{|c||}{$a=\pi$}
&\multicolumn{2}{|c|}{$a=\pi$}  \\
\hline
DOFs   &$L^2$-error  & rate &$L^2$-error  & rate & $L^2$-error & rate  \\ \hline
97     &7.5901e-03  &  -         & 7.6006e-03   & -          &7.6553e-03       & -   \\
348    &2.1664e-03  & 1.81    &    2.1751e-03&  1.81   &      2.0530e-03 & 1.90   \\
1.315   &5.8638e-04  & 1.89    &    5.6614e-04&  1.94   &      5.2238e-04 & 1.97   \\
5.109   &1.5450e-04  & 1.92    &    1.4246e-04&  1.99   &      1.2877e-04 & 2.02   \\
20.137  &4.0197e-05  & 1.94    &    3.4615e-05&  2.04   &      3.1388e-05 & 2.04   \\
79.953  &1.0396e-05  & 1.95    &    8.5086e-06&  2.02   &      7.6403e-06 & 2.04   \\
318.625 &2.6803e-06  & 1.96    &    2.0921e-06&  2.02   &      1.8651e-06 & 2.03   \\
1.272.129&6.8978e-07  & 1.96    &    5.1307e-07&  2.03   &      4.5726e-07 & 2.03   \\
5.083.777&1.7730e-07  & 1.96    &    1.2691e-07&  2.02   &      1.1258e-07 & 2.02   \\
  \hline
\end{tabular}
   \caption{Slit domain, $k=1$: Influence of the position of singularity for $\alpha = 1.5$.}
\label{tab:App_slit-shape7}
\end{center}
\end{table}

\clearpage
\subsection{L-shaped domain}
The geometry is an L-shaped domain: 
$$
\Omega = \Omega^L:= (-1,1)^2 \setminus [0,1) \times (-1,0]. 
$$
The exact solution is given by 
$$
|x - x_0|^\alpha \sin(a\pi\phi) 
$$
for different choices of the parameters $\alpha$, $x_0$,  and $a$. 

\subsubsection{Lowest order discretization $k=1$}
Structurally, the situation is similar to the situation in 
Section~\ref{sec:slit-domain}. From Corollary~\ref{cor:2D-L2-concrete}, we expect 
the following convergence rates: 
\begin{eqnarray*}
\mbox{$x_0 = (0,0)$} &\Longrightarrow&  
\tau  = \min\{2,1+\alpha,-1+2/3+(1+\alpha)\} = \min\{2,2/3+\alpha\} \\
\mbox{$x_0 \ne (0,0)$} &\Longrightarrow&  
\tau  = \min\{2,1+\alpha,-1+2/3+\infty\} = \min\{2,1+\alpha\} 
\end{eqnarray*}
Also in this case the numerical rates depicted in Tables~\ref{tab:App_L-shape1}--\ref{tab:App_L-shape6} are very close to the rates expected by our theory.

\begin{table}[ht]
\begin{center}\begin{tabular}{|r||c|c||c|c||c|c|}
\hline
&\multicolumn{2}{|c||}{$(x_0,y_0) = (0,0)$}
&\multicolumn{2}{|c||}{$(x_0,y_0) = (0.5,0)$}
&\multicolumn{2}{|c|}{$(x_0,y_0) = (0,1)$}  \\
&\multicolumn{2}{|c||}{$a = \pi/2$}
&\multicolumn{2}{|c||}{$a=\pi$}
&\multicolumn{2}{|c|}{$a=\pi$}  \\
\hline
DOFs   &$L^2$-error  & rate &$L^2$-error  & rate & $L^2$-error & rate  \\ \hline
81     &1.1719e-02  &  -         &  1.0132e-02   & -          &9.5383e-03       & -   \\
289    &5.4059e-03  & 1.12    &    4.0619e-03 & 1.32    &      3.4014e-03 &1.49    \\
1.089   &2.3165e-03  & 1.22    &    1.4839e-03 & 1.45    &      1.1426e-03 &1.57    \\
4.225   &9.5790e-04  & 1.27    &    5.1844e-04 & 1.52    &      3.7569e-04 &1.60    \\
16.641  &3.8922e-04  & 1.30    &    1.7681e-04 & 1.55    &      1.2222e-04 &1.62    \\
66.049  &1.5663e-04  & 1.31    &    5.9408e-05 & 1.57    &      3.9513e-05 &1.63    \\
263.169 &6.2682e-05  & 1.32    &    1.9762e-05 & 1.59    &      1.2720e-05 &1.64    \\
1.050.625&2.5002e-05  & 1.33    &    6.5263e-06 & 1.60    &      4.0824e-06 &1.64    \\
4.198.401&9.9525e-06  & 1.33    &    2.1437e-06 & 1.61    &      1.3072e-06 &1.64    \\
  \hline
\end{tabular}
   \caption{L--domain, $k=1$: Influence of the position of singularity for $\alpha = 2/3$.}
\label{tab:App_L-shape1}
\end{center}
\end{table}

\begin{table}[ht]
\begin{center}\begin{tabular}{|r||c|c||c|c||c|c|}
\hline
&\multicolumn{2}{|c||}{$(x_0,y_0) = (0,0)$}
&\multicolumn{2}{|c||}{$(x_0,y_0) = (0.5,0)$}
&\multicolumn{2}{|c|}{$(x_0,y_0) = (0,1)$}  \\
&\multicolumn{2}{|c||}{$a = \pi/2$}
&\multicolumn{2}{|c||}{$a=\pi$}
&\multicolumn{2}{|c|}{$a=\pi$}  \\
\hline
DOFs   &$L^2$-error  & rate &$L^2$-error  & rate & $L^2$-error & rate  \\ \hline
81     &6.1585e-03  &  -         &6.8141e-03    & -          &6.2506e-03       & -   \\
289    &2.6986e-03  &1.19     &   2.5648e-03 &1.41     &      2.1211e-03 &1.56    \\
1.089   &1.1123e-03  &1.28     &   8.8428e-04 &1.54     &      6.7413e-04 &1.65    \\
4.225   &4.4037e-04  &1.34     &   2.9202e-04 &1.60     &      2.0903e-04 &1.69    \\
16.641  &1.7107e-04  &1.36     &   9.4164e-05 &1.63     &      6.4027e-05 &1.71    \\
66.049  &6.5689e-05  &1.38     &   2.9909e-05 &1.65     &      1.9471e-05 &1.72    \\
263.169 &2.5030e-05  &1.39     &   9.4012e-06 &1.67     &      5.8930e-06 &1.72    \\
1.050.625&9.4877e-06  &1.40     &   2.9328e-06 &1.68     &      1.7774e-06 &1.73    \\
4.198.401&3.5834e-06  &1.40     &   9.0968e-07 &1.69     &      5.3475e-07 &1.73    \\
  \hline
\end{tabular}
   \caption{L--domain, $k=1$: Influence of the position of singularity for $\alpha = 0.75$.}
\label{tab:App_L-shape2}
\end{center}
\end{table}

\begin{table}[ht]
\begin{center}\begin{tabular}{|r||c|c||c|c||c|c|}
\hline
&\multicolumn{2}{|c||}{$(x_0,y_0) = (0,0)$}
&\multicolumn{2}{|c||}{$(x_0,y_0) = (0.5,0)$}
&\multicolumn{2}{|c|}{$(x_0,y_0) = (0,1)$}  \\
&\multicolumn{2}{|c||}{$a = \pi/2$}
&\multicolumn{2}{|c||}{$a=\pi$}
&\multicolumn{2}{|c|}{$a=\pi$}  \\
\hline
DOFs   &$L^2$-error  & rate &$L^2$-error  & rate & $L^2$-error & rate  \\ \hline
81     &4.6216e-03  &  -         &1.8387e-04    & -          &1.6841e-04       & -   \\
289    &1.6860e-03  & 1.45    &   6.0370e-05 &1.61     &      5.0364e-05 &1.74    \\
1.089   &5.4867e-04  & 1.62    &   1.8034e-05 &1.74     &      1.3883e-05 &1.86    \\
4.225   &1.7284e-04  & 1.67    &   5.1378e-06 &1.81     &      3.6942e-06 &1.91    \\
16.641  &5.2963e-05  & 1.71    &   1.4253e-06 &1.85     &      9.6399e-07 &1.94    \\
66.049  &1.5970e-05  & 1.73    &   3.8870e-07 &1.87     &      2.4842e-07 &1.96    \\
263.169 &4.7758e-06  & 1.74    &   1.0474e-07 &1.89     &      6.3437e-08 &1.97    \\
1.050.625&1.4238e-06  & 1.75    &   2.7971e-08 &1.90     &      1.6086e-08 &1.98    \\
4.198.401&4.2471e-07  & 1.75    &   7.4181e-09 &1.91     &      4.0561e-09 &1.99    \\
  \hline
\end{tabular}
   \caption{L--domain, $k=1$: Influence of the position of singularity for $\alpha = 1.01$.}
\label{tab:App_L-shape3}
\end{center}
\end{table}

\begin{table}[ht]
\begin{center}\begin{tabular}{|r||c|c||c|c||c|c|}
\hline
&\multicolumn{2}{|c||}{$(x_0,y_0) = (0,0)$}
&\multicolumn{2}{|c||}{$(x_0,y_0) = (0.5,0)$}
&\multicolumn{2}{|c|}{$(x_0,y_0) = (0,1)$}  \\
&\multicolumn{2}{|c||}{$a = \pi/2$}
&\multicolumn{2}{|c||}{$a=\pi$}
&\multicolumn{2}{|c|}{$a=\pi$}  \\
\hline
DOFs   &$L^2$-error  & rate &$L^2$-error  & rate & $L^2$-error & rate  \\ \hline
81     &6.5660e-03  &  -         &  1.7641e-03   & -          &1.6229e-03       & -   \\
289    &2.3309e-03  &1.49     &    5.5465e-04 &1.67     &      4.6837e-04 &1.79    \\
1.089   &7.3413e-04  &1.67     &    1.5847e-04 &1.81     &      1.2424e-04 &1.91    \\
4.225   &2.2257e-04  &1.72     &    4.3172e-05 &1.88     &      3.1761e-05 &1.97    \\
16.641  &6.5650e-05  &1.76     &    1.1458e-05 &1.91     &      7.9588e-06 &2.00    \\
66.049  &1.9056e-05  &1.78     &    2.9916e-06 &1.94     &      1.9695e-06 &2.01    \\
263.169 &5.4810e-06  &1.80     &    7.7249e-07 &1.95     &      4.8311e-07 &2.03    \\
1.050.625&1.5690e-06  &1.80     &    1.9789e-07 &1.96     &      1.1771e-07 &2.04    \\
4.198.401&4.4822e-07  &1.81     &    5.0393e-08 &1.97     &      2.8532e-08 &2.04    \\
  \hline
\end{tabular}
   \caption{L--domain, $k=1$: Influence of the position of singularity for $\alpha = 10/9$.}
\label{tab:App_L-shape4}
\end{center}
\end{table}

\begin{table}[ht]
\begin{center}\begin{tabular}{|r||c|c||c|c||c|c|}
\hline
&\multicolumn{2}{|c||}{$(x_0,y_0) = (0,0)$}
&\multicolumn{2}{|c||}{$(x_0,y_0) = (0.5,0)$}
&\multicolumn{2}{|c|}{$(x_0,y_0) = (0,1)$}  \\
&\multicolumn{2}{|c||}{$a = \pi/2$}
&\multicolumn{2}{|c||}{$a=\pi$}
&\multicolumn{2}{|c|}{$a=\pi$}  \\
\hline
DOFs   &$L^2$-error  & rate &$L^2$-error  & rate & $L^2$-error & rate  \\ \hline
81     &8.6776e-03  &  -         &3.8962e-03    & -          &3.6446e-03       & -   \\
289    &2.8523e-03  & 1.61    &   1.1374e-03 & 1.78    &      1.0008e-03&1.86    \\
1.089   &8.2870e-04  & 1.78    &   3.0272e-04 & 1.91    &      2.5331e-04&1.98    \\
4.225   &2.3073e-04  & 1.84    &   7.7239e-05 & 1.97    &      6.2153e-05&2.03    \\
16.641  &6.2539e-05  & 1.88    &   1.9331e-05 & 2.00    &      1.5073e-05&2.04    \\
66.049  &1.6688e-05  & 1.91    &   4.7956e-06 & 2.01    &      3.6440e-06&2.05    \\
263.169 &4.4099e-06  & 1.92    &   1.1852e-06 & 2.02    &      8.8167e-07&2.05    \\
1.050.625&1.1580e-06  & 1.93    &   2.9260e-07 & 2.02    &      2.1389e-07&2.04    \\
4.198.401&3.0279e-07  & 1.94    &   7.2263e-08 & 2.02    &      5.2069e-08&2.04    \\
  \hline
\end{tabular}
   \caption{L--domain, $k=1$: Influence of the position of singularity for $\alpha = 4/3$.}
\label{tab:App_L-shape5}
\end{center}
\end{table}

\begin{table}[ht]
\begin{center}\begin{tabular}{|r||c|c||c|c||c|c|}
\hline
&\multicolumn{2}{|c||}{$(x_0,y_0) = (0,0)$}
&\multicolumn{2}{|c||}{$(x_0,y_0) = (0.5,0)$}
&\multicolumn{2}{|c|}{$(x_0,y_0) = (0,1)$}  \\
&\multicolumn{2}{|c||}{$a = \pi/2$}
&\multicolumn{2}{|c||}{$a=\pi$}
&\multicolumn{2}{|c|}{$a=\pi$}  \\
\hline
DOFs   &$L^2$-error  & rate &$L^2$-error  & rate & $L^2$-error & rate  \\ \hline
81     &8.9932e-03  &  -         &4.7178e-03    & -          &4.4942e-03       & -   \\
289    &2.8151e-03  &1.68     &   1.3287e-03 & 1.83    &      1.2166e-03 &1.89    \\
1.089   &7.8034e-04  &1.85     &   3.4367e-04 & 1.95    &      3.0580e-04 &1.99    \\
4.225   &2.0751e-04  &1.91     &   8.5903e-05 & 2.00    &      7.5035e-05 &2.03    \\
16.641  &5.3910e-05  &1.94     &   2.1227e-05 & 2.02    &      1.8321e-05 &2.03    \\
66.049  &1.3835e-05  &1.96     &   5.2331e-06 & 2.02    &      4.4827e-06 &2.03    \\
263.169 &3.5256e-06  &1.97     &   1.2917e-06 & 2.02    &      1.1011e-06 &2.03    \\
1.050.625&8.9467e-07  &1.98     &   3.1955e-07 & 2.02    &      2.7158e-07 &2.02    \\
4.198.401&2.2641e-07  &1.98     &   7.9238e-08 & 2.01    &      6.7212e-08 &2.01    \\
  \hline
\end{tabular}
   \caption{L--domain, $k=1$: Influence of the position of singularity for $\alpha = 1.5$.}
\label{tab:App_L-shape6}
\end{center}
\end{table}

\clearpage
\subsubsection{Second order discretization $k=2$}
All calculations are performed for $x_0 = (0,0)$ and $a = 2/3$. 
Only the singularity parameter $\alpha$ is varied. 

Here, we expect the convergence rate 
$$
\tau = \min\{3,-1+2/3+(1+\alpha)\} = \min\{3,2/3+\alpha\}.
$$

Table~\ref{tab:App_L_2ndOrder3} shows the numerical results for the second order case in which the received rates are close to the theoretical expected once. 

\begin{table}[ht]
\begin{center}\begin{tabular}{|r||c|c||c|c||c|c||c|c|}
\hline
&\multicolumn{2}{|c||}{$\alpha = 2/3$}
&\multicolumn{2}{|c||}{$\alpha=3/4$}
&\multicolumn{2}{|c||}{$\alpha=1.01$}
&\multicolumn{2}{|c|}{$\alpha=10/9$}  \\
\hline
DOFs   &$L^2$-error  & rate &$L^2$-error  & rate & $L^2$-error & rate & $L^2$-error & rate  \\ \hline
289      &3.1686e-03  &  -         &1.6898e-03    & -            &4.8115e-04        & -         &5.9011e-04           & -   \\
1.089     &1.2099e-03  &1.39     &   6.0844e-04 &1.47       &      1.4003e-04  &1.78    &         1.5596e-04  &1.92      \\
4.225     &4.6505e-04  &1.38     &   2.1881e-04 &1.48       &      3.7277e-05  &1.91    &         3.8312e-05  &2.03      \\
16.641    &1.8057e-04  &1.36     &   8.0073e-05 &1.45       &      9.9546e-06  &1.90    &         9.3965e-06  &2.03      \\
66.049    &7.0635e-05  &1.35     &   2.9545e-05 &1.44       &      2.6951e-06  &1.89    &         2.3314e-06  &2.01      \\
263.169   &2.7771e-05  &1.35     &   1.0960e-05 &1.43       &      7.4481e-07  &1.86    &         5.8950e-07  &1.98      \\
1.050.625&1.0955e-05  &1.34     &   4.0799e-06 &1.43       &      2.1075e-07  &1.82    &         1.5257e-07  &1.95      \\
  \hline
\end{tabular}
\end{center}
\begin{center}\begin{tabular}{|r||c|c||c|c||c|c||c|c|}
\hline
&\multicolumn{2}{|c||}{$\alpha = 4/3$}
&\multicolumn{2}{|c||}{$\alpha=3/2$}
&\multicolumn{2}{|c||}{$\alpha=2.175$}
&\multicolumn{2}{|c|}{$\alpha=2.275$}  \\
\hline
DOFs   &$L^2$-error  & rate &$L^2$-error  & rate & $L^2$-error & rate & $L^2$-error & rate  \\ \hline
289      &6.1433e-04  &  -         &   5.5363e-04 & -            &      2.7565e-04  & -         &		  2.4570e-04  & -   \\
1.089     &1.5136e-04  &2.02     &   1.3540e-04 &2.03       &      5.1121e-05  &2.43    &         4.1696e-05  &2.56      \\
4.225     &3.3604e-05  &2.17     &   2.8521e-05 &2.25       &      7.5320e-06  &2.76    &         5.7319e-06  &2.86      \\
16.641    &7.7002e-06  &2.13     &   6.2123e-06 &2.20       &      1.1051e-06  &2.77    &         7.8407e-07  &2.87      \\
66.049    &1.8014e-06  &2.10     &   1.3642e-06 &2.19       &      1.5938e-07  &2.79    &         1.0553e-07  &2.89      \\
263.169   &4.2916e-07  &2.07     &   3.0106e-07 &2.18       &      2.2723e-08  &2.81    &         1.4044e-08  &2.91      \\
1.050.625&1.0374e-07  &2.05     &   6.6649e-08 &2.18       &      3.2138e-09  &2.82    &         1.8538e-09  &2.92      \\
  \hline
\end{tabular}
\end{center}
\begin{center}\begin{tabular}{|r||c|c|}
\hline
&\multicolumn{2}{|c|}{$\alpha=2.375$}  \\
\hline
DOFs   &$L^2$-error  & rate   \\ \hline
289      &2.2177e-04    & -   \\
1.089     &3.3912e-05    &  2.71     \\
4.225     &4.3221e-06    &  2.97     \\
16.641    &5.4888e-07    &  2.98     \\
66.049    &6.8762e-08    &  3.00     \\
263.169   &8.5292e-09    &  3.01     \\
1.050.625&1.0497e-09    &  3.02     \\
  \hline
\end{tabular}
   \caption{L-shaped domain, $k=2$: Influence of  $\alpha $ for
     $a=2/3 \pi$ and $(x_0,y_0) =(0,0)$.}
\label{tab:App_L_2ndOrder3}
\end{center}
\end{table}

\clearpage
\subsection{Fichera corner}
\subsubsection{Smooth solution}
The geometry is 
$$
\Omega = \Omega^F:= (-1,1)^3 \setminus [0,1]^3. 
$$
The discretization is done on lowest order hexahedral elements, regularly refined. 
The exact solution is prescribed to be the smooth solution 
$$
u(x,y,z) = \sin ((x+y)\pi) \cos (2 \pi z). 
$$
\begin{table}[h!]
\begin{center}
\begin{tabular}{|r|c|c|}
\hline 
DOFs & L2 error & rate \\ \hline 
316  &   0.075444   &  --- \\
3.032 &    0.017182   &  1.96 \\
26.416&     0.0039376   &  2.04 \\
220.256 &    0.00094597   &  2.02 \\
1.798.336 &     0.00023208 &    2.01 \\
14.532.992 &   5.7491e-05  &   2.00  \\ \hline 
\end{tabular} 
\caption{Fichera cube.}

\end{center}
\end{table}
\subsubsection{Solution of point singularity type}
In the next calculations, the exact solution is given by 
$$
u = r^\alpha, 
$$ 
where $r = \operatorname*{dist}(x,x_0)$ measures the distance from 
the point $x_0$, which is varied. The $L^2$-error is computed 
with a tensor product Gauss rule (5 points in each coordinate direction).

\begin{table}
\begin{center}
\begin{tabular}{|r||c|c||c|c||c|c|}\hline
&\multicolumn{2}{|c|}{$x_0  = (-1,-1,-1)$, $\alpha = 0.55$}
&\multicolumn{2}{|c|}{$x_0  = (0,0.5,0)$, $\alpha = 0.55$}
&\multicolumn{2}{|c|}{$x_0  = (0,0.5,0)$, $\alpha = 2/3$} \\ \hline 
DOFs   			&    L2-error   	&    rate	&    L2-error   		&    rate   &    L2-error   	&    rate			 	\\ \hline 
316    			& 0.00073994 &    --- 	&	 0.00069287   & 		 ---	  & 	0.00074102 &	---	\\
3.032    		& 0.00016401 &   2.00 	&	 0.00023565  	&  	1.43  &	0.00023 	   &	1.55	\\
26.416   		&  3.9176e-05 &   1.98 	&	 6.8242e-05 	&		1.72  &	6.2591e-05  &	1.80	\\
22.0256  		&   9.6835e-06&   1.98 	&	 1.9077e-05	&   	1.80  &	1.6589e-05   &	1.88	\\
1.798.336  &   2.4305e-06 &  1.97 	&   5.2412e-06  	&   	1.85  &	4.3418e-06  &	1.92	\\
14.532.992 &    6.1407e-07&  1.98 	&	  1.44225e-06 &   	1.87  &	1.1264e-06  &	1.94	\\ \hline 
\end{tabular}
\caption{Fichera cube}
\end{center}
\end{table}